\documentclass[12pt]{amsart}
\usepackage{amssymb, a4wide, mathdots, url, graphicx}
\usepackage[usenames]{color}
\usepackage{enumerate}
\usepackage{mathtools}
\usepackage{xcolor}
\usepackage[hyperfootnotes=false, colorlinks, linkcolor={blue}, citecolor={magenta}, filecolor={blue}, urlcolor={blue}, plainpages=false, pdfpagelabels]{hyperref}
\usepackage{amsmath,amscd}
\usepackage{tikz-cd}

\newcommand{\PGL}{\operatorname{PGL}}

\newcommand{\tr}{\operatorname{tr}}

\renewcommand{\Re}{\operatorname{Re}}

\newcommand{\A}{\mathbb{A}}

\newcommand{\C}{\mathbb{C}}

\newcommand{\Hom}{\operatorname{Hom}}
\newcommand{\Res}{\operatorname{Res}}

\newcommand{\ind}{\operatorname{ind}}

\newcommand{\wt}{\widetilde}

\newcommand{\Span}{\operatorname{Span}}

\newcommand{\JL}{\mathrm{JL}}
\newcommand{\BC}{\mathrm{BC}}

\newcommand{\TD}{\mathcal{TD}}

   \newcommand{\CE}{{\mathcal{E}}}
   
    \newcommand{\RG}{{\mathrm {G}}}

\newtheorem{theorem}{Theorem}[section]
\newtheorem{lemma}[theorem]{Lemma}

\newtheorem{proposition}[theorem]{Proposition}
\newtheorem{corollary}[theorem]{Corollary}
\newtheorem{conjecture}[theorem]{Conjecture}
\theoremstyle{remark}
\newtheorem{remark}[theorem]{Remark}

\makeatletter
\@namedef{subjclassname@2020}{%
  \textup{2020} Mathematics Subject Classification}
\makeatother

\newcommand{\pr}{\operatorname{pr}}

\DeclareMathOperator{\diag}{diag}

\DeclareMathOperator{\Ind}{Ind}
\DeclareMathOperator{\GL}{GL}

\DeclareMathOperator{\SO}{SO}

\DeclareMathOperator{\GSpin}{GSpin}
\DeclareMathOperator{\GPin}{GPin}

\DeclareMathOperator{\GSp}{GSp}
\DeclareMathOperator{\Orth}{O}

\newcommand{\bs}{\backslash}

\begin{document}

\title[Twisted automorphic descent to odd GSpin groups and applications]{Twisted automorphic descent to odd GSpin groups and applications}

\author{Pan Yan}
\address{Department of Mathematics, The University of Arizona, Tucson, AZ 85721, USA}
\email{panyan@arizona.edu}

\date{\today}


\begin{abstract}
We establish twisted automorphic descent from $(\GL_{2n}, \GSpin_2^{\delta})$, where $\GSpin_2^{\delta}$ is anisotropic, to the odd $\GSpin$ group $\GSpin^{\delta,\alpha}_{2n+1}$, which may be split or non-split, generalizing a construction of Jiang--Liu--Xu--Zhang \cite{JiangLiuXuZhang2016}. 
As an application, we give a new explicit construction of the global Jacquet--Langlands correspondence between $\GL_2$ and $D^\times$, where $D$ is a quaternion division algebra over a number field. As another application, we prove a new case of the global Gan--Gross--Prasad conjecture for $(\GSpin_{2n+1}, \GSpin^{\delta}_2)$. 
\end{abstract}

\maketitle

\goodbreak 

\tableofcontents

\goodbreak

\section{Introduction}

The automorphic descent method, developed by Ginzburg, Rallis, and Soudry in a series of works
(see, for example, \cite{GinzburgRallisSoudry1997IMRN, GinzburgRallisSoudry1999Annals, GinzburgRallisSoudry1999JAMS, GinzburgRallisSoudry1999Duke, GinzburgRallisSoudry2001IMRN}) and culminating in their book \cite{GinzburgRallisSoudry2011}, 
provides an explicit construction of the inverse to the Langlands functorial transfer
(see, for example, \cite{CogdellKimPSShahidi2001, CogdellKimPSShahidi2004, CogdellPSShahidi2011, CaiFriedbergKaplan2024})
for automorphic representations of quasi-split classical groups. 
In its original form, one starts with a generic isobaric sum automorphic representation of a general linear group satisfying suitable symmetry conditions, forms a residual Siegel Eisenstein series, and then takes appropriate Fourier coefficients of the Eisenstein series. 
This construction produces globally generic cuspidal automorphic representations of quasi-split classical groups. 
The analogous descent construction for quasi-split $\GSpin$ groups was established by Hundley and Sayag in \cite{HundleySayag2016}.

The automorphic descent method has played an important role in the study of automorphic representations.
It gives an explicit construction of automorphic representations of quasi-split groups belonging to global Arthur packets (see \cite{Arthur2013, Mok2015, KalethaMinguezShinWhite}) and, in particular, of their generic members. It is also closely related to Rankin--Selberg integrals; see, for example, \cite{GinzburgRallisSoudry1998, GinzburgRallisSoudry2011, AsgariCogdellShahidi2024}.
Moreover, it can be applied to obtain precise formulas for Whittaker--Fourier coefficients of globally generic cuspidal automorphic forms in the spirit of the Ichino--Ikeda conjecture \cite{IchinoIkeda2010}; see \cite{LapidMao2017}.
We refer the reader to \cite{Soudry2005, Soudry2006ICM} for surveys of the method.

A twisted version of automorphic descent was developed in the works \cite{JiangLiuXuZhang2016, JiangZhang2017, JiangLiuXu2020, JiangZhang2020Annals, LiuXu2023}, with earlier precursors in  \cite{GinzburgJiangRallis2004, GinzburgJiangRallis2005, GinzburgJiangRallis2009}.
In the twisted setting, the input data consist not only of an automorphic representation $\tau$ of a general linear group, but also of an auxiliary representation $\sigma$ of a smaller classical group.
The construction is then realized by taking suitable Bessel type or Fourier--Jacobi type Fourier coefficients of certain residual Eisenstein series attached to the pair $(\tau,\sigma)$; unlike in the original descent construction, these Eisenstein series are not of Siegel type.

Twisted automorphic descent provides a systematic method for explicitly constructing additional members, including non-generic ones, in global Arthur packets of classical groups.
It can also be used to construct automorphic representations on pure inner forms, and is therefore closely related to global Vogan packets.
Moreover, twisted automorphic descent is closely connected with more general Rankin--Selberg integrals (see \cite{JiangZhang2014, JiangZhang2020Annals}).
After unfolding these global integrals, Bessel and Fourier--Jacobi periods naturally arise, leading to a natural connection with the global Gan--Gross--Prasad conjecture (global GGP conjecture for short)  \cite{GanGrossPrasad2012}, which relates such periods to central values of $L$-functions.
Using this approach, various non-vanishing results for central $L$-values and several cases of the global GGP conjecture have been obtained; see \cite{GinzburgJiangRallis2004, GinzburgJiangRallis2005, GinzburgJiangRallis2009, JiangZhang2020Annals, JiangZhang2020JEMS, Yan2025GGP}. 
We refer the reader to \cite{JiangZhang2017, JiangZhang2021BesselDescent} for further discussion of twisted automorphic descent as well as its applications.

The purpose of this paper is to establish twisted automorphic descent from $(\GL_{2n}, \GSpin_2^\delta)$ to the odd $\GSpin$ group $\GSpin_{2n+1}^{\delta, \alpha}$. Here, $\GSpin_2^\delta$ is associated to an anisotropic two-dimensional quadratic space, and $\GSpin_{2n+1}^{\delta, \alpha}$ is associated to a quadratic space of dimension $2n+1$ which can be split or non-split over a number field $F$. This construction generalizes the twisted automorphic descent of Jiang, Liu, Xu and Zhang \cite{JiangLiuXuZhang2016} from special orthogonal groups to the corresponding $\GSpin$ setting.  We refer the reader to Section~\ref{subsection-Intro-twisted-descent} for the statement of the main result on twisted automorphic descent. As an application, by specializing to $n=1$, we obtain a new proof of the global Jacquet--Langlands correspondence for $\GL_2$; see Section~\ref{subsection-Intro-application-JacquetLanglands}. As another application, we prove a new case of the global GGP conjecture for $(\GSpin_{2n+1}, \GSpin_2^\delta)$; see Section~\ref{subsection-Intro-application-GGP}.  

\subsection{Twisted automorphic descent to odd $\GSpin$ groups}
\label{subsection-Intro-twisted-descent}
To state our results, we introduce some notation. Let $F$ be a number field and let $\A=\A_F$ denote its ring of adeles. Let $V$ be a quadratic space over $F$ of dimension $4n+2$ with Witt decomposition $V=V^+\oplus V_0\oplus V^-$, where 
$V^+=\Span\{e_1, \cdots, e_{2n}\}$, $V^-=\Span\{e_{-2n}, \cdots, e_{-1}\}$
 are maximal totally isotropic subspaces of dimension $2n$, in duality with each other, and $V_0$ is an anisotropic subspace of dimension 2. We assume the quadratic form on $V$ is associated with the matrix $\left( \begin{smallmatrix}
   &&w_{2n}\\
   &J_\delta&\\
   w_{2n}&&
  \end{smallmatrix}\right)$, where $w_{2n}=\left( \begin{smallmatrix} & & 1\\ &\iddots & \\ 1&& \end{smallmatrix}\right)$, $J_\delta=\left(\begin{smallmatrix} 1 &0\\ 0 &\delta\end{smallmatrix}\right)$, and $-\delta\not\in {F^\times}^2$. Let $P=M\ltimes U$ be the standard parabolic subgroup of $H^\delta=\GSpin_{4n+2}^\delta=\GSpin(V)$ stabilizing $V^+$. Its Levi subgroup $M$ is isomorphic to $\GL_{2n}\times \GSpin(V_0)$. Moreover, $\GSpin(V_0)=\GSpin^\delta_2=\Res_{E/F}\GL_1$, where $E=F(\sqrt{-\delta})$ is a quadratic extension of $F$.
Let $\tau=\tau_1\boxplus \tau_2\boxplus \cdots \boxplus \tau_r$ be an irreducible unitary generic isobaric sum automorphic representation of $\GL_{2n}(\A)$ associated to distinct cuspidal automorphic representations $\tau_1, \cdots, \tau_r$, such that $L(s, \tau_i, \wedge^2\otimes\omega^{-1})$ has a pole at $s=1$ for each $1\le i\le r$, where $\omega$ is a Hecke character.  
Let $\sigma$ be an automorphic character of $\GSpin^\delta_2(\A)=\Res_{E/F}\GL_1(\A)$ and denote its restriction to $\mathbb{A}^\times$ by
\begin{equation*}
\omega_{\sigma}:=\sigma|_{\A^\times}.	
\end{equation*}
Assume that $\omega_{\sigma}=\omega$ and 
 $L(\frac{1}{2}, \sigma\times \tau\otimes \omega_\sigma^{-1})\neq 0$. Here, 
 \begin{equation*}
 	L(s,\sigma\times\tau\otimes \omega_{\sigma}^{-1})= L(s, \sigma\times \BC(\tau\otimes \omega_{\sigma}^{-1})),
 \end{equation*}
 where $\BC(\tau\otimes \omega_{\sigma}^{-1})$ is the base change of $\tau\otimes \omega_{\sigma}^{-1}$ to $\Res_{E/F}\GL_{2n}(\A)$. Then the Eisenstein series $E(g, s, \phi_{\tau\otimes\sigma})$ associated to $\phi_{\tau\otimes\sigma}\in \mathcal{A}(M(F)U(\A)\bs H^\delta(\A))_{\tau\otimes \sigma}$, defined in \eqref{eq-EisenSeries} has a pole of order $r$ at $s=\frac{1}{2}$ (see Proposition~\ref{prop-Eisenstein-pole}), and we denote by $\CE_{\tau\otimes \sigma}$ the automorphic representation of $\GSpin_{4n+2}^\delta(\A)$ generated by the $r$-th iterated residues of the Eisenstein series at $s=\frac{1}{2}$.

For an integer $\ell$ with $1\le \ell\le 2n$, let $N_{\ell}$ be the unipotent radical of the parabolic subgroup $Q_\ell=L_{{\ell}}\ltimes N_{{\ell}}$ of $\GSpin_{4n+2}^{\delta}$ stabilizing the flag
\begin{equation*}
0\subset V_1^+ \subset V_2^+ \subset \cdots \subset 	V_{\ell}^+
\end{equation*}
where $V_i^+=\Span\{e_1, \cdots, e_i\}$. 
Let $\psi_{\ell,\alpha}$ be the character of $N_{\ell}(\A)$ which is trivial on $N_{\ell}(F)$, defined in \eqref{eq-psi-ell-w_0-global}, depending on an anisotropic vector $w_0$. If $1\le \ell <2n$, we take $w_0=e_{2n}-\frac{\alpha}{2}e_{-2n}$, where $\alpha\in F^\times$. If $\ell=2n$, we take any anisotropic vector $w_0\in V_0$. The subgroup $\GSpin_{4n+2-2\ell}^{\delta}(\A)$ normalizes $N_{\ell}(\A)$ and acts on the set of characters of $N_{\ell}(\A)$. The stabilizer of $\psi_{\ell,\alpha}$ in $\GSpin_{4n+2-2\ell}^{\delta}(\A)$ is an odd $\GSpin$ group $\GSpin_{4n+1-2\ell}^{\delta,\alpha}(\A)$, associated with a quadratic space of dimension $4n+1-2\ell$ whose symmetric bilinear form is associated with the matrix in \eqref{eq-matrix-quad-form-on-4n-2l+1}. The group $\GSpin_{4n+1-2\ell}^{\delta,\alpha}$ may be split or non-split over $F$, depending on the Hilbert symbol $(\delta,\alpha)$.

For an automorphic form $f\in \CE_{\tau\otimes \sigma}$, we define the Bessel coefficient (or Gelfand-Graev coefficient) of $f$, with respect to $ \psi_{\ell,\alpha}$, 
by
\begin{equation*}
 f^{\psi_{\ell, \alpha}}(g):=\int_{N_\ell(F)\backslash  N_{\ell}(\A)} f(ug)\psi_{\ell,\alpha}^{-1}(u)\ \mathrm{d}u, \quad g\in \GSpin_{4n+2}^{\delta}(\A).
\end{equation*}
Then $f^{\psi_{\ell, \alpha}}$ is invariant under left multiplication by $\GSpin_{4n+1-2\ell}^{\delta,\alpha}(F)$. Following \cite[\S 3.1]{GinzburgRallisSoudry2011} and \cite[\S 1.2]{JiangLiuXuZhang2016}, we consider the space 
$$
\TD_{\psi_{\ell,\alpha}}(\CE_{\tau\otimes \sigma})=\GSpin_{4n+1-2\ell}^{\delta,\alpha}(\A)-\Span \left\{f^{\psi_{\ell, \alpha}}|_{\GSpin_{4n+1-2\ell}^{\delta,\alpha}(\A)}\ | \ f\in \CE_{\tau\otimes \sigma} \right\},
$$
which is a representation of $\GSpin_{4n+1-2\ell}^{\delta,\alpha}(\A)$ under the right translation action. 

The first main result of the paper is the following theorem. 

\begin{theorem}\label{Intro-thm-Main}(Theorem~\ref{thm-Main})
 Let $\tau=\tau_1\boxplus \tau_2\boxplus \cdots \boxplus \tau_r$ be an irreducible unitary generic isobaric sum automorphic representation of $\GL_{2n}(\A)$ associated to distinct cuspidal automorphic representations $\tau_1, \cdots, \tau_r$, such that $L(s, \tau_i, \wedge^2\otimes\omega^{-1})$ has a pole at $s=1$ for each $1\le i\le r$, where $\omega$ is a Hecke character. 
 If there is an automorphic character $\sigma$ of $\GSpin^\delta_2(\A)$ such that $\omega_{\sigma}=\omega$ and 
 $L(\frac{1}{2}, \sigma\times \tau\otimes \omega_\sigma^{-1})\neq 0$, then the following assertions hold.
\begin{enumerate}
\item The representation $\TD_{\psi_{\ell,\alpha}}(\CE_{\tau\otimes\sigma})$ of $\GSpin^{\delta,\alpha}_{4n+1-2\ell}(\A)$ is zero
for all $n<\ell\leq 2n$.
\item For any square class $\alpha$ in $F^\times$, the representation $\TD_{\psi_{n,\alpha}}(\CE_{\tau\otimes\sigma})$ of $\GSpin^{\delta,\alpha}_{2n+1}(\A)$ is cuspidal automorphic.
\item There exists a square class $\alpha$ in $F^\times$ such that the representation $\TD_{\psi_{n,\alpha}}(\CE_{\tau\otimes\sigma})$ of $\GSpin^{\delta,\alpha}_{2n+1}(\A)$ is non-zero, and in this case
    $$
    \TD_{\psi_{n,\alpha}}(\CE_{\tau\otimes\sigma})=\pi_1\oplus\pi_2\oplus\cdots\oplus\pi_k\oplus\cdots,
    $$
    where the $\pi_i$ are irreducible cuspidal automorphic representations of $\GSpin^{\delta,\alpha}_{2n+1}(\A)$, which are nearly equivalent, but
    are not globally equivalent, i.e. the decomposition is multiplicity-free.
\item When $\TD_{\psi_{n,\alpha}}(\CE_{\tau\otimes\sigma})$ is non-zero, any direct summand $\pi$ of $\TD_{\psi_{n,\alpha}}(\CE_{\tau\otimes\sigma})$ has a weak Langlands functorial transfer to $\tau$ in the sense that the Satake parameter of the local unramified component $\tau_v$ of $\tau$ is the local functorial transfer of that of the local unramified component $\pi_v$ of $\pi$ for almost all unramified local places $v$ of $F$.
\item When $\TD_{\psi_{n,\alpha}}(\CE_{\tau\otimes\sigma})$ is non-zero, every irreducible direct summand of $\TD_{\psi_{n,\alpha}}(\CE_{\tau\otimes\sigma})$ has a non-zero Fourier coefficient attached to the partition $[2n-1,1^2]$.  
\item The residual representation $\CE_{\tau\otimes\sigma}$ has a non-zero Fourier coefficient attached to the partition $[2n+1,2n-1,1^2]$.  
\end{enumerate}
\end{theorem}

We refer the reader to Section~\ref{section-twisted-auto-descent} for the unexplained notations. 

Following \cite{JiangLiuXuZhang2016} and \cite{JiangZhang2020Annals}, we call the space $\TD_{\psi_{\ell,\alpha}}(\CE_{\tau\otimes \sigma})$ the \textit{twisted automorphic descent} of $(\tau, \sigma)$ to $\GSpin_{4n+1-2\ell}^{\delta,\alpha}$. Part (1) of Theorem~\ref{Intro-thm-Main}  describes the vanishing property of the ``deeper" twisted automorphic descent (i.e., when $\ell>n$). Parts (2), (3) and (4) of Theorem~\ref{Intro-thm-Main} concern, respectively, the cuspidality property, non-vanishing property, and Langlands functorial property, of the twisted automorphic descent to $\GSpin_{2n+1}^{\delta,\alpha}$ (i.e., when $\ell=n$). Part (5) of Theorem~\ref{Intro-thm-Main} establishes the non-vanishing of a certain Fourier coefficient of the twisted automorphic descent to $\GSpin_{2n+1}^{\delta,\alpha}$. Finally, part (6) of Theorem~\ref{Intro-thm-Main} concerns the non-vanishing of certain Fourier coefficient of the residual representation $\CE_{\tau\otimes\sigma}$. 

\begin{remark}
We remark that our twisted automorphic descent construction makes essential use of the non-vanishing condition $L(\frac{1}{2},\sigma\times\tau\otimes \omega_{\sigma}^{-1})\not=0$, which is one of the main new inputs in our construction.
\end{remark}

\begin{remark}
Let $\tau=\tau_1\boxplus \tau_2\boxplus \cdots \boxplus \tau_r$ be an irreducible unitary generic isobaric sum automorphic representation of $\GL_{2n}(\A)$ associated to distinct $\tau_1, \cdots, \tau_r$, such that $L(s, \tau_i, \wedge^2\otimes\omega)$ has a pole at $s=1$ for each $1\le i\le r$, for a Hecke character $\omega$ of $\GL_1(\A)$. Then the automorphic descent construction of Hundley and Sayag \cite{HundleySayag2016} produces an irreducible generic cuspidal automorphic representation $\pi_0$ of $\GSpin_{2n+1}(\A)$, where $\GSpin_{2n+1}$ is $F$-split, such that $\pi_0$ has a weak Langlands functorial transfer to $\tau$ and has central character $\omega_{\pi_0}=\omega$. 
Now suppose that there is an auxiliary automorphic representation $\sigma$ of $\GSpin_2^{\delta}(\A)$ with $\omega_{\sigma}=\omega$ satisfying the non-vanishing condition $L(\frac{1}{2},\sigma\times\tau\otimes \omega_{\sigma}^{-1})\not=0$. Then Theorem~\ref{Intro-thm-Main} produces additional irreducible cuspidal automorphic representations $\pi$ of the $F$-split group $\GSpin_{2n+1}(\A)$. These arise when $\alpha\in F^\times$ is chosen so that $\GSpin_{2n+1}^{\delta,\alpha}$ is $F$-split. In this case, every such representation $\pi$ obtained via twisted automorphic descent is nearly equivalent to $\pi_0$ by part (4) of Theorem~\ref{Intro-thm-Main}. Moreover, for those $\alpha$ for which $\GSpin_{2n+1}^{\delta,\alpha}$ is non-split over $F$, the twisted automorphic descent $\TD_{\psi_{n,\alpha}}(\CE_{\tau\otimes\sigma})$  produces irreducible cuspidal automorphic representations of $\GSpin_{2n+1}^{\delta,\alpha}(\A)$ that are still nearly equivalent to $\pi_0$, and are therefore expected to lie in the same global Vogan packet as $\pi_0$.
\end{remark}

\begin{remark}
When $\sigma$ satisfies the additional property that $\omega_\sigma=1$, it may be regarded as a character of $\SO_2^{\delta}(\A)$, and the representation $\TD_{\psi_{\ell,\alpha}}(\CE_{\tau\otimes\sigma})$ may be viewed as a representation of $\SO^{\delta,\alpha}_{4n+1-2\ell}(\A)$. If, in addition, $\tau$ is assumed to be cuspidal, then Theorem~\ref{Intro-thm-Main} recovers \cite[Theorem 1.2]{JiangLiuXuZhang2016}. Thus, our result generalizes \cite[Theorem 1.2]{JiangLiuXuZhang2016} in that $\omega_{\sigma}$ is allowed to be nontrivial and $\tau$ is allowed to be an isobaric sum. 
\end{remark}

\subsection{Application to global Jacquet--Langlands correspondence between $\GL_2$ and $D^\times$}
\label{subsection-Intro-application-JacquetLanglands}
The global Jacquet--Langlands correspondence between automorphic representations on $\GL_2$ and $D^\times$, where $D$ is a quaternion division algebra over a number field $F$, is one of the earliest and most fundamental examples of Langlands functoriality. The existence of this correspondence was first established by Jacquet and Langlands in 1970 in \cite{JacquetLanglands1970} by means of the trace formula method. Shortly afterwards, Shimizu provided an explicit realization of the correspondence via theta series in \cite{Shimizu1972}. This was later incorporated into the general framework of theta correspondence for reductive dual pairs; see \cite{Howe1979}. More recently, Jiang, Liu, Xu and Zhang gave a new explicit construction of the global Jacquet--Langlands correspondence for $\PGL_2$ in \cite{JiangLiuXuZhang2016}, as an application of twisted automorphic descent for odd special orthogonal groups.

In this paper, following the method in \cite{JiangLiuXuZhang2016}, we give a new proof of the global Jacquet--Langlands correspondence between $\GL_2$ and $D^\times$. Our proof is based on Theorem~\ref{Intro-thm-Main} in the special case $n=1$, and may be viewed as a $\GSpin$ generalization of \cite{JiangLiuXuZhang2016}.

When $n=1$, if the quadratic form associated with $J_{\delta,\alpha}$ in \eqref{eq-J-delta-alpha} is non-split over $F$, then $\GSpin_{3}^{\delta,\alpha}\cong D_{\delta,\alpha}^\times$, where $D_{\delta,\alpha}$ is the quaternion algebra $\left( \frac{-\delta,-\alpha}{F}\right)$. 
The quaternion algebra $D_{\delta,\alpha}$ is uniquely determined by the class of $\alpha$ in the quotient group $F^\times/ \mathrm{Nm}_{E/F} (E^\times)$, where $E=F(\sqrt{-\delta})$. Let $\eta_{E/F}$ denote the quadratic character of $F^\times\backslash \A^\times$ associated to the quadratic extension $E/F$ via the global class field theory.

\begin{theorem}(Theorem~\ref{thm-JL})
\label{Intro-thm-JL}
Let $\delta\in F^\times$ be such that $-\delta\not\in (F^\times)^2$. Let $\sigma$ be an automorphic character of $\GSpin_2^{\delta}(\A)=\Res_{E/F} \GL_1(\A)$, and $\tau$ an irreducible unitary cuspidal automorphic representation of $\GL_{2}(\A)$. 
If the residual representation $\CE_{\tau\otimes\sigma}$ is not identically zero, then the following hold. 
\begin{enumerate}
\item  The set of $\alpha\in F^\times$ such that the $\psi_{1,\alpha}$-Fourier coefficient $\CE^{\psi_{1,\alpha}}_{\tau\otimes\sigma}$ is not identically zero is a single coset $\alpha_0\cdot \mathrm{Nm}_{E/F} (E^\times)$.
\item For any $\alpha$ in the coset $\alpha_0\cdot \mathrm{Nm}_{E/F} (E^\times)$, the twisted automorphic descent $\TD_{\psi_{1,\alpha}} (\CE_{\tau\otimes\sigma})$ is irreducible, and has the property that 
\begin{equation*}
\JL(\TD_{\psi_{1,\alpha}} (\CE_{\tau\otimes\sigma}) ) \cong \tau. 
\end{equation*}
Here, $\JL$ denotes the Jacquet--Langlands transfer map. 

\item The norm class $\alpha_0$ is determined by the property that  $D_{\delta, \alpha_0}$ is ramified at a place $v$ of $F$ if and only if $\epsilon(\frac{1}{2}, \BC( \tau_v)\times \sigma_v, \psi_v) \sigma_v(-1) \eta_{E/F,v}(-1)=-1$. Here, $\BC$ denotes the base change map.
\end{enumerate}
\end{theorem}

By choosing $\delta$ and $\sigma$ appropriately, one can recover all infinitely-dimensional cuspidal automorphic representations of $D^\times(\A)$, as follows. 

\begin{theorem}(Theorem~\ref{thm-JL-2})
\label{Intro-thm-JL-2}	
Let $\delta\in F^\times$ be such that $-\delta\not\in (F^\times)^2$. Let $D$ be a quaternion division algebra containing the quadratic extension $E=F(\sqrt{-\delta})$ of $F$. For any given infinite-dimensional irreducible cuspidal automorphic representation $\pi$ of $D^\times(\A)$, there exists a character $\sigma$ of $\GSpin_2^{\delta}(\A)$ such that the residual representation $\CE_{\JL(\pi)\otimes \sigma}$ is non-zero, and an element $\alpha\in F^\times$ such that $D_{\delta,\alpha}\cong D$ and 
\begin{equation*}
	\TD_{\psi_{1,\alpha}} (\CE_{\JL(\pi)\otimes\sigma})\cong \pi.
\end{equation*}
Moreover, any such $\sigma$ satisfies the period condition 
\begin{equation*}
	\mathcal{P}(\varphi_\pi, \sigma^{-1}) \not=0, \quad \text{ for some }\varphi_\pi \in V_\pi,
\end{equation*}
which is equivalent to the conditions that
\begin{equation*}
L(\frac{1}{2}, \BC(\JL(\pi))\otimes\sigma^{-1})\not=0 \text{ and } \Hom_{\GSpin_2^{\delta}(F_v)} (\pi_v,\sigma_v)\not=0 \text{ for all }v.	
\end{equation*}
Here, $\mathcal{P}(\varphi_\pi, \sigma^{-1})$ denotes the Bessel period for $(\varphi_\pi,\sigma^{-1})$ defined in \eqref{eq-JL-eqBesselPeriod}. 
\end{theorem}

\subsection{Application to the global Gan--Gross--Prasad conjecture for $(\GSpin_{2n+1}, \GSpin_2^\delta)$}
\label{subsection-Intro-application-GGP}
We present another application of Theorem~\ref{Intro-thm-Main}, namely to the global GGP conjecture for $\GSpin$ groups. This conjecture was first proposed by Gross and Prasad in \cite{GrossPrasad1992, GrossPrasad1994} for special orthogonal groups, and was subsequently reformulated by Gan, Gross, and Prasad in \cite{GanGrossPrasad2012} in full generality for all classical groups and metaplectic groups. The global GGP conjecture has been extensively studied through various different techniques. The approach using automorphic descent originates from \cite{GinzburgJiangRallis2004}, and more recent attack along this line include \cite{JiangZhang2020Annals, JiangLiuXu2020, LiuXu2023}.

The $\GSpin$ analogue of the global GGP conjecture can be formulated in a manner similar to  the special orthogonal group case in \cite{GanGrossPrasad2012}, after taking central characters into consideration.
We first recall the conjecture for $\GSpin$ groups, following \cite{GanGrossPrasad2012}. 
Let $W$ be a non-degenerate quadratic space over $F$ and let $W_0\subset W$ be a non-degenerate subspace such that the orthogonal complement $W_0^\perp$ is split of dimension  $2r+1$. Let $N_r$ be the unipotent radical of the standard parabolic subgroup of $\GSpin(W)$ stabilizing a complete flag of isotropic subspaces determined by $W_0^\perp$. Let $\psi_{r,w_0}$ be a character of $N_r(\A)$ as in \eqref{eq-psi-ell-w_0-global}, depending on an anisotropic vector $w_0$. Then $\GSpin(W_0)(\A)$ stabilizes the character $\psi_{r,w_0}$ via the adjoint action, and the subgroup $\GSpin(W_0) \ltimes N_r$ is a Bessel subgroup of $\GSpin(W)$.

Let $\pi$ be an automorphic representation on $\GSpin(W)(\A)$, and $\pi^\prime$ a cuspidal automorphic representation on $\GSpin(W_0)(\A)$, such that $\omega_{\pi}\omega_{\pi^\prime}=1$. Here, $\omega_{\pi}$ and $\omega_{\pi^\prime}$ are the restrictions of the central characters of $\pi$ and $\pi^\prime$ to $\ker(\pr)(\A)=\GL_1(\A)$ respectively, where  $\pr: \GSpin(V)\to \SO(V)$ is the projection map, and $\ker(\pr)=\GL_1$ lies in both the centers of $\GSpin(W)$ and $\GSpin(W_0)$.
Let $\varphi\in V_{\pi}$ and $\varphi^\prime\in V_{\pi^\prime}$. 
 The Bessel period of $(\varphi,\varphi^\prime)$, with respect to the character $\psi_{r,w_0}$, is
\begin{equation*}
\mathcal{P}(\varphi, \varphi^\prime):=\int_{\ker(\pr)(\A) \GSpin(W_0)(F)\backslash \GSpin(W_0)(\A)} \varphi^{N_r, \psi_{r, w_0}}(g) \varphi^\prime(g) dg,	
\end{equation*}
where
\begin{equation*}
	\varphi^{N_r, \psi_{r, w_0}}(g)=\int_{N_r(F)\backslash N_r(\A)} \varphi(ug) \psi_{r,w_0}^{-1}(u)du.
\end{equation*}
Since $\omega_\pi \omega_{\pi^\prime}=1$, the integrand $\varphi^{N_r, \psi_{r, w_0}}(g) \varphi^\prime(g)$ in $\mathcal{P}(\varphi, \varphi^\prime)$ is invariant under $\ker(\pr)(\A)$. Moreover, the integral defining $\mathcal{P}(\varphi, \varphi^\prime)$ is absolutely  convergent since $\varphi^\prime$ is a cusp form. 

The $\GSpin$ analogue of \cite[Conjecture 24.1]{GanGrossPrasad2012} is the following.
 
\begin{conjecture}(Global GGP conjecture, \cite{GanGrossPrasad2012})
\label{GGP-Conjecture}
Let $\pi$, $\pi^\prime$ be irreducible tempered cuspidal automorphic representations of $\GSpin(W)(\A)$ and $\GSpin(W_0)(\A)$ respectively, such that $\omega_{\pi}\omega_{\pi^\prime}=1$, which occur with multiplicity one in the discrete spectrum. The following are equivalent.
\begin{enumerate}
\item The Bessel period for $(\pi, \pi^\prime)$ is non-zero, i.e.,  $\mathcal{P}(\varphi, \varphi^\prime)\not=0$ for some $\varphi\in V_\pi$ and $\varphi^\prime\in V_{\pi^\prime}$.
\item The central $L$-value $L(\frac{1}{2}, \pi\times \pi^\prime)\not=0$.
\end{enumerate}
\end{conjecture}

\begin{remark}
Some low rank cases of Conjecture~\ref{GGP-Conjecture} are considered in \cite{PrasadTakloo-Bighash2011} (in view of the isomorphism $\GSp_4\cong \GSpin_5$) and \cite{Emory2020}.
\end{remark}

\begin{remark}
As explained in \cite[\S 26]{GanGrossPrasad2012}, one can reformulate Conjecture~\ref{GGP-Conjecture} in terms of the (conjectural) endoscopic classification of representations for $\GSpin$ groups and Vogan $L$-packets,, in a form analogous to \cite[Conjecture 26.1]{GanGrossPrasad2012}. In this case, the condition (1) in Conjecture~\ref{GGP-Conjecture} can be replaced by the following condition:
\begin{enumerate}
\item[(1')] There exists a pair $(\Pi, \Pi^\prime)$ in the global Vogan packet associated to $(\pi, \pi^\prime)$ such that the Bessel period for $(\Pi, \Pi^\prime)$ is non-zero.
\end{enumerate}
\end{remark}

In our previous work \cite{Yan2025GGP}, we obtained a certain case of Conjecture~\ref{GGP-Conjecture} for $\GSpin$ groups, using the Rankin--Selberg method. We recall the following result from \cite{Yan2025GGP}, specialized to the pair $(\GSpin_{2n+1}, \GSpin_2^\delta)$. Here, $\GSpin_{2n+1}$ is $F$-split.

\begin{theorem}\cite{Yan2025GGP}
\label{Intro-thm-GGP-forward}
Let $\delta\in F^\times$ be such that $-\delta\not\in (F^\times)^2$. 
Let $\sigma$ be an automorphic character of $\GSpin_2^{\delta}(\A)$.
Let $\pi$ be an irreducible cuspidal automorphic representation of $\GSpin_{2n+1}(\A)$ with central character $\omega_{\sigma}^{-1}$ that admits a weak functorial transfer to an isobaric sum automorphic representation $\tau$ of $\GL_{2n}(\A)$, and assume that $\tau$ is also a weak functorial transfer of an irreducible generic cuspidal representation of $\GSpin_{2n+1}(\A)$. If the Bessel period of $(\pi, \sigma)$ is non-zero, then the central $L$-value $L(\frac{1}{2}, \pi\times\sigma)$ is non-zero.
\end{theorem}

\begin{remark}
The existence of weak functorial transfer for irreducible cuspidal automorphic representations of split $\GSpin$ groups is known, due to the work of Cai, Friedberg, and Kaplan in \cite{CaiFriedbergKaplan2024}. 
\end{remark}

\begin{remark}
Shahidi's conjecture \cite{Shahidi1990}  asserts that each tempered $L$-packet contains a generic member (see also \cite{Shahidi2011}).
Thus, the condition that $\pi$ is tempered is conjecturally equivalent to the condition that the weak functorial transfer of $\pi$ is also a weak functorial transfer of an irreducible cuspidal automorphic representation of $\GSpin_{2n+1}(\A)$.
\end{remark}

\begin{remark}
Theorem~\ref{Intro-thm-GGP-forward} confirms the $(1)\Rightarrow (2)$ direction of the global GGP conjecture for $(\GSpin_{2n+1}, \GSpin_2^\delta)$.
\end{remark}

In this paper, we prove the following theorem, as another application of Theorem~\ref{Intro-thm-Main}.

\begin{theorem}(Theorem~\ref{thm-GGP-converse})
\label{Intro-thm-GGP-converse}
Let $\delta\in F^\times$ be such that $-\delta\not\in (F^\times)^2$. 
Let $\sigma$ be an automorphic character of $\GSpin_2^{\delta}(\A)$.
Let $\pi$ be an irreducible cuspidal automorphic representation of $\GSpin_{2n+1}(\A)$ with central character $\omega_{\sigma}^{-1}$ that admits a weak functorial transfer to an isobaric sum automorphic representation $\tau$ of $\GL_{2n}(\A)$, and assume that $\tau$ is also a weak functorial transfer of an irreducible generic cuspidal representation of $\GSpin_{2n+1}(\A)$.
If
\begin{equation*}
L(\frac{1}{2}, \pi\times\sigma )\not=0,	
\end{equation*}
then the following holds.
\begin{enumerate}
\item There exist some $\alpha\in F^\times$ such that the twisted automorphic descent $\TD_{\psi_{n,\alpha}}(\CE_{\tau\otimes\sigma^{-1}})$ to $\GSpin^{\delta,\alpha}_{2n+1}(\A)$ is non-zero.
\item The twisted automorphic descent has a multiplicity-free direct sum decomposition
\begin{equation*}
\TD_{\psi_{n,\alpha}}(\CE_{\tau\otimes\sigma^{-1}})= \bigoplus_{i} \pi_\alpha^{(i)},	
\end{equation*}
where each $\pi_\alpha^{(i)}$ is an irreducible cuspidal automorphic representation of $\GSpin^{\delta,\alpha}_{2n+1}(\A)$ with central character $\omega_{\sigma}^{-1}$, and each $\pi_\alpha^{(i)}$ is nearly equivalent to $\pi$.
\item 	The Bessel period for $(\pi_{\alpha}^{(i)}, \sigma)$ is non-zero for each $\pi_{\alpha}^{(i)}$.
\end{enumerate} 
\end{theorem}

As an immediate consequence of Theorem~\ref{Intro-thm-GGP-converse}, we obtain the following result on the  global GGP conjecture for $(\GSpin_{2n+1}, \GSpin_2^\delta)$.

\begin{theorem}
\label{Intro-thm-GGP-converse-concise}
Let $\delta\in F^\times$ be such that $-\delta\not\in (F^\times)^2$. 
Let $\sigma$ be an automorphic character of $\GSpin_2^{\delta}(\A)$.
Let $\pi$ be an irreducible cuspidal automorphic representation of $\GSpin_{2n+1}(\A)$ with central character $\omega_{\sigma}^{-1}$ that admits a weak functorial transfer to an isobaric sum automorphic representation $\tau$ of $\GL_{2n}(\A)$, and assume that $\tau$ is also a weak functorial transfer of an irreducible generic cuspidal representation of $\GSpin_{2n+1}(\A)$.
If
\begin{equation*}
L(\frac{1}{2}, \pi\times\sigma )\not=0,	
\end{equation*}
then there exist some $\alpha\in F^\times$ and a cuspidal representation $\pi^\prime$ of $\GSpin_{2n+1}^{\delta, \alpha}(\A)$ with central character $\omega_{\sigma}^{-1}$ such that $\pi^\prime$ is nearly equivalent to $\pi$, and 
	the Bessel period of $(\pi^\prime, \sigma)$ is non-zero.
\end{theorem}

\begin{remark}
Theorem~\ref{Intro-thm-GGP-converse-concise} confirms the $(2)\Rightarrow (1^\prime)$ direction of the global GGP conjecture for $(\GSpin_{2n+1}, \GSpin_2^\delta)$. 
\end{remark}

\begin{remark}
The same method as in the proof of Theorem~\ref{Intro-thm-GGP-converse} applies to other cases of the global GGP conjecture for $\GSpin$ groups, which we will address in future work.	
\end{remark}

This paper is organized as follows. 
In Section~\ref{section-GSpin-groups}, we discuss the structure of $\GSpin$ groups.
In Section~\ref{section-twisted-auto-descent}, we discuss the construction of the twisted automorphic descent, including the residual representations and Bessel coefficients, and we restate the main theorem on twisted automorphic descent in Section~\ref{subsection-twisted-auto-descent}. In Section~\ref{section-vanishing-and-cuspidality}, we prove the vanishing property of the ``deeper" twisted automorphic descent, i.e., part (1) of Theorem~\ref{Intro-thm-Main}. We also prove the cuspidality, i.e., part (2) of Theorem~\ref{Intro-thm-Main}, using a formula for the constant term of the Bessel coefficients of the residual representation. This formula is proved in Section~\ref{section-certain-fourier-coef}. In Section~\ref{section-non-vanishing}, we prove parts (3), (5), (6) of Theorem~\ref{Intro-thm-Main}. Part (4) of Theorem~\ref{Intro-thm-Main} is proved in Section~\ref{section-functorial}. In Section~\ref{section-application-JacquetLanglands}, we prove Theorem~\ref{Intro-thm-JL} and Theorem~\ref{Intro-thm-JL-2}. Finally, in Section~\ref{section-application-GGP}, we prove Theorem~\ref{Intro-thm-GGP-converse}.

\subsection*{Acknowledgements}
I would like to thank Dihua Jiang and Hang Xue for insightful discussions related to this project, and Jim Cogdell for his support over the years.  
The author has been partially supported by a summer grant from the Department of Mathematics at the University of Arizona in 2026.

\section{The $\GSpin$ groups}
\label{section-GSpin-groups}
In this section we recall how the $\GSpin$ group is realized via the Clifford algebra. Our main reference is \cite{Shimura2004}, which defines the $\GSpin$ group over any field of characteristic not equal to 2. 
\subsection{The $\GSpin$ groups}
\label{subsection-GSpin-groups}

Let $V$ be a finite-dimensional quadratic space over a field $F$, which is either a number field or a local field of characteristic zero, equipped with a non-degenerate quadratic form $q_V$. We also use $q_V(\cdot, \cdot)$ to denote the induced symmetric bilinear form. The quadratic form $q_V$  gives rise to the orthogonal group and special orthogonal group:
\begin{equation*}
\begin{split}
\Orth(V) &=\{g\in \GL_F(V): q_V(gv)=q_V(v) \text{ for every $v\in V$}\}, \\
\SO(V) &=\{g\in \Orth(V):\det(g)=1\}.
\end{split}
\end{equation*}
The Clifford algebra $C(V)$ associated to $(V, q_V)$ is the quotient of the tensor algebra $T(V)=\bigoplus_{d=0}^{\infty} V^{\otimes d}$
by the two sided ideal generated by all $v\otimes v - q_V(v)$, for $v\in V$. In this paper, to ease notation, we also write $v_1v_2\cdots v_k$ for $v_1\otimes v_2\otimes \cdots \otimes v_k$. 

The inclusion map $V\to T(V)$ induces a canonical injection $\iota: V\to C(V)$. 
Let \( C^\pm(V) \) be the even and odd part of \( C(V) \), so that $C(V)=C^+(V)\oplus C^-(V)$. 
The Clifford algebra $C(V)$ has a canonical involution $*:C(V)\to C(V)$, which is given by reversing the order of products in the tensor algebra. 
We also define
\begin{equation*}
\alpha:C(V)\to C(V), \quad \alpha(v_{+}+v_{-})=v_{+}-v_{-}, 	
\end{equation*}
where $v_{+}\in C^{+}(V)$, and $v_{-}\in C^{-}(V)$. 
Moreover, for all $x\in C(V)$ we define the Clifford involution
\begin{equation*}
\overline{x}=(\alpha(x))^*=\alpha(x^*)
\end{equation*}
and the Clifford norm 
\begin{equation*}
N: C(V)\to C(V), \quad x\mapsto x \overline{x}.	
\end{equation*}
Note that $N(\lambda x)=\lambda^2 N(x)$ for all $\lambda\in F$.

The $\GPin$ group and $\GSpin$ group of $(V, q_V)$ are defined by
\begin{equation*}
\begin{split}
\GPin(V) &=\{g\in C(V)^\times: \alpha(g) V g^{-1}=V\},	\\
\GSpin(V) &=\GPin(V)\cap C^{+}(V).
\end{split}
\end{equation*}
Because $\alpha$ acts as identity on $C^+(V)$, we have the inclusion $\GSpin(V)\subset \GPin(V)$. In fact, we have $[\GPin(V):\GSpin(V)]=2$, and $\GSpin(V)$ is the identity component of $\GPin(V)$.

There is a natural projection map 
\begin{equation*}
\pr:\GPin(V)\to \Orth(V)	
\end{equation*} 
sending $g\in \GPin(V)$ to the map $v\mapsto \alpha(g)vg^{-1}$. 
 The kernel of this homomorphism consists of scalars, thus we get an exact sequence
\begin{equation*}
1\to F^\times \to \GPin(V)\to \Orth(V)\to 1.	
\end{equation*} 
Moreover, $\pr^{-1}(\SO(V))=\GSpin(V)$, and thus we have an exact sequence
\begin{equation*}
1\to F^\times \to \GSpin(V)\to \SO(V)\to 1,	
\end{equation*} 
and
a commutative diagram
\[
\begin{tikzcd}
1 \ar[r] & F^\times \ar[r] \ar[d, equals] & \mathrm{GSpin}(V) \ar[r] \ar[d, "\subseteq"] & \mathrm{SO}(V) \ar[r] \ar[d, "\subseteq"] & 1 \\
1 \ar[r] & F^\times \ar[r] & \mathrm{GPin}(V) \ar[r] & \mathrm{O}(V) \ar[r] & 1
\end{tikzcd}
\]
where the rows are exact.

Note that the restriction of the Clifford norm $N$ to $\GPin(V)$ has its image in $F^\times$, thus we have a homomorphism
\begin{equation*}
N:\GPin(V)\to F^\times , \quad g\mapsto g\overline{g}.	
\end{equation*}
For any $z\in \ker(\pr)=F^\times$, we have $N(z)=z^2$. Also, for any $g\in \GPin(V)$, we have $g^{-1}=\frac{1}{N(g)}\overline{g}$. In particular, for any $g\in \GSpin(V)$, we have $g^{-1}=\frac{1}{N(g)}\overline{g}=\frac{1}{g g^*}g^*$.

In the next lemma, we describe the centers of $\GPin(V)$ and $\GSpin(V)$. 
Let $\{b_1, \cdots, b_n\}$ be an orthogonal basis of $V$, and denote $\zeta=b_1\cdots b_n$. 
It is clear that $\zeta\in \GSpin(V)$ if $\dim(V)$ is even, and $\zeta\in \GPin(V)\backslash  \GSpin(V)$ if $\dim(V)$ is odd.

\begin{lemma}\cite[Theorem 3.6]{Shimura2004}
\label{lemma-center}
Let $Z_{\GPin(V)}$ and $Z_{\GSpin(V)}$ denote the centers of $\GPin(V)$ and $\GSpin(V)$ respectively. 	
\begin{enumerate}
\item If $\dim(V)$ is even and $\dim(V)>2$, then 
\begin{equation*}
Z_{\GPin(V)}=F^\times, \quad Z_{\GSpin(V)}=F^\times \cup F^\times \zeta.	
\end{equation*}
If $\dim(V)=2$, then 
\begin{equation*}
	Z_{\GPin(V)}=F^\times, \quad Z_{\GSpin(V)}=\GSpin(V).
\end{equation*}

\item If $\dim(V)$ is odd, then 
\begin{equation*}
Z_{\GPin(V)}=F^\times \cup F^\times \zeta, \quad Z_{\GSpin(V)}=F^\times. 
\end{equation*}
\end{enumerate}
In particular, $\ker(\pr)$ is the connected component of the center of $\GPin(V)$ as well as $\GSpin(V)$. 
\end{lemma}

Note that if $W\subset V$ is a non-degenerate subspace, then we have the inclusions of Clifford algebra $C(W)\subset C(V)$ and of even Clifford algebra $C^+(W)\subset C^+(V)$, which map $F^\times$ onto itself. These
induce the inclusions
 $\GPin(W)\subset \GPin(V)$ and $\GSpin(W)\subset \GSpin(V)$.

\begin{remark}
When $V$ is quasi-split, the $\GSpin$ group $\GSpin(V)$ can also be defined using based root datum. See, for example, \cite{AsgariShahidi2006, HundleySayag2016, AsgariCogdellShahidi2024}, for more details. 
\end{remark}

\begin{remark}
When $\dim_F V$ is small, we have the following description of $\GSpin(V)$:
\begin{itemize}
\item When $\dim_F V=1$, $\GSpin(V)(F)=F^\times$.
\item When $\dim_F V=2$ and $V$ is isotropic, $\GSpin(V)(F)=\GL_1(F)\times \GL_1(F)$.
\item When $\dim_F V=2$ and $V$ is anisotropic, by choosing a suitable basis and scaling the quadratic form if necessary, we may assume the quadratic form of $V$ is associated with $\left(\begin{smallmatrix} 1 & 0\\0& \delta\end{smallmatrix}\right)$ where $-\delta\not\in {F^\times}^2$, then $\GSpin(V)(F)=E^\times=\Res_{E/F}\GL_1(F)$ where $E=F(\sqrt{-\delta})$. 
\item When $\dim_F V=3$, $\GSpin(V)(F)=D^\times$ where $D$ is a quaternion algebra over $F$. In particular, if $D$ is split then $\GSpin(V)(F)\cong \GL_2(F)$. 
\end{itemize}
\end{remark}

\subsection{Parabolic subgroups}
\label{subsection-parabolic-subgroups}

In this section, we discuss the parabolic subgroups of $\GSpin(V)$ in terms of Clifford algebra.
Let $X$ be an isotropic subspace of $V$, and let $P_X$ be the parabolic subgroup of $\GSpin(V)$ stabilizing $X$, i.e., 
\begin{equation*}
 P_X=\{ g\in \GSpin(V): g X g^{-1}=X\}. 	
\end{equation*}
The image of $P_X$ under $\pr$, $\pr(P_X)$, is the corresponding parabolic subgroup of $\SO(V)$.

Associated to $X$, we can define an increasing filtration $\mathcal{W}^{X}_{\bullet}$ on $C(V)$ as follows. First, we define an increasing filtration on $V$ by
\begin{equation*}
\mathcal{W}_{k}^{X}(V):=\begin{cases}
0 & \text{ if }k<-1,\\
X & \text{ if }k=-1,\\
X^\perp & \text{ if }k=0,\\
V &\text{ if }k\ge 1.
\end{cases}
\end{equation*}
Then we define
\begin{equation*}
\mathcal{W}_{k}^{X}(C(V)):=\{x\in C(V): x=\sum x_1 x_2 \cdots x_m, x_i\in \mathcal{W}_{\alpha_i}^{X}(V), \alpha_1+\alpha_2+\cdots \alpha_m\le k \}.
\end{equation*}

\begin{lemma}
\cite[Lemma 2.3]{Pollack2018}
The parabolic subgroup $P_X$ and its unipotent radical $N_X$ are given by
\begin{equation*}
P_X = 	\mathcal{W}_{0}^{X}(C(V))\cap \GSpin(V), \quad N_X=(1+ \mathcal{W}_{-1}^{X}(C(V))) \cap \GSpin(V).
\end{equation*}
\label{lemma-parabolic-via-Clifford}
\end{lemma}

Now suppose that $Y$ is an isotropic subspace of $V$ which is dual to $X$, and we take an orthogonal sum decomposition $V=X\oplus W\oplus Y$. 
Let $M_X$ be the Levi subgroup of the parabolic subgroup $P_X$ of $\GSpin(V)$. 
Then the image in $\SO(V)$ of $M_X$ under the projection map $\pr$ is identified with $\GL_F(X)\times \SO(W)$. 
To describe the lifting of the $\GL_F(X)$ factor to $\GSpin(V)$ explicitly, we assume  $X=\Span\{e_1, \cdots, e_m\}$ and $Y=\Span\{e_{-m}, \cdots, e_{-1}\}$ such that $q_V( e_i, e_{-j})=\delta_{i,j}$ for $1\le i, j\le m$. Let $\{f_1, \cdots,f_{\dim W}\}$ be a basis of $W$. Then we have the following basis of $V$:
\begin{equation}
\label{eq-V-basis}
	\{e_1, \cdots, e_m, f_1, \cdots, f_{\dim W}, e_{-m}, \cdots, e_{-1}\}.
\end{equation}
Recall that $w_m$ is the $m\times m$ anti-diagonal matrix with entries 1 on the anti-diagonal and 0 everywhere else. 
\begin{lemma}
\cite[Proposition 4.8]{Shimura2004}
\label{lemma-Levi-lifting}
There exists an injective algebraic group homomorphism 
\begin{equation*}
	\ell_m:\GL_F(X)=\GL_m(F)\to P_X
\end{equation*}
with the following properties:
\begin{enumerate}
\item For any $g\in \GL_m(F)$, we have $\pr(\ell_m(g))=\begin{pmatrix} g & & \\ &I_W &\\ & & w_m {}^tg^{-1} w_m \end{pmatrix}\in \SO(V)$, where the matrix is written with respect to the ordered basis \eqref{eq-V-basis}.
\item For any $g\in \GL_m(F)$, we have $N(\ell_m(g))=\det(g)$. 
\item For any diagonal element $\diag(a_1, \cdots, a_m)\in \GL_m(F)$, we have $\ell_m(\diag(a_1, \cdots, a_m))=\prod_{i=1}^m (a_i e_ie_{-i}+e_{-i} e_i)$. 
\item If $1\le m^\prime <m$ and $g_0\in \GL_F(\Span\{e_1, \cdots, e_{m^\prime} \} )=\GL_{m^\prime}(F)$, let $g\in \GL_m(F)$ be the element such that $g$ acts as $g_0$ on $\Span\{e_1, \cdots, e_{m^\prime} \}$ and acts trivially on $\Span\{e_{m^\prime+1}, \cdots, e_{m} \}$. Then $\ell_m(g)=\ell_{m^\prime}(g_0)$.
\item The image $\ell_m(\GL_m(F))$ is an algebraic subgroup of $P_X$.
\end{enumerate}
Moreover, $\ell_m$ is uniquely determined, as an algebraic group homomorphism, by properties (1)-(3).
\end{lemma}

Let $Z_m\subset \GL_m(F)$ be the unipotent radical of the Borel subgroup $B_m=T_m Z_m$ of $\GL_m(F)$ stabilizing the flag
\begin{equation*}
\Span\{e_1\} \subset \Span\{e_1, e_2\} \subset \cdots \subset \Span\{e_1, e_2, \cdots, e_m\}.	
\end{equation*}
In the next lemma, we describe the lifting of the unipotent group $Z_m$.

\begin{lemma}
\label{lemma-Levi-Unipotent-lifting}
For $1\le i <j\le m$, let $E_{i,j}$ denote the $m\times m$ matrix with one in the $(i,j)$-th entry and zeros everywhere else, and let $u_{i,j}(\gamma)=I_m + \gamma  E_{i,j}$ be the unipotent matrix in $\GL_m(F)$ where $\gamma \in F$. Then $\ell_m(u_{i,j}(\gamma))=1+\gamma  e_i e_{-j}$. 
\end{lemma}

\begin{proof}
See \cite[Lemma 2.6]{Yan2025}.	
\end{proof}

\begin{remark}
\label{remark-unipotent-variety-isomorphism}
Note that the projection map $\pr:\GSpin(V)\to \SO(V)$ induces an isomorphism of unipotent varieties, so we may simply use unipotent elements of subgroups of $\SO(V)$ to represent their lifts to $\GSpin(V)$.
\end{remark}

\section{Twisted automorphic descent}
\label{section-twisted-auto-descent}

\subsection{The residual representations}
Let $F$ be a number field and let $\A=\A_F$ be the ring of adeles of $F$.
Let $V$ be a quadratic space over $F$ of dimension $4n+2$ equipped with a non-degenerate quadratic form $q_V$. 
We assume that the Witt index of $V$ is $2n$. 
Let $V^+$ be a maximal totally isotropic subspace of $V$ and $V^-$ to be its dual, so that $V$ has the following decomposition
\begin{equation*}
V=V^+\oplus V_0 \oplus V^-,	
\end{equation*}
where $V_0=(V^+\oplus V^-)^\perp$ denotes the anisotropic subspace of dimension 2. 
We choose a basis $\{e_1, e_2, \cdots, e_{2n}\}$ of $V^+$ and a basis $\{e_{-1}, e_{-2}, \cdots, e_{-2n}\}$ of $V^-$ so that $q_V(e_i, e_{-j})=\delta_{i,j}$ for all $1\le i, j\le 2n$. For each $1\le i \le 2n$, denote $V_i^+=\Span\{e_1, \cdots, e_i \}$, and $V_i^{-}=\Span\{e_{-1}, \cdots, e_{-i}\}$.

We may take the quadratic form on $V_0$ to be associated with $J_\delta=\left(\begin{smallmatrix} 1 &0\\ 0 &\delta\end{smallmatrix}\right)$, where $-\delta\not\in {F^\times}^2$, and the quadratic form on $V$ is taken to be associated with the matrix 
\begin{equation}
\label{eq-matrix-quad-form-on-V}
	\begin{pmatrix}
   &&w_{2n}\\
   &J_\delta&\\
   w_{2n}&&
  \end{pmatrix}.
\end{equation}
Recall that $w_{r}$ is the matrix defined inductively by $w_{r}=\left(\begin{smallmatrix} 0 & w_{r-1}\\ 1&0 \end{smallmatrix}\right)$ for $r\ge 2$ and $w_1=1$.
Let $\SO_{4n+2}^\delta=\SO(V)$ be the $F$-quasisplit special even orthogonal group associated to the quadratic space $(V, q_V)$, and let $H^\delta=\GSpin_{4n+2}^\delta=\GSpin(V)$ be the  $\GSpin$ cover of $\SO_{4n+2}^\delta$. Then we have a projection map
\begin{equation}
\label{eq-pr}
\pr:\GSpin_{4n+2}^\delta \to \SO_{4n+2}^\delta	
\end{equation} 
and $\ker(\pr)\cong \GL_1$.

Let $P=M\ltimes U$ be the standard parabolic subgroup of $\GSpin_{4n+2}^\delta$ stabilizing $V^+$. Then $M\cong \GL_{2n}\times \GSpin_2^\delta$, where $\GSpin_2^\delta=\GSpin(V_0)$ is the $\GSpin$ cover of $\SO(V_0)$. 
Let $\sigma$ be an automorphic character of $\GSpin^\delta_2(\A)$. Let $\tau=\tau_1\boxplus \tau_2\boxplus \cdots \boxplus \tau_r$ be an irreducible unitary generic isobaric sum automorphic representation of $\GL_{2n}(\A)$ associated to distinct $\tau_1, \cdots, \tau_r$, such that $L(s, \tau_i, \wedge^2\otimes\omega_{\sigma}^{-1})$ has a pole at $s=1$ for each $1\le i\le r$.   
Following \cite[\S I.1.4]{MoglinWaldspurger1995}, we denote $X_M $ to be the group of continuous homomorphisms of $M(\A)/M^1$ into $\mathbb{C}^\times$, where $M^1:=\cap_{\chi\in \Hom(M, \mathbb{G}_m)}\ker |\chi|$. Since the parabolic subgroup $P$ is maximal, the $\mathbb{C}$-vector space $X_M$ is one-dimensional. For any $s\in \mathbb{C}$, we define $\lambda_s(m,h):=|\det(m)|^s$ for $(m,h)\in \GL_{2n}(\A)\times \GSpin_2^\delta(\A)$.  Then  $\lambda_s\in X_M$. Using the Iwasawa decomposition, we trivially extend $\lambda_s$ to a function on $H^\delta(\A)$. 

For any function 
\begin{equation}
\label{eq-phi-section}
\phi=\phi_{\tau\otimes\sigma}\in \mathcal{A}(M(F)U(\A)\bs H^\delta(\A))_{\tau\otimes \sigma},
\end{equation}
we set $\phi_{s}:=\lambda_s\cdot \phi$ and form the associated Eisenstein series
\begin{equation}
\label{eq-EisenSeries}
E(g, s, \phi_{\tau\otimes\sigma})=\sum_{\gamma\in P(F)\backslash H^\delta(F)} \phi_s(\gamma g), \quad g\in H^\delta(\A),
\end{equation}
which converges absolutely for $\mathrm{Re}(s)\gg 0$ and has meromorphic continuation to the entire complex plane \cite[\S IV]{MoglinWaldspurger1995}.
 
\begin{remark}
\label{remark-Eisenstein-series-central-char}
Note that for any $z\in \ker(\pr)(\A)$ and $g\in H^{\delta}(\A)$, we have 
\begin{equation*}
	E(zg, s, \phi_{\tau\otimes\sigma})=\omega_{\sigma}(z) E(g, s, \phi_{\tau\otimes\sigma}).
\end{equation*}
\end{remark}

\begin{proposition}[\cite{Yan2025GGP}]\label{prop-Eisenstein-pole}
Let $\sigma$ be an automorphic character of $\GSpin^\delta_2(\A)$. Let $\tau=\tau_1\boxplus \tau_2\boxplus \cdots \boxplus \tau_r$ be an irreducible unitary generic isobaric sum automorphic representation of $\GL_{2n}(\A)$ associated to distinct cuspidal automorphic representations $\tau_1, \cdots, \tau_r$, such that $L(s, \tau_i, \wedge^2\otimes\omega_{\sigma}^{-1})$ has a pole at $s=1$ for each $1\le i\le r$.  Then the Eisenstein series $E(g, s, \phi_{\tau\otimes\sigma})$ has a pole at $s=\frac{1}{2}$ of order $r$ if and only if
   $L(\frac{1}{2}, \sigma\times \tau\otimes \omega_\sigma^{-1})\neq 0$.
\end{proposition}

\begin{remark}
\label{remark-tau-in-generic-spectrum}
By the works of Asgari and Shahidi \cite{AsgariShahidi2006, AsgariShahidi2014}, the assumption that 	$\tau=\tau_1\boxplus \tau_2\boxplus \cdots \boxplus \tau_r$ is an irreducible unitary generic isobaric sum automorphic representation of $\GL_{2n}(\A)$ associated to distinct $\tau_1, \cdots, \tau_r$, such that $L(s, \tau_i, \wedge^2\otimes\omega_{\sigma}^{-1})$ has a pole at $s=1$ for each $1\le i\le r$, implies that $\tau$ is a functorial transfer of an irreducible globally generic cuspidal automorphic representation of   $\GSpin_{2n+1}(\A)$ where $\GSpin_{2n+1}$ is $F$-split.
\end{remark}

From now on, we assume that  $L(\frac{1}{2}, \sigma\times \tau\otimes \omega_\sigma^{-1})\neq 0$.
We denote by $\CE_{\tau\otimes \sigma}$ the automorphic representation of $H^\delta(\A)$ generated by the $r$-th iterated  residues at $s=\frac{1}{2}$ of $E(g, s, \phi_{\tau\otimes\sigma})$ for all $\phi_{\tau\otimes\sigma}\in \mathcal{A}(M(F)U(\A)\bs H^\delta(\A))_{\tau\otimes \sigma}$.
By Proposition~\ref{prop-Eisenstein-pole}, the residual representation $\CE_{\tau\otimes \sigma}$ is non-zero.

If $\CE_{\tau\otimes \sigma}$ has trivial central character, we can think of $\CE_{\tau\otimes \sigma}$ as a residual representation of $\SO(V)(\A)$. In this case, the square-integrability of $\CE_{\tau\otimes \sigma}$ follows from the $L^2$-criterion in \cite[Lemma I.4.11]{MoglinWaldspurger1995} (see also \cite[Theorem 6.1]{JiangLiuZhang2013}), and the irreducibility of $\CE_{\tau\otimes \sigma}$ follows from \cite[Theorem A]{Moeglin2011}. In general, regardless of whether the central character of $\CE_{\tau\otimes \sigma}$ is trivial or not, the residual representation $\CE_{\tau\otimes \sigma}$ is square-integrable and irreducible.

\subsection{Bessel coefficients}
\label{subsection-Bessel-coefficient}
Let $\ell$ be an integer such that $1\le \ell \le 2n$. Let $Q_\ell=L_{{\ell}}\ltimes N_{{\ell}}$ be the standard parabolic subgroup of $H^\delta$ stabilizing the following flag of isotropic subspaces:
\begin{equation*}
0\subset V_1^+ \subset V_2^+ \subset \cdots \subset 	V_{\ell}^+.
\end{equation*}
Then the Levi subgroup $L_{\ell}$ is isomorphic to $(\GL_1)^\ell \times \GSpin(W_\ell)$, where
\begin{equation}
\label{eq-W_ell}
W_{\ell}=(V_\ell^+ \oplus 	V_\ell^-)^\perp.
\end{equation}
The unipotent radical $N_{\ell}$ is given by 
\begin{equation*}
N_{\ell} = 	\left\{ u=\begin{pmatrix} z & y &x\\ &I_{4n+2-2\ell} &y^\prime\\ &&z^*\end{pmatrix}: z\in Z_{\ell} \right\}
\end{equation*}
where $Z_{\ell}$ is the standard maximal upper-triangular unipotent subgroup of $\GL_{\ell}=\GL(V_{\ell}^+)$. Throughout this paper, we use unipotent elements of subgroups in $\SO(V)$ to represent their lifts to $\GSpin(V)$ (see Remark~\ref{remark-unipotent-variety-isomorphism}).

Let $\psi:F\backslash \A\to \C^\times$ be a fixed non-trivial additive character. Let $w_0\in W_\ell$ be an anisotropic vector  with $q_V(w_0,w_0)$ in a given square class of $F^\times$, and let $\chi_{\ell,w_0}:N_{\ell}\to \mathbb{G}_a$ be the homomorphism defined by
\begin{equation}\label{eq-chi-ell-w_0-global}
\chi_{\ell,w_0}(u)=	\sum_{i=1}^{\ell-1}q_V(u\cdot e_{i+1}, e_{-i})+q_V(u\cdot w_0, e_{-\ell}).
\end{equation}
We now define a character  $\psi_{\ell, w_0}$ of $N_{\ell}(\A)$ by
\begin{equation}\label{eq-psi-ell-w_0-global}
\psi_{\ell, w_0}(u)=\psi\circ \chi_{\ell, w_0}(u).	
\end{equation}
Here we view $\chi_{\ell,w_0}$ also as a homomorphism from $N_{\ell}(\A)$ to $\A$. 
Note that $\psi_{\ell, w_0}$ is trivial on $N_{\ell}(F)$. The adjoint action of $L_{\ell}$ on $N_{\ell}$ induces an action of $\GSpin(W_{\ell})$ on the set of all such characters $\psi_{\ell, w_0}$. The stabilizer $H_{\ell,w_0}$ of $\psi_{\ell,w_0}$ in $\GSpin(W_{\ell})$ is 
\begin{equation*}
	H_{\ell,w_0}=\GSpin(w_0^\perp \cap W_{\ell}).
\end{equation*}

Let $\Pi$ be an automorphic representation of $H^\delta(\A)$ occurring in the discrete spectrum $L^2_{\text{disc}}(H^{\delta}(F)\backslash H^{\delta}(\A))$. 
For $f\in V_\Pi$, we define the Bessel coefficient (or Gelfand-Graev coefficient) of $f$, with respect to $ \psi_{\ell,w_0}$, 
by
\begin{equation}
\label{eq-Bessel-coef}
 f^{\psi_{\ell, w_0}}(g):=\int_{[N_\ell]} f(ug)\psi_{\ell,w_0}^{-1}(u)\ \mathrm{d}u, \quad g\in H^\delta(\A).
\end{equation}
Here, for an algebraic group $R$ defined over $F$, we use $[R]$ to denote the adelic quotient $R(F)\bs R(\A)$. 
This is a Fourier coefficient of $f$ attached to partition $[2\ell+1, 1^{4n+1-2\ell}]$. 
Note that for any $\gamma\in H_{\ell,w_0}(F)$ and $g\in H^\delta(\A)$, we have
\begin{equation*}
	 f^{\psi_{\ell, w_0}}(\gamma g)= f^{\psi_{\ell, w_0}}(g).
\end{equation*}
Following \cite[\S 3.1]{GinzburgRallisSoudry2011} and \cite[\S 1.2]{JiangLiuXuZhang2016}, we define the space 
$$
\TD_{\psi_{\ell,w_0}}(\Pi)=H_{\ell,w_0}(\A)-\Span \left\{f^{\psi_{\ell, w_0}}|_{H_{\ell,w_0}(\A)}\ | \ f\in V_\Pi\right\},
$$
which is a representation of $H_{\ell,w_0}(\A)$ under the right translation action.

Given an automorphic form $f$ of $H^{\delta}(\A)$ and a cusp form $\varphi^\prime$ of $H_{\ell,w_0}(\A)$ such that $f(z)\varphi^\prime(z)=1$ for all $z\in \ker(\pr)(\A)$, we define the Bessel period of $(f, \varphi^\prime)$, with respect to the character $\psi_{\ell, w_0}$, as
\begin{equation*}
\mathcal{P}(f, \varphi^\prime):=\int_{ \ker(\pr)(\A) H_{\ell,w_0}(F)\backslash H_{\ell,w_0}(\A)} f^{\psi_{\ell, w_0}}(g) \varphi^\prime(g)dg.	
\end{equation*}

When $1\le \ell<2n$, we make a more precise choice of the anisotropic vector $w_0$, given by 
\begin{equation}
\label{eq-w_0-global}
w_0=y_{\alpha}=e_{2n}-\frac{\alpha}{2}e_{-2n}, 	
\end{equation}
for $\alpha\in F^\times$, so that $q_V(w_0, w_0)=-\alpha$. 
For such an $\alpha$, we consider the three-dimensional quadratic form associated with 
\begin{equation}
\label{eq-J-delta-alpha}
J_{\delta, \alpha}=\left(\begin{smallmatrix} 1&&\\&\delta&\\&&\alpha\end{smallmatrix}\right), 
\end{equation}
which may be split or non-split over $F$, depending on the Hilbert symbol $(\delta, \alpha)$. 
In this case, we have $H_{\ell,w_0}=H_{\ell,\alpha}=\GSpin_{4n+1-2\ell}^{\delta, \alpha}$ which is an odd $\GSpin$ group associated to the quadratic space 
$
w_0^\perp \cap W_{\ell}
$
with quadratic form defined by
\begin{equation}
\label{eq-matrix-quad-form-on-4n-2l+1}
J=\left( \begin{smallmatrix}
   &&w_{2n-\ell-1}\\&J_{\delta,\alpha}&\\w_{2n-\ell-1}&&
  \end{smallmatrix}\right).
\end{equation}
To simplify notation, we write $H_{\ell, \alpha}=H_{\ell, y_\alpha}$, $\psi_{\ell, \alpha}=\psi_{\ell, y_\alpha}$, and $\TD_{\psi_{\ell, \alpha}}(\Pi)=\TD_{\psi_{\ell,\alpha}}(\Pi)$.
When $\ell=2n$, we take any non-zero vector $w_0\in V_0$.

When $\ell=0$, we take $N_{\ell}=N_{0}$ to be the trivial subgroup, $\psi_{\ell,w_0}$ the trivial character, and $H_{0,w_0}=\GSpin(w_0^\perp)$, with $w_0=e_{2n}-\frac{\alpha}{2}e_{-2n}$ chosen above. Then $H_{0, w_0}$ is a $\GSpin$ group associated to a quadratic space of dimension $4n+1$.

\subsection{The twisted automorphic descent}
\label{subsection-twisted-auto-descent}
In this section we restate our main theorem as below. 
\begin{theorem}\label{thm-Main}
 Let $\tau=\tau_1\boxplus \tau_2\boxplus \cdots \boxplus \tau_r$ be an irreducible unitary generic isobaric sum automorphic representation of $\GL_{2n}(\A)$ associated to distinct cuspidal automorphic representations $\tau_1, \cdots, \tau_r$, such that $L(s, \tau_i, \wedge^2\otimes\omega^{-1})$ has a pole at $s=1$ for each $1\le i\le r$, where $\omega$ is a Hecke character. 
 If there is an automorphic character $\sigma$ of $\GSpin^\delta_2(\A)$ such that $\omega_{\sigma}=\omega$ and 
 $L(\frac{1}{2}, \sigma\times \tau\otimes \omega_\sigma^{-1})\neq 0$, then the following assertions hold.
\begin{enumerate}
\item The representation $\TD_{\psi_{\ell,\alpha}}(\CE_{\tau\otimes\sigma})$ of $\GSpin^{\delta,\alpha}_{4n+1-2\ell}(\A)$ is zero
for all $n<\ell\leq 2n$.
\item For any square class $\alpha$ in $F^\times$, the representation $\TD_{\psi_{n,\alpha}}(\CE_{\tau\otimes\sigma})$ of $\GSpin^{\delta,\alpha}_{2n+1}(\A)$ is cuspidal automorphic.
\item There exists a square class $\alpha$ in $F^\times$ such that the representation $\TD_{\psi_{n,\alpha}}(\CE_{\tau\otimes\sigma})$ of $\GSpin^{\delta,\alpha}_{2n+1}(\A)$ is non-zero, and in this case
    $$
    \TD_{\psi_{n,\alpha}}(\CE_{\tau\otimes\sigma})=\pi_1\oplus\pi_2\oplus\cdots\oplus\pi_k\oplus\cdots,
    $$
    where $\pi_i$ are irreducible cuspidal automorphic representations of $\GSpin^{\delta,\alpha}_{2n+1}(\A)$, which are nearly equivalent, but
    are not globally equivalent, i.e. the decomposition is multiplicity-free.
\item When $\TD_{\psi_{n,\alpha}}(\CE_{\tau\otimes\sigma})$ is non-zero, any direct summand $\pi$ of $\TD_{\psi_{n,\alpha}}(\CE_{\tau\otimes\sigma})$ has a weak Langlands functorial transfer to $\tau$ in the sense that the Satake parameter of the local unramified component $\tau_v$ of $\tau$ is the local functorial transfer of that of the local unramified component $\pi_v$ of $\pi$ for almost all unramified local places $v$ of $F$.
\item When $\TD_{\psi_{n,\alpha}}(\CE_{\tau\otimes\sigma})$ is non-zero, every irreducible direct summand of $\TD_{\psi_{n,\alpha}}(\CE_{\tau\otimes\sigma})$ has a non-zero Fourier coefficient attached to the partition $[2n-1,1^2]$.  
\item The residual representation $\CE_{\tau\otimes\sigma}$ has a non-zero Fourier coefficient attached to the partition $[2n+1,2n-1,1^2]$.  
\end{enumerate}
\end{theorem}

\section{Vanishing property and cuspidality}
\label{section-vanishing-and-cuspidality}

\subsection{Twisted Jacquet modules}

Let $v$ be a finite place of $F$, and let $k=F_v$. For an algebraic group $\mathbf{G}$ we denote its group of $k$-points by $G(k)$ or simply by $G$. Let $(V, q_V)$ be a non-degenerate quadratic space of dimension $4n+2$ over
$k$, with a polar decomposition
\begin{equation*}V=V^+\oplus V_0 \oplus V^-,\end{equation*}
where $V^{\pm}$ are maximal isotropic subspaces of $V$ in duality with respect to $q_V$, and $V_0$ is anisotropic. We will consider two cases:
\begin{itemize}
 \item Case (1): $\dim V_0=0$, if $\GSpin(V)$ is split;
 \item Case (2):  $\dim V_0=2$, if $\GSpin(V)$ is quasi-split, non-split.
\end{itemize}
We denote the Witt index of $V$ by $\wt m$, so that
\begin{equation*}
\wt m =	\begin{cases}
	2n+1 &\text{ Case (1)};\\
	2n	&\text{ Case (2)}.
	\end{cases}
\end{equation*}
We fix a basis $\{e_1, \cdots, e_{\wt m}\}$ of $V^+$ and a basis $\{e_{-1}, e_{-2}, \cdots, e_{-\wt m}\}$ of $V^-$ so that $q_V(e_i, e_{-j})=\delta_{i,j}$ for all $1\le i, j\le \wt m$. In Case (2), we take a basis $\{e_0^{(1)}, e_0^{(2)}\}$ for $V_0$ such that $q_V(e_0^{(1)}, e_0^{(1)})=1$ and $q_V(e_0^{(2)}, e_0^{(2)})=\delta$ for some $\delta\in k^\times$.  
We put the basis in the following order:
\begin{equation*}
e_1, e_2, \cdots, e_{\tilde{m}}, e_0^{(1)}, e_0^{(2)},  	e_{-\tilde{m}}, \cdots, e_{-2}, e_{-1}, \quad \textrm{ if } \dim_F V_0=2, 
\end{equation*}
\begin{equation*}
e_1, e_2, \cdots, e_{\tilde{m}},	e_{-\tilde{m}}, \cdots, e_{-2}, e_{-1}, \quad \textrm{ if } \dim_F V_0=0.
\end{equation*}

Let $H=\GSpin(V)$ be the $\GSpin$ group associated to $V$ over
$k$. As in \S\ref{subsection-Bessel-coefficient}, for an integer $\ell$ such that $1\le \ell \le 2n$, we let $Q_\ell=L_{{\ell}}\ltimes N_{{\ell}}$ be the standard parabolic subgroup of $H$ stabilizing the following flag of isotropic subspaces:
\begin{equation}\label{eq-parobolic-Ql-flag}
0\subset V_1^+ \subset V_2^+ \subset \cdots \subset 	V_{\ell}^+.
\end{equation}
Then the Levi subgroup $L_{\ell}$ is isomorphic to $(\GL_1)^\ell \times \GSpin(W_\ell)$, where $W_{\ell}$ is given in \eqref{eq-W_ell}. 
The unipotent radical $N_{\ell}$ is given by 
\begin{equation*}
N_{\ell} = 	\left\{ u=\begin{pmatrix} z & y &x\\ &I_{4n+2-2\ell} &y^\prime\\ &&z^*\end{pmatrix}: z\in Z_{\ell} \right\}
\end{equation*}
where $Z_{\ell}$ is the standard maximal upper-triangular unipotent subgroup of $\GL_{\ell}=\GL(V_{\ell}^+)$. 
For $\alpha\in k^\times$, we take the following anisotropic vector $w_0\in V$:
\begin{equation}\label{eq-w_0-local}
w_0=\begin{cases}
     y_\alpha=e_{2n}-\frac{\alpha}{2}e_{-2n}, & \text{if $\ell < 2n$};\\
     &\\
     \text{$\alpha e_0^{(1)}$ or $\alpha e_0^{(2)}$ in $V_0$}, & \text{if $\ell=2n$ and $V_0\neq 0$}.
    \end{cases}
\end{equation}

\begin{remark}\label{remark-w_0-local}
Note that by a change of basis if necessary, every anisotropic vector in $V$ is of the form $c_1 e_i+c_2 e_{-i}$ or $c e_0^{(j)}$ (this happens only in Case (2)) for some $c_1, c_2, c\in k^\times$. Thus the above choice of $w_0$ essentially represents any anisotropic vector in $V$.  	
\end{remark}

Let $\psi$ be a fixed non-trivial additive character of $k$. Similar to \eqref{eq-chi-ell-w_0-global}, we define a character $\psi_{\ell, w_0}=\psi\circ \chi_{\ell, w_0}$ on $N_{\ell}(k)$. Let $H_{\ell,w_0}=\GSpin(w_0^\perp\cap W_{\ell})$ be the stabilizer of $\psi_{\ell, w_0}$ in $\GSpin(W_{\ell})$. Denote
\begin{equation}
\label{eq-R_l,w0}
R_{\ell, w_0}=H_{\ell, w_0}\ltimes N_{\ell}.	
\end{equation}
This is a Bessel subgroup of $H$.

Let $(\Pi,V_\Pi)$ be a smooth representation of $H(k)$. Let $J_{\psi_{\ell,w_0}}(\Pi)$ be the twisted Jacquet module of $\Pi$, with respect to $N_{\ell}(k)$ and the character $\psi_{\ell, w_0}$. The space of $J_{\psi_{\ell,w_0}}(\Pi)$ is given by
\begin{equation*}
	J_{\psi_{\ell,w_0}}(V_\Pi)=V_\Pi/\Span\{\Pi(u)\xi-\psi_{\ell, w_0}(u)\xi\ |\ u\in N_\ell(k),\  \xi\in V_\Pi\}.
\end{equation*}
This is a module over $H_{\ell, w_0}(k)$. To ease notation, we will write $\chi_{\ell,\alpha}=\chi_{\ell,y_\alpha}$, $\psi_{\ell,\alpha}=\psi_{\ell,y_\alpha}$, $H_{\ell,\alpha}=H_{\ell,y_\alpha}$, $J_{\psi_{\ell,\alpha}}=J_{\psi_{\ell,y_\alpha}}$ if $w_0=y_\alpha$. Note that for $0\le \ell< 2n$, the quadratic form associated to the group $H_{\ell,w_0}=\GSpin(w_0^\perp\cap W_{\ell})$ is given by
$$J=\left(\begin{smallmatrix}
   &&w_{2n-\ell-1}\\&J_{\delta,\alpha}&\\w_{2n-\ell-1}&&
  \end{smallmatrix}\right),
$$
and thus we have the following two cases:
\begin{itemize}
\item If $J_{\delta, \alpha}$ is non-split over $k$, then the group 	$H_{\ell,w_0}=\GSpin(w_0^\perp\cap W_{\ell})$ is non-split, and  $\mathrm{Witt}(k\cdot y_{-\alpha}+V_0)=0$.
\item If $J_{\delta, \alpha}$ is split over $k$, then the group 	$H_{\ell,w_0}=\GSpin(w_0^\perp\cap W_{\ell})$ is split over $k$. In this case, $\mathrm{Witt}(k\cdot y_{-\alpha}+V_0)=1$ if $V_0\neq 0$ (i.e. $-\delta\notin {k^\times}^2$), and
 $\mathrm{Witt}(k\cdot y_{-\alpha}+V_0)=0$ if $V_0=0$.
\end{itemize}
If $\ell=2n$, then $H_{2n,w_0}=\GSpin(w_0^\perp \cap V_0)$ is either $\GSpin(k\cdot e_0^{(2)}$) or $\GSpin(k\cdot e_0^{(1)}$).

For each $1\le j \le 2n$, denote $V_j^+=\Span\{e_1, \cdots, e_j \}$. Let $P_j=M_{j}\ltimes U_{j}$ be the standard maximal parabolic subgroup of $H$ which stabilizes $V_{j}^+$. Then the Levi subgroup $M_{j}$ is isomorphic to $\GL(V_j^+) \times \GSpin(W_j)$. 

We now describe the double coset $P_j(k)\backslash H(k)/P_\ell(k)$.

First, we consider the case when $\widetilde{m}=2n$, so that $\dim V_0=2$ (Case (2)). In this case, define
\begin{equation*}
\omega_{b,\Orth(V)}= \left(\begin{smallmatrix}I_{2n} & & &\\&1&&\\&&-1&\\&&&I_{2n}  \end{smallmatrix} \right) \in \Orth(V)\setminus \SO(V),	
\end{equation*}
and 
\begin{equation*}
\omega_{b,\Orth(V_0)}= \left(\begin{smallmatrix}1 & \\ & -1  \end{smallmatrix} \right) \in \Orth(V_0)\setminus \SO(V_0),	
\end{equation*}
Choose an element $\omega_b\in \GPin(V)$ and an element $\omega_{b}^0\in \GPin(V_0)$ such that 
\begin{equation*}
\pr(\omega_b)=\omega_{b,\Orth(V)}, \qquad \pr(\omega_b^0)=\omega_{b,\Orth(V_0)}.	
\end{equation*}
For later reference, for any $j\le \widetilde{m}$, we also denote
\begin{equation*}
\omega_{b,j,\Orth(W_j)}=	\left(\begin{smallmatrix}I_{2n-j} & & &\\&1&&\\&&-1&\\&&&I_{2n-j}  \end{smallmatrix} \right)\in \Orth(W_j)\setminus \SO(W_j),
\end{equation*}
and choose $\omega_{b,j}\in \GPin(W_j)$ such that 
\begin{equation}
\label{eq-omega_b,j-1}
\pr(\omega_{b,j})=\omega_{b,j,\Orth(W_j)}.	
\end{equation}

Next, we consider the case when $\widetilde{m}=2n+1$, so that $\dim V_0=0$ (Case (1)). In this case, set
\begin{equation*}
\omega_{b,\Orth(V)}= \left(\begin{smallmatrix}I_{2n} & & &\\&&1&\\&1&&\\&&&I_{2n}  \end{smallmatrix} \right)\in \Orth(V)\setminus \SO(V),
\end{equation*}
and choose $\omega_b\in \GPin(V)\setminus \GSpin(V)$ such that 
\begin{equation}
\pr(\omega_b)=\omega_{b,\Orth(V)}.	
\end{equation}
For any $j\le \widetilde{m}$, set 
\begin{equation*}
\omega_{b,j,\Orth(W_j)}=	 \left(\begin{smallmatrix}I_{2n-j} & & &\\&&1&\\&1&&\\&&&I_{2n-j}  \end{smallmatrix} \right)\in \Orth(W_j)\setminus \SO(W_j),
\end{equation*}
and choose $\omega_{b,j}\in \GPin(W_j)$ such that 
\begin{equation}
\label{eq-omega_b,j-2}
\pr(\omega_{b,j})=\omega_{b,j,\Orth(W_j)}.	
\end{equation}

We have the following description on the double cosets $P_j(k) \backslash H(k)/P_{\ell}(k)$. 

\begin{lemma}
\label{lemma-double-coset-decomposition}
\begin{enumerate}
\item[(a)] Suppose that $\wt m=2n$. The double cosets representatives for $P_j(k)\backslash H(k)/P_\ell(k)$ are indexed by the pairs of non-negative integers $(\alpha, \beta)$ in
\begin{equation}
\label{eq-double-coset-index-set-Ekl}
\mathcal{E}_{j,\ell}=\{ (\alpha, \beta): 0\le \alpha \le \beta \le j \text{ and } j\le \ell +\beta-\alpha \le \tilde{m} \}.	
\end{equation}
For each $(\alpha, \beta)\in \mathcal{E}_{j,\ell}$, choose an element $\epsilon_{\alpha,\beta}\in H(k)$ such that
\begin{equation}
\label{eq-double-coset-epsilon-alpha-beta}
	 \pr(\epsilon_{\alpha,\beta}) 	 =  \omega_{b,\Orth(V)}^{j-\beta}   a_{\alpha, \beta, \Orth(V)} 
\end{equation}
where   
\begin{equation}
\label{eq-a-alpha-beta}
 a_{\alpha, \beta, \Orth(V)} = 
\begin{pmatrix}
I_{\alpha} &&&&&&&\\
&0&0&I_{\beta-\alpha} &0&0&0&0&\\
&0&0&0&0&0&I_{j-\beta}&0&\\
&I_{\ell+\beta-\alpha-j} &0&0&0&0&0&0&\\
&0&0&0&I_{4n+2-2(\ell+\beta-\alpha)} &0&0&0&\\
&0&0&0&0&0&0&I_{\ell+\beta-\alpha-j}&\\
&0&I_{j-\beta}&0&0&0 &0&0&\\
&0&0 &0&0&I_{\beta-\alpha}&0&0&\\
&&&&&&&&I_\alpha 
\end{pmatrix}.
\end{equation}
Then $\{\epsilon_{\alpha, \beta}: (\alpha, \beta)\in \mathcal{E}_{j,\ell} \}$ is a set of representatives for $P_j(k)\backslash H(k)/P_\ell(k)$.

\item[(b)] Suppose that $\wt m =2n+1$. The double cosets in $P_j(k)\backslash H(k)/P_\ell(k)$ are described as follows.
\begin{itemize}
\item[(1)] For each $(\alpha,\beta)\in \mathcal{E}_{j,\ell}$ with $\ell+\beta-\alpha<\wt m$, there is one double coset representative $\epsilon_{\alpha, \beta}$ given by\eqref{eq-double-coset-epsilon-alpha-beta}.
\item[(2)] For each $(\alpha,\beta)\in \mathcal{E}_{j,\ell}$ with $\ell+\beta-\alpha=\wt m$, there are two double coset representatives $\epsilon_{\alpha, \beta}$ and $\tilde{\epsilon}_{\alpha, \beta}$, where $ \tilde{\epsilon}_{\alpha, \beta} =\omega_b  {\epsilon}_{\alpha, \beta}  \omega_b^{-1}$.
\end{itemize}
\end{enumerate}
\end{lemma}

\begin{proof}
This is \cite[Lemma 3.5]{Yan2025GGP} and \cite[Lemma 3.6]{Yan2025GGP}.
\end{proof}

To simplify notation in this section, we will simply write $\GL_j=\GL_j(k)$, $H=H(k)$, etc. 
Let $\tau$ and $\sigma$ be smooth representations of $\GL_j$ and $\GSpin(W_j)$ respectively. 
We denote temporarily 
\begin{equation*}
\pi=	\Ind_{P_j}^H \tau\otimes \sigma,
\end{equation*}
and we  would like to study the Jacquet module $J_{\psi_{\ell, w_0}}(\pi)$.

Since $J_{\psi_{\ell, w_0}}(\pi)=J_{\psi_{\ell, w_0}}(\Res_{P_\ell}\pi)$, we consider the restriction of $\pi$ to $P_\ell$ first. 
By Bruhat theory, this representation admits a finite $P_{\ell}$-filtration, whose subquotients are parametrized by elements of $P_j(k) \backslash H(k)/P_{\ell}(k)$. We choose the representatives $w$ for $P_j(k) \backslash H(k)/P_{\ell}(k)$ as in Lemma~\ref{lemma-double-coset-decomposition}. 
If $\ell=0$, we will skip the following discussion and just consider $P_k(k)\backslash H(k) / H_{0,w_0}(k)$ where $H_{0,w_0}(k)$ is the subgroup of $H(k)$ consisting of elements which fix $w_0$.  For a general $1\le \ell\le 2n$, the subquotient corresponding to the representative $w$ is
\begin{equation*}
\rho_{w}=	\ind_{P_{\ell, j}^{(w)}}^{P_\ell} \pi_{\ell, j}^{(w)}(\tau, \sigma),
\end{equation*}
where 
``$\ind$" denotes compact induction, $P_{\ell, j}^{(w)}=P_{\ell}\cap  w^{-1}P_jw$, and for $g\in P_{\ell, j}^{(w)}$, we have
\begin{equation*}
	\pi_{\ell, j}^{(w)}(\tau, \sigma)(g)=\delta_{P_{\ell}}^{\frac{1}{2}}(g)\delta_{P_{\ell, j}^{(w)}}^{-\frac{1}{2}}(g)\delta_{P_j}^{\frac{1}{2}}(wgw^{-1}) (\tau\otimes\sigma)(wgw^{-1}),
\end{equation*}
where $\delta_{P_{\ell}}, \delta_{P_{\ell, j}^{(w)}}, \delta_{P_j}$ are the modular functions of $P_{\ell}, P_{\ell, j}^{(w)}, P_j$ respectively.  

Now our goal is to analyze $J_{\psi_{\ell, w_0}}(\rho_w)$ for each representative $w=\epsilon_{\alpha, \beta}$ (or $w= \tilde{\epsilon}_{\alpha,\beta}$), as an $H_{\ell, w_0}$-module.
Recall that $R_{\ell, w_0}$ is the Bessel subgroup of $H$ given in \eqref{eq-R_l,w0}.  
Since  $J_{\psi_{\ell, w_0}}(\rho_w)=J_{\psi_{\ell, w_0}}(\Res_{R_{\ell, w_0}}\rho_w)$, we consider $\Res_{R_{\ell, w_0}}\rho_w$. Again, by Bruhat theory, we need to consider the double coset space $P_{\ell, j}^{(w)}(k)\backslash P_{\ell}(k)/ R_{\ell,w_0}(k)$. By \cite[Lemma 3.8]{Yan2025GGP}, the representatives for this double coset space can be given by $\{\eta_{\epsilon, \gamma}\}$ with
\begin{equation}
\label{eq-eta}
\pr(\eta_{\epsilon,\gamma}) = \begin{pmatrix}
 \epsilon & & \\ & \gamma &\\ && \epsilon^*	
 \end{pmatrix} ,
 \end{equation}
where $\epsilon$ runs through a set of representatives for the quotient of the Weyl groups 
\begin{equation*}
	W_{\GL_{\alpha}}\times W_{\GL_{\ell+\beta-\alpha-j}}\times W_{\GL_{j-\beta}}\backslash W_{\GL_{\ell}},
\end{equation*}
and $\gamma$ runs through a set of representatives for $P_{w,\SO_{4n+2}}^\prime\backslash \SO(W_{\ell}) / \mathrm{Stab}_{L_{\ell,\SO_{4n+2}}}(\psi_{\ell, w_0})$, where 
\begin{equation*}
P_{w, \SO_{4n+2}}^\prime:= \SO(W_{\ell})\cap \pr(w^{-1})P_{j,\SO_{4n+2}} \pr(w)  ,
\end{equation*}
${L_{\ell,\SO_{4n+2}}}\cong (\GL_1)^{\ell}\times \SO_{4n+2-2\ell}$  is the Levi subgroup of the standard parabolic subgroup of $\SO_{4n+2}$ stabilizing \eqref{eq-parobolic-Ql-flag}, $P_{j,\SO_{4n+2}}$ is the standard parabolic subgroup of $\SO_{4n+2}$ stabilizing $V_j^+$. 

\begin{remark}
\label{remark-parabolic-subgroup}
When $\wt m=2n$, $P_{w, \SO_{4n+2}}^\prime$ is a parabolic subgroup of $\SO(W_{\ell})$ preserving the  $(\beta-\alpha)$-dimensional totally isotropic subspace $V_{\ell, \beta-\alpha}^+$ of $W_{\ell}$, where
\begin{equation}
\label{eq-subspace-V-lt}
V_{\ell, t}^{\pm}=\Span\{e_{\pm (\ell+ 1)}, \cdots, e_{\pm (\ell+t)} \}
\end{equation}
for $1\le t\le \wt{m}-\ell$. Now consider $\wt m=2n+1$. When $\ell+\beta-\alpha<\wt m$, $P_{w, \SO_{4n+2}}^\prime$ is the parabolic subgroup of $\SO(W_{\ell})$ preserving the $(\beta-\alpha)$-dimensional totally isotropic subspace $V_{\ell, \beta-\alpha}^+$ of $W_{\ell}$. When $\ell+\beta-\alpha=\wt m$ and $w=\epsilon_{\alpha,\beta}$, $P_{w, \SO_{4n+2}}^\prime$ is the parabolic subgroup of $\SO(W_{\ell})$ preserving the $(\wt m-\ell)$-dimensional totally isotropic subspace $V_{\ell, \wt m-\ell}^+$ of $W_{\ell}$. When $\ell+\beta-\alpha=\wt m$ and $w=\tilde{\epsilon}_{\alpha,\beta}$, $P_{w, \SO_{4n+2}}^\prime$ is the parabolic subgroup of $\SO(W_{\ell})$ preserving the $(\wt m-\ell)$-dimensional subspace $\omega_b V_{\ell,\wt m-\ell}^+$ of $W_{\ell}$.
\end{remark}

\begin{proposition}
\label{prop-JacquetModuleVanishing-alpha>0}
Let $w$ be any representative for $P_j(k)\backslash H(k)/P_\ell(k)$ corresponding to $(\alpha, \beta)$ given in Lemma~\ref{lemma-double-coset-decomposition} with $\alpha>0$. Then 
\begin{equation*}
J_{\psi_{\ell, w_0}}(\rho_w)=0.	
\end{equation*}
\end{proposition}

\begin{proof}
The representation $\Res_{R_{\ell, w_0}}\rho_w$ has a finite $R_{\ell, w_0}$-filtration, whose subquotients are given by
\begin{equation}
\label{eq-rho_w,eta}
\rho_{w,\eta}=\ind_{R_{\ell,w_0}^{\eta}}^{R_{\ell,w_0}} \pi_{\ell,j}^{(w,\eta)}(\tau,\sigma),	
\end{equation}
where $\eta$ runs over the representatives given in \eqref{eq-eta}, the subgroup $R_{\ell,w_0}^{\eta}$ is given by
 \begin{equation*}
 	R_{\ell,w_0}^{\eta}:=R_{\ell,w_0}\cap \eta^{-1} P_{\ell,j}^{(w)} \eta,
 \end{equation*}
and the representation $\pi_{\ell,j}^{(w,\eta)}(\tau,\sigma)$ takes $g\in R_{\ell,w_0}^{\eta}$ to 
\begin{equation}
\label{eq-ind-action-R^eta}
\delta_{R_{\ell,w_0}}^\frac{1}{2}(g) \delta_{R_{\ell,w_0}^{\eta}}^{-\frac{1}{2}}(g) \delta_{P_{\ell,j}^{(w)}}^{\frac{1}{2}}(\eta g \eta^{-1})	\pi_{\ell, j}^{(w)}(\tau, \sigma)(\eta g \eta^{-1})
\end{equation}
where $\delta_{R_{\ell,w_0}}, \delta_{R_{\ell,w_0}^{\eta}}, \delta_{P_{\ell,j}^{(w)}}$ are modular functions of $R_{\ell,w_0}, R_{\ell,w_0}^{\eta}, P_{\ell,j}^{(w)}$.
Since $R_{\ell, w_0}=H_{\ell,w_0}\ltimes N_{\ell}$, we have
\begin{equation*}
	R_{\ell,w_0}^{\eta}=(H_{\ell,w_0}\cap \eta^{-1} M_{\ell,j}^{(w)} \eta) \cdot (N_{\ell}\cap \eta^{-1} P_{\ell,j}^{(w)} \eta) 
\end{equation*}
where $P_{\ell,j}^{(w)}=M_{\ell,j}^{(w)}U_{\ell,j}^{(w)}$, $M_{\ell, j}^{(w)}=M_{\ell}\cap  w^{-1}P_jw$, $U_{\ell, j}^{(w)}=U_{\ell}\cap  w^{-1}P_jw$. Note that $H_{\ell,w_0}\cap \eta^{-1} M_{\ell,j}^{(w)} \eta$ normalizes $N_{\ell}\cap \eta^{-1} P_{\ell,j}^{(w)} \eta$. A function $f$ in the space of $\rho_{w,\eta}$ is supported in a set of the form 
$(H_{\ell,w_0}\cap \eta^{-1} M_{\ell,j}^{(w)} \eta)N_{\ell} \Omega_2$, where $\Omega_2$ is a compact subset of $R_{\ell,w_0}$, and $f$ is determined by its restriction to the set $N_{\ell} \Omega_2$. The restriction of $f$ to the set $N_{\ell} \Omega_2$ is supported in a set of the form $(N_{\ell}\cap \eta^{-1} P_{\ell,j}^{(w)}\eta)\Omega_1 \Omega_2$, where $\Omega_1$ is a compact subset of $N_{\ell}$, and this restriction of $f$ is determined by its value on $\Omega_1 \Omega_2$. 

We consider the following integral
\begin{equation}
\label{eq-prop-JacquetModuleVanishing-alpha>0-eq1}
\int_{N_{\ell}^{(i)}} \psi_{\ell,w_0}^{-1}(u)\rho(u)f du,
\end{equation}
where $\rho(u)$ denotes the right-translation action by $u$, and $\{N_{\ell}^{(i)}\}_{i=0}^\infty$ is an ascending sequence of compact open subgroups in $N_{\ell}$ whose union is equal to $N_{\ell}$. We need to prove that the integral \eqref{eq-prop-JacquetModuleVanishing-alpha>0-eq1} vanishes for $i$ large enough. It suffices to prove that, for $i$ large enough, we have
\begin{equation}
\label{eq-prop-JacquetModuleVanishing-alpha>0-eq2}
\int_{N_{\ell}^{(i)}} \psi_{\ell,w_0}^{-1}(u)f( n x u) du=0, \quad \forall n\in N_{\ell}, x\in \Omega_2. 
\end{equation}
We can find a compact open subgroup $H^0$ of $H_{\ell,w_0}$ containing the identity element, such that $x N_{\ell}^{(i)}x^{-1}=N_{\ell}^{(i)}$, for all $x\in H^0$ and $i$ large enough. Since $\Omega_2\subset \Omega_2 H^0$, we may replace $\Omega_2$ by $\Omega_2 H^0$.  Thus we may write $\Omega_2$ as a finite union of cosets $x_j H^0$, for $1\le j \le N$, where $x_j\in \Omega_2$. The number of these cosets depends on $f$ only. Moreover, we can make a change of variable, $u\mapsto a^{-1}u a$ for $a\in H^0\subset H_{\ell, w_0}$, and since $H_{\ell, w_0}$ stabilizes the character $\psi_{\ell, w_0}$, we conclude that it suffices to prove that for $i^\prime$ large enough and $1\le j \le N$, we have
\begin{equation*}
\int_{N_{\ell}^{(i^\prime)}} \psi_{\ell,w_0}^{-1}(u)f( n u x_j  a) du=0, \quad \forall n\in N_{\ell}, a\in H^0. 
\end{equation*}
We may assume that, given $i^\prime$ large enough, there is $i$, large enough, such that $x_j^{-1} N_{\ell}^{(i^\prime)} x_j \supset N_{\ell}^{(i)}$. We write $x_j N_{\ell}^{(i^\prime)} x_j^{-1}$ as a finite union of cosets $y N_{\ell}^{(i)}$,  with representatives $y\in N_{\ell}$. Thus, it is enough to show that, for $i$ large enough and $1\le j \le N$, we have 
\begin{equation*}
\int_{N_{\ell}^{(i)}} \psi_{\ell,w_0}^{-1}(u)f( n u x_j a) du=0, \quad \forall n\in N_{\ell}, a\in H^0. 
\end{equation*}
Since $x_j H^0\subset \Omega_2$, it is enough to prove that for $i$ large enough, 
\begin{equation}
\label{eq-prop-JacquetModuleVanishing-alpha>0-eq3}
\int_{N_{\ell}^{(i)}} \psi_{\ell,w_0}^{-1}(u)f( n u x) du=0, \quad \forall n\in N_{\ell}, x\in \Omega_2. 
\end{equation}
From the above discussion, it suffices to consider $f$ whose restriction to the set $N_{\ell}\Omega_2$ is supported in $(N_{\ell}\cap \eta^{-1} P_{\ell,j}^{(w)}\eta)\Omega_1 \Omega_2$. Thus, it is enough to take, in \eqref{eq-prop-JacquetModuleVanishing-alpha>0-eq3}, $nu\in (N_{\ell}\cap \eta^{-1} P_{\ell,j}^{(w)}\eta)\Omega_1$. We may take $i$ to be large enough, so that $\Omega_1 N_{\ell}^{(i)}=N_{\ell}^{(i)}$. Thus, it is enough to show
\begin{equation*}
\int_{N_{\ell}^{(i)}} \psi_{\ell,w_0}^{-1}(u)f( n u x) du=0, \quad \forall n\in N_{\ell}^{(i)}, x\in \Omega_2. 
\end{equation*}
By a change of variable in $u$, it is enough to prove that 
\begin{equation}
\label{eq-prop-JacquetModuleVanishing-alpha>0-eq4}
\int_{N_{\ell}^{(i)}} \psi_{\ell,w_0}^{-1}(u)f(  u x) du=0, \quad \forall  x\in \Omega_2. 
\end{equation}
It is enough to find a subgroup $J\subset N_{\ell}^{(i)}\cap \eta^{-1}P_{\ell, j}^{(w)}\eta$, such that, for $i$ large enough, we have
\begin{equation}
\label{eq-prop-JacquetModuleVanishing-alpha>0-eq5}
\psi_{\ell,w_0}|_J\not=1, \quad \text{ and }  	\pi_{\ell,j}^{(w,\eta)}(\tau,\sigma)|_J=1.
\end{equation}
Since $\alpha>0$, in the special orthogonal group case, by \cite[Proposition 5.1]{GinzburgRallisSoudry2011}, there exists a unipotent subgroup $S_0$ belong to $\pr(\eta^{-1} U_{\ell, j}^{(w)} \eta) $ satisfying the analogous properties in \eqref{eq-prop-JacquetModuleVanishing-alpha>0-eq5}. Since unipotent subgroups for special orthogonal groups and $\GSpin$ groups are isomorphic, the existence of a unipotent subgroup $S\subset \eta^{-1} U_{\ell, j}^{(w)} \eta$ satisfying the properties in \eqref{eq-prop-JacquetModuleVanishing-alpha>0-eq5} remains true. Now we take $J=S\cap N_{\ell}^{(i)}$. Then $J\subset N_{\ell}^{(i)}\cap \eta^{-1}P_{\ell, j}^{(w)}\eta$ and the properties in \eqref{eq-prop-JacquetModuleVanishing-alpha>0-eq4} hold. We have completed the proof. 
\end{proof}

Next, we consider $J_{\psi_{\ell, w_0}}(\rho_w)$ for $w$ corresponding to $(\alpha, \beta)$ with $\alpha=0$. Thus, we consider
\begin{equation}
\label{eq-local-Jacquet-module-inequalities}
0\le \beta\le j, \quad j\le \ell+\beta\le \wt m.	
\end{equation}
Let $t=j-\beta$. 
 Let $\eta=\eta_{\epsilon, \gamma}$ be as in \eqref{eq-eta}. By \cite[Lemma 5.1]{GinzburgRallisSoudry2011}, we may take 
\begin{equation}
\label{eq-epsilon}
\epsilon=\begin{pmatrix}
& I_{\ell-t}\\
I_{t} &	
\end{pmatrix},
\end{equation}
and with this choice of $\epsilon$ we re-denote $\eta_{t,\gamma} =\eta_{\epsilon, \gamma}$ such that
\begin{equation}
\label{eq-eta-2}
\pr(\eta_{t,\gamma}) =\pr(\eta_{\epsilon, \gamma}) =  \left( \begin{smallmatrix}
 & I_{\ell-t} & & &\\ I_t &&&&\\ & & \gamma &&\\ && && I_t \\ &&&I_{\ell-t} &	
 \end{smallmatrix} \right).
 \end{equation}
Note that $\gamma$ runs through the double coset $P_{w,\SO_{4n+2}}^\prime\backslash \SO(W_{\ell}) / \mathrm{Stab}_{L_{\ell,\SO_{4n+2}}}(\psi_{\ell, w_0})$, whose cardinality is given below. Note that we are only interested in the situation where $\alpha=0$.

\begin{lemma}
\label{lemma-double-coset2}
\begin{enumerate}
\item Suppose $\ell+\beta<\tilde{m}$. Then 
\begin{equation*}
	| P_{w,\SO_{4n+2}}^\prime\backslash \SO(W_{\ell}) / \mathrm{Stab}_{L_{\ell,\SO_{4n+2}}}(\psi_{\ell, w_0}) | =2.
\end{equation*}

\item Suppose $\ell+\beta=\tilde{m}$. Then we have the following.

\begin{enumerate}

\item 	If $\SO(W_{\ell})$ is a $k$-quasi-split non-split even special orthogonal group (with $\dim V_0=2$), then
\begin{equation*}
	| P_{w,\SO_{4n+2}}^\prime\backslash \SO(W_{\ell}) / \mathrm{Stab}_{L_{\ell,\SO_{4n+2}}}(\psi_{\ell, w_0}) | = \begin{cases}
2  & \text{ if } \mathrm{Witt}(w_0^\perp \cap W_{\ell})=\tilde{m}-\ell,  \\
 1	 & \text{ if } \mathrm{Witt}(w_0^\perp \cap W_{\ell})=\tilde{m}-\ell-1.
 \end{cases}
\end{equation*}

\item If $\SO(W_{\ell})$ is a $k$-split even special orthogonal group, then 
\begin{equation*}
	| P_{w,\SO_{4n+2}}^\prime\backslash \SO(W_{\ell}) / \mathrm{Stab}_{L_{\ell,\SO_{4n+2}}}(\psi_{\ell, w_0}) | =1.
\end{equation*}
\end{enumerate}
\end{enumerate}
\end{lemma}

\begin{proof}
This follows from \cite[Proposition 4.4]{GinzburgRallisSoudry2011}.	
\end{proof}

We first compute $J_{\psi_{\ell, w_0}}(\rho_{w,\eta_{t,\gamma}})$ when $\gamma=I_{4n+2-2\ell}$. To simplify notation, we may re-denote
\begin{equation*}
\eta_{t}=\eta_{t,I_{4n+2-2\ell}}.	
\end{equation*}

\begin{proposition}
\label{prop-JacquetModuleVanishing-alpha=0}
Let $w$ be a representative for $P_j(k)\backslash H(k)/P_\ell(k)$ corresponding to $(\alpha, \beta)$ given in Lemma~\ref{lemma-double-coset-decomposition} with $\alpha=0$, $j<\ell+\beta$. Let $\eta_{t, I_{4n+2-2\ell}}$ be given in \eqref{eq-eta-2} with $\gamma=I_{4n+2-2\ell},$ such that $\gamma w_0=w_0$ is not orthogonal to $V_{\ell,\beta}^-$. Then 
\begin{equation*}
J_{\psi_{\ell, w_0}}(\rho_{w,\eta_{t,I_{4n+2-2\ell}}})=0.	
\end{equation*}
In particular, if $w_0=y_\alpha=e_{2n}-\frac{\alpha}{2}e_{-2n}$ for $\alpha\in k^\times$ (hence $\ell<2n$), $\ell+\beta=\wt m$ and $j<\wt m$, then 
\begin{equation*}
J_{\psi_{\ell, w_0}}(\rho_{w,\eta_{t,I_{4n+2-2\ell}}})=0.	
\end{equation*}
\end{proposition}

\begin{proof}
As in the proof of Proposition~\ref{prop-JacquetModuleVanishing-alpha>0}, it is enough to find a unipotent subgroup $S\subset \eta_{t, I_{4n+2-2\ell}}^{-1} U_{\ell, j}^{(w)} \eta_{t, I_{4n+2-2\ell}}$ such that the properties in  \eqref{eq-prop-JacquetModuleVanishing-alpha>0-eq5} are satisfied. For special orthogonal groups, the existence of such unipotent subgroup $S$ is given in the proof of \cite[Proposition 5.2]{GinzburgRallisSoudry2011}. Hence we can also find such a unipotent subgroup $S$ in the $\GSpin$ case. 
\end{proof}

Let $\tau^{(t)}$ denote the $t$-th Bernstein-Zelevinsky derivative \cite{BernsteinZelevinsky1976} of $\tau$, along
the subgroup
$$
Z_t'=\left\{\begin{pmatrix}
               I_{\beta}&y\\0&z
                          \end{pmatrix}
\in \GL_j(k)\ \big|\ z\in Z_t(k)\right\}
$$
and associated to the character
$$
\psi_t'\left(
\begin{pmatrix}
I_\beta &y\\0&z
\end{pmatrix}
\right)=\psi^{-1}(z_{1,2}+\cdots+z_{t-1,t}).
$$
Then $\tau^{(t)}$ is a representation of $\GL_\beta(k)$, acting on the Jacquet module $J_{Z_t',\psi_t'}(V_\tau)$ via
the embedding $d\mapsto \diag(d,I_t)$ of $\GL_\beta(k)$ in $\GL_j(k)$. 
For $a\in k^\times$, we consider the character 
$\psi_{t,a}^{''}$ on $Z_t'$, defined by
$$
\psi_{t,a}^{''}\left(
\begin{pmatrix}
 I_\beta &y\\0&z
\end{pmatrix}\right)=\psi^{-1}(z_{1,2}+\cdots+z_{t-1,t}+ay_{\beta,1}).
$$
Let $\tau_{(t),a}$ denote the corresponding Jacquet module $J_{Z_t',\psi_{t,a}^{''}}(V_\tau)$, which is a
representation of the mirabolic subgroup $P_{\beta-1}^1(k)=\left\{ \left(\begin{smallmatrix} g&*\\&1\end{smallmatrix}\right)| g\in \GL_{\beta-1}(k) \right\}$ of $\GL_\beta(k)$. 
By \cite[Lemma 5.2]{GinzburgRallisSoudry2011}, for any $a, a^\prime\in k^\times$, $\tau_{(t),a}$ and $\tau_{(t),a^\prime}$ are isomorphic as representations of $P_{\beta-1}^1(k)$. Thus, we may simply write $\tau_{(t)}$ for any of such representations $\tau_{(t),a}$.

Let $N_{\ell-t}^\prime$ be the unipotent radical of the standard parabolic subgroup of $\GSpin(W_{j})$, preserving the flag
\begin{equation*}
0\subset \Span\{e_{j+1}\} \subset \Span\{e_{j+1}, e_{j+2}\} \subset \cdots \Span\{e_{j+1}, e_{j+2}, \cdots, e_{j+\ell-t}\}.	
\end{equation*}
Note that $N^\prime_{\ell-t}\subset \GSpin(W_j)$ is a subgroup similar to $N_{\ell}\subset H$. 
If $j=\wt m$, then $\ell=t$ and $N_{\ell-t}^\prime$ is the identity subgroup. 
Let $\psi_{\ell-t, \alpha}^\prime$ be the character of $N_{\ell-t}^\prime$ defined similarly as $\psi_{\ell,\alpha}$.
Let
\begin{equation*}
P_{\beta,t,\gamma}^\prime=H_{\ell,w_0}\cap \eta	_{t,\gamma}^{-1} P_{\ell,j}^{(w)} \eta_{t,\gamma}.
\end{equation*}
When $\gamma=I_{4n+2-2\ell}$, we re-denote
\begin{equation}
\label{eq-P_beta_t^prime}
P_{\beta,t}^\prime=P_{\beta,t,I_{4n+2-2\ell}}^\prime.
\end{equation}
When $w_0\in W_{\ell+\beta}$, or when $H$ is split, $P_{\beta,t}^\prime$ is the maximal parabolic subgroup of $H_{\ell,w_0}$ preserving the isotropic subspace $\omega_b^{t}\cdot V_{\ell,\beta}^+\cap w_0^\perp$. Otherwise, $P_{\beta,t}^\prime$ is a proper subgroup of this maximal parabolic subgroup.

We continue with the computation of $J_{\psi_{\ell, w_0}}(\rho_{w,\eta_{t,I_{4n+2-2\ell}}})$. By Proposition~\ref{prop-JacquetModuleVanishing-alpha=0}, it remains to consider the following cases:
\begin{itemize}
\item $\ell<\wt m$, $w_0=y_\alpha$, and $\ell+\beta<\wt m$;
\item $\ell<\wt m$, $w_0=y_\alpha$, and $j=\ell+\beta=\wt m$;
\item $\ell=\wt m$, $w_0\in V_0$ (and $\beta=0$, $t=j$, $1\le j\le \ell=\wt m=2n$).
\end{itemize}

First, we assume $\ell<\wt m$, $w_0=y_\alpha$, and $\ell+\beta<\wt m$. Recall that   
\begin{equation*}
\rho_{w,\eta_t}=\rho_{w,\eta_{t,I_{4n+2-2\ell}}}=\ind_{R_{\ell,w_0}^{\eta_t}}^{R_{\ell,w_0}} \pi_{\ell,j}^{(w,\eta_t)}(\tau,\sigma).
\end{equation*}
We consider the following map on $\rho_{w,\eta_t}$ given by
\begin{equation}
\label{eq-iso-T-1}
T(f)(x):=\int_{N_{\ell}\cap \eta_t^{-1} P_{\ell,j}^{(w)}\eta_t \backslash N_{\ell} } \left( J_{Z_t^\prime,\psi_t^\prime}\otimes J_{N_{\ell-t}^\prime, \psi_{\ell-t,\alpha} ^\prime} \right) (f(nx)) \psi_{\ell,\alpha}^{-1}(n)dn.
\end{equation}
Here, $f$ is in the space of $\rho_{w,\eta_t}$, $x\in H_{\ell,w_0}$, and $N_{\ell}\cap \eta_t^{-1} P_{\ell,j}^{(w)}\eta_t$ consists of elements of the form
\begin{equation}
\label{eq-N-ell-sub1}
	  \begin{pmatrix}
c & 0 & 0&0&y_6 &0 &0\\	
&b&0&y_4&y_5 &z_4&0\\ 
&&I_{\beta} &0 &0 &y_5^\prime &y_6^\prime\\
&&&I_{4n+2-2(\ell+\beta)} & 0 &0&0\\
&&&&I_\beta &0&0\\
&&&&&b^* &0\\
&&&&&&c^* 
\end{pmatrix}^{\omega_b^{t^\prime}}  
\end{equation}
where $c\in Z_t, b\in Z_{\ell-t}$. 
Here $t^\prime=0$, except in the following cases, where $t^\prime=1$: $\GSpin(V)$ is split, and either $j=2n+1$ and $t$ is odd, or, $j, \ell<2n+1$, $\ell+\beta=2n+1$, and $w=\widetilde \epsilon_{0,\beta}=\omega_b  {\epsilon}_{0, \beta}  \omega_b^{-1}$ as in Lemma~\ref{lemma-double-coset-decomposition}(b). 
The functor $J_{N_{\ell-t}^\prime, \psi_{\ell-t,\alpha} ^\prime}$ is applied to $\sigma^{\omega_{b,j}^t}$, where we recall $\omega_{b,j}$ is given by \eqref{eq-omega_b,j-1} and \eqref{eq-omega_b,j-2}.
To check that the integral on the right-hand side of \eqref{eq-iso-T-1} is well-defined, we take $u\in N_{\ell}\cap \eta_t^{-1} P_{\ell,j}^{(w)}\eta_t$ as in \eqref{eq-N-ell-sub1}, where $c\in Z_t, b\in Z_{\ell-t}$. Then 
\begin{equation*}
\begin{split}
&	\left( J_{Z_t^\prime,\psi_t^\prime}\otimes J_{N_{\ell-t}^\prime, \psi_{\ell-t,\alpha} ^\prime} \right) (f(ux))   \\
=&  \left( J_{Z_t^\prime,\psi_t^\prime}\otimes J_{N_{\ell-t}^\prime, \psi_{\ell-t,\alpha} ^\prime} \right) \left( \left(\tau\begin{pmatrix}I_\beta & y_6^\prime \\ & c^*\end{pmatrix}\otimes \sigma^{\omega_{b,j}^t}\left[\begin{pmatrix} b &y_4&z_4\\ &I_{4n+2-2(\ell+\beta)} & y_4^\prime\\ &&b^*\end{pmatrix}^{\omega_{b,j}^{t^\prime}} \right] \right)(f(x))  \right)\\
=& \psi_t^\prime \begin{pmatrix}I_\beta & y_6^\prime \\ & c^*\end{pmatrix}  \psi_{\ell-t,\alpha} ^\prime \left[\begin{pmatrix} b &y_4&z_4\\ &I_{4n+2-2(\ell+\beta)} & y_4^\prime\\ &&b^*\end{pmatrix}^{\omega_{b,j}^{t^\prime}} \right] \left( J_{Z_t^\prime,\psi_t^\prime}\otimes J_{N_{\ell-t}^\prime, \psi_{\ell-t,\alpha} ^\prime} \right) (f(ux)) \\
=& \psi_{\ell,\alpha}(u) \left( J_{Z_t^\prime,\psi_t^\prime}\otimes J_{N_{\ell-t}^\prime, \psi_{\ell-t,\alpha} ^\prime} \right) (f(ux)).
\end{split}
\end{equation*}
Thus, the integral is well-defined. Moreover, this integral is absolutely convergent, since the function $u\mapsto f(ux)$ on $N_{\ell}$ has compact support, modulo $N_{\ell}\cap \eta_t^{-1} P_{\ell,j}^{(w)}\eta_t$.

Set $|\cdot|:=|\det(\cdot)|$.

\begin{lemma}
\label{lemma-JacquetModule-lemma1}
Assume that $\ell+\beta<\wt m$ and $w_0=y_\alpha$. 
The map $T$ in \eqref{eq-iso-T-1} induces an $H_{\ell,\alpha}$-isomorphism
\begin{equation}
\label{eq-isom-Tprime-1}
T^\prime: J_{\psi_{\ell,\alpha}}(	\rho_{w,\eta_{t}} )  \cong  \ind_{P^\prime_{\beta,t}}^{H_{\ell,\alpha}} |\cdot|^{\frac{1-t}{2}} \tau^{(t)} \otimes J_{N_{\ell-t}^\prime, \psi_{\ell-t,\alpha} ^\prime}(\sigma^{\omega_{b,j}^t}).
\end{equation}
The subgroup $P^\prime_{\beta,t}$ is a parabolic subgroup of $H_{\ell,\alpha}$, and the right-hand side of \eqref{eq-isom-Tprime-1} is a parabolic induction.
\end{lemma}

\begin{proof}
It is clear that $T$ factors through the Jacquet module on the left-hand side of  \eqref{eq-isom-Tprime-1}. 
We denote by $T^\prime$ the induced map on $J_{\psi_{\ell,\alpha}}(	\rho_{w,\eta_{t}} )$. 
Given $p\in P_{\beta,t}^\prime$, note that $\eta_t p \eta_t^{-1}=p$, and we may write
\begin{equation*}
w \eta_t p \eta_t^{-1} w^{-1} = w p w^{-1}= (g_0, h_0)\cdot u	
\end{equation*}
where $(g_0,h_0)\in \GL_j\times \GSpin(W_j)\cong M_j$, and $u\in U_j$. Then we have
\begin{equation*}
\delta_{P_j}^{\frac{1}{2}}( w p w^{-1})=\delta_{P_{\beta,t}^\prime}^{\frac{1}{2}}(p) \cdot  |\det(g_0)|^\ell \cdot  |\det(g_0)|^{\frac{1-t}{2}}.	
\end{equation*}
Thus $T^\prime(f)$ (and also $T(f)$) belongs to the space of
\begin{equation*}
	\ind_{P^\prime_{\beta,t}}^{H_{\ell,\alpha}} |\cdot|^{\frac{1-t}{2}} \tau^{(t)} \otimes J_{N_{\ell-t}^\prime, \psi_{\ell-t,\alpha} ^\prime}(\sigma^{\omega_{b,j}^t}).
\end{equation*}
Note that if $t=\ell$, then $J_{N_{\ell-t}^\prime, \psi_{\ell-t,\alpha} ^\prime}(\sigma^{\omega_{b,j}^t})=\Res_{H_{\ell,\alpha}\cap \GSpin(W_j)}(\sigma^{\omega_{b,j}^t})$.

We first prove that $T^\prime$ is injective. Assume that $T(f)=0$. We will use the same notation as in the proof of Proposition~\ref{prop-JacquetModuleVanishing-alpha>0}. Then this assumption means that $T(f)(x)=0$ for all $x\in \Omega_2$. For $x\in \Omega_2$, the domain of integration in \eqref{eq-iso-T-1} can be replaced by $\Omega_1$. Hence, for all $x\in \Omega_2$, we have
\begin{equation*}
	\left( J_{Z_t^\prime,\psi_t^\prime}\otimes J_{N_{\ell-t}^\prime, \psi_{\ell-t,\alpha} ^\prime} \right) \left( \int_{\Omega_1} \psi_{\ell,\alpha}^{-1}(u)f(nx)dn \right)=0.
\end{equation*}
This means that there is $i_0$, such that, for all $i\ge i_0$ and all $x\in \Omega_2$, we have 
\begin{equation*}
\int_{N_{\ell}^{(i)}\cap \eta_t^{-1} P_{\ell,j}^{(w)}\eta_t } \psi_{\ell,\alpha}^{-1}(u) (\tau\otimes \sigma^{\omega_{b,j}^t})(u) 	\left( \int_{\Omega_1} \psi_{\ell,\alpha}^{-1}(n)f(nx)dn \right) du=0.
\end{equation*}
Hence, for all $i\ge i_0$ and all $x\in \Omega_2$, we have
\begin{equation}
\label{eq-lemma-JacquetModule-lemma1-eq1}
	\int_{N_{\ell}^{(i)}\cap \eta_t^{-1} P_{\ell,j}^{(w)}\eta_t } \int_{\Omega_1}  \psi_{\ell,\alpha}^{-1}(un)f(unx)dudn=0.
\end{equation}
To prove that $T^\prime$ is injective, it suffices to show that $f$ satisfies the property
\begin{equation}
\label{eq-lemma-JacquetModule-lemma1-eq2}
\int_{N_{\ell}^{(i)}} \psi_{\ell,\alpha}^{-1}(n)\rho(n)fdn=0	
\end{equation}
for $i$ large enough, 
where $\rho$ denotes the right translation action on the space of $\ind_{R_{\ell,w_0}^{\eta_t}}^{R_{\ell,w_0}} \pi_{\ell,j}^{(w,\eta_t)}(\tau,\sigma)$.
It is enough to show that, there is $i_0^\prime$, such that for all $n_0\in N_{\ell}$ and $x\in \Omega_2$, we have
\begin{equation*}
\int_{N_{\ell}^{(i^\prime)}} \psi_{\ell,\alpha}^{-1}(n)f(n_0 x n)dn=0	
\end{equation*}
for all $i^\prime\ge i_0^\prime$. 
We can find a compact open subgroup $H^0$ of $H_{\ell,w_0}$ containing the identity, such that $x N_{\ell}^{(i)} x^{-1}= N_{\ell}^{(i)}$, for all $x\in H^0$ and $i$ large enough. We may replace $\Omega_2$ by $\Omega_2 H^0$, and write $\Omega_2$ as a finite union of cosets $x_j H^0$, for $1\le j \le N$, where $x_j\in \Omega_2$. We can make a change of variable $n\mapsto a^{-1}n a$ for $a\in H^0\subset H_{\ell,w_0}$. Since $H_{\ell,w_0}$ stabilizes the character $\psi_{\ell,w_0}$, it suffices to prove that, for $i^\prime$ large enough and $1\le j \le N$, we have
\begin{equation*}
\int_{N_{\ell}^{(i^\prime)}} \psi_{\ell,w_0}^{-1}(n)f( n_0   x_j n a) dn=0, \quad \forall n_0\in N_{\ell}, a\in H^0. 
\end{equation*}
We may assume that, given $i^\prime$ large enough, there is $i$, large enough, such that $x_j^{-1} N_{\ell}^{(i^\prime)} x_j \supset N_{\ell}^{(i)}$. We write $x_j N_{\ell}^{(i^\prime)} x_j^{-1}$ as a finite union of cosets $y N_{\ell}^{(i)}$,  with representatives $y\in N_{\ell}$. Thus, it is enough to show that for $i$ sufficiently large and $1\le j \le N$, we have
\begin{equation*}
\int_{N_{\ell}^{(i)}} \psi_{\ell,w_0}^{-1}(n)f( n_0  n x_j  a) dn=0, \quad \forall n_0\in N_{\ell}, a\in H^0. 
\end{equation*}
Since $x_jH^0\subset \Omega_2$, it is enough to show that for $i$ sufficiently large, we have
\begin{equation}
\label{eq-lemma-JacquetModule-lemma1-eq3}
\int_{N_{\ell}^{(i)}} \psi_{\ell,w_0}^{-1}(n)f( n_0  n x) dn=0, \quad \forall n_0\in N_{\ell}, x\in \Omega_2. 
\end{equation}
Recall that the restriction of $f$ to the set $N_{\ell}\Omega_2$ is supported in $(N_{\ell}\cap \eta_t^{-1} P_{\ell,j}^{(w)}\eta_t)\Omega_1 \Omega_2$. Thus, it is enough to take, in \eqref{eq-lemma-JacquetModule-lemma1-eq3}, $n_0 n\in (N_{\ell}\cap \eta_t^{-1} P_{\ell,j}^{(w)}\eta_t)\Omega_1$. We may take $i$ to be large enough, so that $\Omega_1 N_{\ell}^{(i)}=N_{\ell}^{(i)}$. Thus, it is enough to show that for $i$ large enough, we have
\begin{equation}
\label{eq-lemma-JacquetModule-lemma1-eq4}
\int_{N_{\ell}^{(i)}} \psi_{\ell,w_0}^{-1}(n)f( n_0  n x) dn=0, \quad \forall n_0\in N_{\ell}^{(i)}, x\in \Omega_2. 
\end{equation}
By a change of variable in $n$, it suffices to prove \eqref{eq-lemma-JacquetModule-lemma1-eq4} when $n_0=1$. Thus, it suffices to prove that for $i$ large enough, we have
\begin{equation}
\label{eq-lemma-JacquetModule-lemma1-eq5}
\int_{N_{\ell}^{(i)}} \psi_{\ell,w_0}^{-1}(n)f(    n x) dn=0, \quad \forall   x\in \Omega_2. 
\end{equation}
Again, using the fact that the restriction of $f$ to the set $N_{\ell}\Omega_2$ is supported in $(N_{\ell}\cap \eta_t^{-1} P_{\ell,j}^{(w)}\eta_t)\Omega_1 \Omega_2$, we may replace the domain of integration in \eqref{eq-lemma-JacquetModule-lemma1-eq5} by $(N_{\ell}^{(i)}\cap \eta_t^{-1} P_{\ell,j}^{(w)}\eta_t)\Omega_1$. This is \eqref{eq-lemma-JacquetModule-lemma1-eq1}. Thus, we have proved that $T^\prime$ is injective. 

Now we prove that $T^\prime$ is surjective. Let $\phi$ be an element in the space of $\ind_{P^\prime_{\beta,t}}^{H_{\ell,\alpha}} |\cdot|^{\frac{1-t}{2}} \tau^{(t)} \otimes J_{N_{\ell-t}^\prime, \psi_{\ell-t,\alpha} ^\prime}(\sigma^{\omega_{b,j}^t})$. We assume that the support of $\phi$ is $P^\prime_{\beta,t}\Omega_2$, where $\Omega_2$ is compact and open. Since the map $J_{Z_t^\prime,\psi_t^\prime}\otimes J_{N_{\ell-t}^\prime, \psi_{\ell-t,\alpha} ^\prime} $ from the space of  $\pi_{\ell,j}^{(w,\eta_t)}(\tau,\sigma)$ to the space of $\tau^{(t)}\otimes J_{N_{\ell-t}^\prime, \psi_{\ell-t,\alpha} ^\prime} (\sigma^{\omega_{b,j}^t} )$ is surjective, it is clear that we can find a smooth function $f$ in the space of $\ind_{P^\prime_{\beta,t}}^{H_{\ell,w_0}} \pi_{\ell,j}^{(w,\eta_t)}(\tau,\sigma)$, such that, 
\begin{equation*}
	\left( J_{Z_t^\prime,\psi_t^\prime}\otimes J_{N_{\ell-t}^\prime, \psi_{\ell-t,\alpha} ^\prime} \right)(f(x))=\phi(x), \quad \forall x\in \Omega_2. 
\end{equation*}
We now extend $f$ to be a function in the space of $\ind_{R_{\ell,w_0}^{\eta_t}}^{R_{\ell,w_0}} \pi_{\ell,j}^{(w,\eta_t)}(\tau,\sigma)$, such that, for all $x\in \Omega_2$, the following function on $N_{\ell}$, 
\begin{equation*}
n\mapsto f(nx),	
\end{equation*}
is supported inside $(N_{\ell}\cap \eta_t^{-1} P_{\ell,j}^{(w)} \eta_t)\Omega_1^0$, where $\Omega_1^0$ is a sufficiently small compact open subgroup of $N_{\ell}$, and is such that $f$ is right invariant under $x^{-1}\Omega_1^0 x$ for all $x\in \Omega_2$ (which is possible because $\Omega_2$ is compact and open). Then we have
\begin{equation*}
T(f)(x)=c_0 \phi(x), \quad \forall x\in \Omega_2,	
\end{equation*}
where $c_0$ is the measure of $\Omega_1^0$, modulo $N_{\ell}\cap \eta_t^{-1}P_{\ell,j}^{(w)} \eta_t$. This proves that $T^\prime$ is surjective. 
\end{proof}

The Jacquet module in the case $\ell<\wt m$, $w_0=y_\alpha$, and $j=\ell+\beta=\wt m$ is given below. 
\begin{lemma}
\label{lemma-JacquetModule-lemma2}
Assume that $\ell<\wt m$, $w_0=y_\alpha$, and $j=\ell+\beta=\wt m$. Then we have the following isomorphism of $H_{\ell,\alpha}$-modules:
\begin{equation*}
	J_{\psi_{\ell,\alpha}}(	\rho_{w,\eta_{t}} ) \cong  \ind_{P^\prime_{\beta,t}}^{H_{\ell,\alpha}} |\cdot|^{\frac{-\ell}{2}} \tau_{(\ell)} \otimes  \sigma^{\omega_{b,j}^t}.
\end{equation*}
The subgroup $P^\prime_{\beta,t}$ is a parabolic subgroup of $H_{\ell,\alpha}$ only when $\wt m=2n+1$.
\end{lemma}

\begin{proof}
The isomorphism is induced by the following map on $\rho_{w,\eta_t}$ given by
\begin{equation}
\label{eq-iso-T-2}
\wt T(f)(x):=\int_{N_{\ell}\cap \eta_t^{-1} P_{\ell,j}^{(w)}\eta_t \backslash N_{\ell} } \left( J_{Z_t^\prime,\psi_{t,a}^{\prime\prime}}\otimes \Res_{H_{\ell,\alpha}^\prime } \right) (f(ux)) \psi_{\ell,\alpha}^{-1}(u)du
\end{equation}  
where $H_{\ell,\alpha}^\prime=H_{\ell,\alpha}\cap \GSpin(W_{j})$ and $a=(-\frac{\alpha}{2})^{1-t^\prime}$.
The proof now proceeds in a similar way as in Lemma~\ref{lemma-JacquetModule-lemma1}. 
\end{proof}

Now we consider the case $\ell=\wt m$, $w_0\in V_0$, and $V_0\not=0$. Note that this implies that $\beta=0$, $t=j$, $1\le j\le \ell=\wt m=2n$.

\begin{lemma}
\label{lemma-JacquetModule-lemma3}
Assume that $\ell=\wt m$, $w_0\in V_0$, $\beta=0$, $t=j$, and $1\le j\le \ell=\wt m=2n$. Then we have  
\begin{equation*}
	J_{\psi_{\ell,\alpha}}(	\rho_{w,\eta_{t}} ) \cong   d_{\tau} \cdot J_{\psi_{\ell-t,w_0}^\prime}( \sigma^{\omega_{b,j}^t})
\end{equation*}
where $d_\tau=\dim \tau^{(j)}$, and $\psi_{\ell-t,w_0}^\prime$ is the character defined similarly as $\psi_{\ell,w_0}$, but on a unipotent subgroup of $\GSpin(W_j)$.
\end{lemma}

\begin{proof}
By the same proof as in Lemma~\ref{lemma-JacquetModule-lemma1}, we can show that $J_{\psi_{\ell,\alpha}}(	\rho_{w,\eta_{t}} )$ is isomorphic to $\tau^{(t)} \otimes J_{\psi_{\ell-t,w_0}^\prime}( \sigma^{\omega_{b,j}^t})$. Here, $\tau^{(t)}$ is the Jacquet module of $\tau$ with respect to the Whittaker character $\psi^{-1}$. Thus, if $\tau$ has no Whittaker functionals, then $J_{\psi_{\ell,\alpha}}(	\rho_{w,\eta_{t}} )=0$, and if $\tau$ has Whittaker functionals, then $J_{\psi_{\ell,\alpha}}(	\rho_{w,\eta_{t}} ) \cong   d_{\tau} \cdot J_{\psi_{\ell-t,w_0}^\prime}( \sigma^{\omega_{b,j}^t})$ where $d_\tau=\dim \tau^{(j)}$.
\end{proof}

We now move on to consider the Jacquet module $J_{\psi_{\ell, w_0}}(\rho_{w,\eta_{t,\gamma}})$ when $\gamma \not=I_{4n+2-2\ell}$. In this case, $P_{w, \SO_{4n+2}}^\prime$ is a proper parabolic subgroup of $\SO(W_{\ell})$ either stabilizing a $\beta$-dimensional isotropic subspace of $W_{\ell}$ or a $(\wt m-\ell)$-dimensional isotropic subspace of $W_{\ell}$ (see Remark~\ref{remark-parabolic-subgroup} and Lemma~\ref{lemma-double-coset2}). Thus we must have $\ell<\wt m$ and $\beta>0$. Then \eqref{eq-local-Jacquet-module-inequalities} becomes
\begin{equation}
\label{eq-local-Jacquet-module-inequalities-2}
0< \beta\le j \le \ell+\beta\le \wt m.	
\end{equation}
In particular, we may take $w_0=y_{\alpha}$. 

First, we assume that $\ell+\beta<\wt m$. By Lemma~\ref{lemma-double-coset2} (a), the double coset space $$P_{w,\SO_{4n+2}}^\prime\backslash \SO(W_{\ell}) / \mathrm{Stab}_{L_{\ell,\SO_{4n+2}}}(\psi_{\ell, w_0})$$ 
has two elements. Since we are interested in the case when $\gamma\not=I_{4n+2-2\ell}$, we may take
\begin{equation}
\label{eq-gamma}
\gamma=\gamma_\beta= \begin{pmatrix}
 0 & I_{\beta} & & &\\
 I_{\wt m -\ell-\beta} & 0 & & &\\
 &&I_{V_0} &&\\
 &&&0&I_{\wt m-\ell-\beta} \\
 &&&I_{\beta} & 0
 \end{pmatrix}\in \SO(W_{\ell}).
\end{equation}
With the choices of $\epsilon$ in \eqref{eq-epsilon} and $\gamma_\beta$ in \eqref{eq-gamma}, we re-denote $\eta=\eta_{\beta, t}$ such that 
\begin{equation*}\
\pr(\eta)=\pr(\eta_{\beta, t})=  \begin{pmatrix}
 \epsilon & & \\ & \gamma_\beta &\\ && \epsilon^*	
 \end{pmatrix}.	
\end{equation*}
Let
\begin{equation}
\label{eq-P-double-prime}
P_{\beta}^{\prime\prime} = H_{\ell,\alpha}\cap \eta_{\beta, t}^{-1} P_{\ell,j}^{(w)}\eta_{\beta,t} = \eta_{\beta, t}^{-1} (H_{\ell, \eta_{\beta, t}y_{\alpha}} \cap P_{\ell,j}^{(w)} ) \eta_{\beta, t}.
\end{equation}
This is a subgroup of the maximal parabolic subgroup of $H_{\ell,\alpha}$, which preserves the isotropic subspace $(\pr(\omega_{b,j})^{t^\prime} \eta_{\beta, t})^{-1} V_{\ell, \beta}^+\cap y_{\alpha}^\perp$. The computation of $J_{\psi_{\ell,\alpha}}(	\rho_{w,\eta_{\beta,t}} )$ when $\ell+\beta<\wt m$ is given in the following two lemmas. 

\begin{lemma}
\label{lemma-JacquetModule-gamma-lemma1}
Assume that $0< \beta\le j < \ell+\beta< \wt m$. 	Then we have
\begin{equation*}
	J_{\psi_{\ell,\alpha}}(	\rho_{w,\eta_{\beta,t}} ) =0.
\end{equation*}
\end{lemma}

\begin{proof}
As in the proof of Proposition~\ref{prop-JacquetModuleVanishing-alpha>0}, there exists a unipotent subgroup $S\subset \eta_{\beta, t}^{-1}U_{\ell,j}^{(w)}\eta_{\beta, t}$ which satisfies the properties in \eqref{eq-prop-JacquetModuleVanishing-alpha>0-eq5}. We now proceed in the same way as in the proof of Proposition~\ref{prop-JacquetModuleVanishing-alpha>0} to conclude that $J_{\psi_{\ell,\alpha}}(	\rho_{w,\eta_{\beta,t}} ) =0$. 
\end{proof}

\begin{lemma}
\label{lemma-JacquetModule-gamma-lemma2}
Assume that $0< \beta\le j = \ell+\beta< \wt m$. 	Then we have
\begin{equation*}
	J_{\psi_{\ell,\alpha}}(	\rho_{w,\eta_{\beta,t}} ) = \ind_{P_{\beta}^{\prime\prime}}^{H_{\ell,\alpha}} |\cdot|^{-\frac{\ell}{2}} \tau_{(\ell)}\otimes \sigma^{\omega_{b,j}^{\ell}}.
\end{equation*}
\end{lemma}
\begin{proof}
Note that the assumption $0< \beta\le j = \ell+\beta< \wt m$ forces $\ell<j<\wt m$ and $t=j-\beta=\ell$. The proof is similar to the proof of Lemma~\ref{lemma-JacquetModule-lemma1}. 
\end{proof}

It remains to compute $J_{\psi_{\ell,\alpha}}(	\rho_{w,\eta_{\epsilon,\gamma}} )$ where $\epsilon$ is given in \eqref{eq-epsilon}, $\gamma\not=I_{4n+2-2\ell}$, and $0< \beta\le j \le \ell+\beta= \wt m$. This puts us in Lemma~\ref{lemma-double-coset2} (2). Since $\gamma\not=I_{4n+2-2\ell}$, we are in the case of Lemma~\ref{lemma-double-coset2} (2)(a), and we must have $V_0\not=0$ and
\begin{equation*}
\mathrm{Witt}( y_{\alpha}^\perp \cap W_{\ell} )=\wt m-\ell.
\end{equation*}
This is equivalent to the condition that
\begin{equation*}
\mathrm{Witt}(k y_{-\alpha}+V_0)=1.	
\end{equation*}
In this case, we may take
\begin{equation}
\label{eq-gamma-alpha}
\gamma=
\gamma_\alpha=
\begin{pmatrix}
I_{2n-\ell-1} &&&&\\
&1&&&\\
&-v_{\alpha} &I_{V_0} &&\\
&\frac{\alpha}{2} &v_{\alpha}^\prime &1&\\
&&&&I_{2n-\ell-1}
\end{pmatrix},
\end{equation}
where $v_\alpha\in V_0$ is an element such that $q_V(v_\alpha, v_\alpha)=-\alpha$. Denote 
\begin{equation}
\label{eq-P_alpha^prime}
P_{\alpha}^\prime=H_{\ell,\alpha}\cap \eta_{t,\gamma_\alpha}^{-1} P_{\ell,j}^{(w)} \eta_{t,\gamma_\alpha}= \eta_{t,\gamma_\alpha}^{-1} (H_{\ell, \eta_{t,\gamma_\alpha} \cdot y_\alpha }\cap P_{\ell,j}^{(w)}) \eta_{t,\gamma_\alpha}.
\end{equation}
This is a parabolic subgroup of $H_{\ell,\alpha}$. The Levi subgroup of $H_{\ell, \eta_{t,\gamma_\alpha} \cdot y_\alpha }\cap P_{\ell,j}^{(w)}$ is isomorphic to $\GL_\beta\times \GSpin(v_\alpha^\perp\cap V_0)$.

\begin{lemma}
\label{lemma-JacquetModule-gamma-lemma22}
Assume that $0< \beta\le j \le \ell+\beta= \wt m$, and $\mathrm{Witt}(k y_{-\alpha}+V_0)=1$.  Then we have
\begin{equation*}
	J_{\psi_{\ell,\alpha}}(	\rho_{w,\eta_{\beta,t}} ) = \ind_{P_{\alpha}^{\prime}}^{H_{\ell,\alpha}} |\cdot|^{\frac{1-t}{2}} \tau^{(t)}\otimes  J_{N_{\ell-t}^\prime, \psi_{\ell-t,v_\alpha} ^\prime}(\sigma^{\omega_{b,j}^t}).
\end{equation*}
Here, $\psi_{\ell-t,v_\alpha} ^\prime$ is a character defined similarly as $\psi_{\ell,y_{\alpha}}$, and
when $t=\ell$, we interpret $J_{N_{\ell-t}^\prime, \psi_{\ell-t,v_\alpha} ^\prime}(\sigma^{\omega_{b,j}^t})=\Res_{\GSpin(v_\alpha^\perp \cap V_0)}(\sigma^{\omega_{b,j}^t})$.
\end{lemma}
\begin{proof}
 The proof is similar to the proof of Lemma~\ref{lemma-JacquetModule-lemma1}. 
\end{proof}

We now summarize the above computation in the following Leibniz rule. We let ``$\equiv$'' denote isomorphism of representations, up to semi-simplification.

\begin{theorem}
\label{thm-JacquetModule}
 We have the following. 
\begin{enumerate}
\item Assume that $0\leq \ell <\wt m$ and $1\leq j <\wt m$, then
$$
\begin{aligned} J_{\psi_{\ell, \alpha}}(\Ind_{P_j}^H \tau\otimes \sigma)\equiv
&\bigoplus_{\ell+j-\wt m<t\leq \ell,\ 0\leq t \leq j}
\ind_{P^\prime_{\beta,t}}^{H_{\ell,\alpha}} |\cdot|^{\frac{1-t}{2}} \tau^{(t)} \otimes J_{N_{\ell-t}^\prime, \psi_{\ell-t,\alpha} ^\prime}(\sigma^{\omega_{b,j}^t})\\
&\oplus \begin{cases}
            \ind_{P_{j-\ell}^{\prime\prime}}^{H_{\ell,\alpha}} |\cdot|^{-\frac{\ell}{2}} \tau_{(\ell)}\otimes \sigma^{\omega_{b,j}^{\ell}}, & \text{if $\ell<j$},\\
            0, & \text{otherwise};
           \end{cases}\\
&\oplus \delta_{\alpha}\cdot
\begin{cases}
 \ind_{P_{\alpha}^{\prime}}^{H_{\ell,\alpha}} |\cdot|^{\frac{1+\wt m-\ell-j}{2}} \tau^{(t)}\otimes  J_{N_{\ell-t}^\prime, \psi_{\ell-t,v_\alpha} ^\prime}(\sigma^{\omega_{b,j}^{\ell+j-\wt m}}), & \text{if $0<2n-\ell \leq j$},\\
 0, & \text{otherwise}.
\end{cases}
\end{aligned}
$$
Here, $\delta_\alpha=0$ unless $\mathrm{Witt}(k\cdot y_{-\alpha}+V_0)=1$, in which case, $\delta_\alpha=1$. The subgroups $P_{\beta,t}^\prime$, $P_{j-\ell}^{\prime\prime}$, and $P_{\alpha}^\prime$ are defined by  \eqref{eq-P_beta_t^prime}, \eqref{eq-P-double-prime}, \eqref{eq-P_alpha^prime} respectively.
\item Assume that $0\leq \ell \leq \wt m$ and $j=\wt m$, then
 $$
 J_{\psi_{\ell, \alpha}}(\Ind_{P_j}^H \tau\otimes \sigma) \equiv
\ind_{P^\prime_{\wt m-\ell ,t}}^{H_{\ell,\alpha}} |\cdot|^{\frac{-\ell}{2}} \tau_{(\ell)} \otimes  \sigma^{\omega_b^t}
\oplus
 \delta_\alpha\cdot \ind_{P_{\alpha}^{\prime}}^{H_{\ell,\alpha}} |\cdot|^{\frac{1-\ell}{2}} \tau^{(\ell)}\otimes  J_{N_{0}^\prime, \psi_{0,v_\alpha} ^\prime}(\sigma^{\omega_{b,j}^\ell}).
 $$

  \item Assume that $\ell=\wt m$ and $w_0\in V_0$, then
 $$J_{\psi_{\ell,w_0}}(\Ind_{P_j}^H \tau\otimes \sigma)\equiv d_\tau\cdot  J_{\psi_{\ell-j,w_0}^\prime}( \sigma^{\omega_{b,j}^j}).$$
 Here, $d_\tau=\dim \tau^{(j)}$.
\end{enumerate}
 \end{theorem}

\begin{remark}
Theorem~\ref{thm-JacquetModule} is a generalization of \cite[Theorem 5.1]{GinzburgRallisSoudry2011} to the $\GSpin$ group case under consideration. 
\end{remark}

\subsection{Vanishing of twisted Jacquet module}
The goal of this section is to compute the twisted Jacquet module for any local unramified component of the residual representation corresponding to the $\psi_{\ell,w_0}$-Fourier coefficient for $\ell>n$. 
This will be used to prove part (1) of Theorem~\ref{thm-Main} on
the vanishing of the representation $\TD_{\psi_{\ell,\alpha}}(\CE_{\tau\otimes\sigma})$ of $\GSpin^{\delta,\alpha}_{4n+1-2\ell}(\A)$ 
for all $n<\ell\leq 2n$.

Let $\sigma$ be an irreducible unramified representation of $\GSpin(W_{2n})$. When $\GSpin(W_{2n})$ is $k$-split, we have $\GSpin(W_{2n})\cong \GL_1 \times \GL_1$, and the representation $\sigma=\sigma_1\otimes\sigma_2$ for a pair $(\sigma_1, \sigma_2)$ of unramified characters  of $k^\times$, with $\omega_\sigma=\sigma_1\sigma_2$. When $\GSpin(W_{2n})$ is $k$-quasi-split but not $k$-split, we have $\GSpin(W_{2n})\cong \Res_{K/k}(\GL_1)$ for a quadratic extension $K/k$, and the representation $\sigma$ is an unramified character of $K^\times$, with $\omega_{\sigma}=\sigma|_{k^\times}$. 

Let $\tau$ be an irreducible unramified generic representation of $\GL_{2n}(k)$ with central character $\omega_{\tau}$, such that $\wt{\tau}\cong \tau\otimes\omega_{\sigma}^{-1}$. Then $\omega_{\tau}=\omega_{\sigma}^n$, and we may write $\tau$ as 
\begin{equation*}
\tau=\Ind_{B_{\GL_{2n}}}^{\GL_{2n}} (\mu_1 \otimes   \cdots \otimes 	\mu_n \otimes \mu_n^{-1}  \omega_{\sigma} \otimes \cdots \mu_1^{-1}\omega_{\sigma})
\end{equation*}
where $B_{\GL_{2n}}$ is the standard Borel subgroup of $\GL_{2n}$. 

Let $\pi_{\tau\otimes\sigma}$ be the unramified constituent of $\Ind_{P_{2n}}^{H}(\tau\cdot |\det|^{\frac{1}{2}}\otimes \sigma)$.

\begin{proposition}
\label{prop-Jacquet-Module-Unramified-Component-vanishing}
Let $\sigma$ be an irreducible unramified representation of $\GSpin(W_{2n})$ and let $\tau$ be an irreducible unramified generic representation of $\GL_{2n}(k)$ with central character $\omega_{\tau}$, such that $\wt{\tau}\cong \tau\otimes\omega_{\sigma}^{-1}$. 
\begin{enumerate}
\item Assume that $J_{\delta}$ is non-split over $k$. 
\begin{enumerate}
\item If $w_0=y_{\alpha}$ for $\alpha\in k^\times$ such that   $J_{\delta,\alpha}$ is split over $k$, then we have 
 $J_{\psi_{\ell,\alpha}}(\pi_{\tau\otimes \sigma})=0$ for $\ell\geq n+1$.
  \item If $w_0\in V_0$, then $J_{\psi_{{2n},w_0}}(\pi_{\tau\otimes \sigma})=0$.
\end{enumerate}

\item Assume that $J_\delta$ splits over $k$.  For any choice of $\alpha\in k^\times$, we have $J_{\psi_{\ell,\alpha}}(\pi_{\tau\otimes \sigma})=0$ for $\ell\geq  n+1$.
\end{enumerate}
\end{proposition}

\begin{proof}
 We first assume that $J_{\delta}$ is non-split over $k$, so $\wt m=2n$ and $V_0\not=0$. In this case,  $\pi_{\tau\otimes \sigma}$ is the unramified constituent of the representation of $H$ induced from the character 
\begin{equation}
\label{eq-Jacquet-Module-Unramified-Component-eq1}
 \mu_1|\cdot|^{\frac{1}{2}}\otimes\cdots \otimes\mu_n|\cdot|^{\frac{1}{2}}\otimes \mu_n^{-1}|\cdot|^{\frac{1}{2}}\omega_{\sigma} \otimes \cdots \otimes   \mu_1^{-1}|\cdot|^{\frac{1}{2}}\omega_{\sigma}\otimes \sigma.
 \end{equation}
We can find an element in the Weyl group which conjugates the above character \eqref{eq-Jacquet-Module-Unramified-Component-eq1} to the character
\begin{equation*}
 \mu_1|\cdot|^{\frac{1}{2}}\otimes   \mu_1|\cdot|^{-\frac{1}{2}}  \otimes  \cdots \otimes\mu_n|\cdot|^{\frac{1}{2}}\otimes  \mu_n|\cdot|^{\frac{1}{2}}  \otimes    \sigma.
\end{equation*}
Using induction by stages, we conclude that $\pi_{\tau\otimes \sigma}$ is the unramified constituent of the representation $\Ind_{P_{2n}}^{H}(\tau^\prime \otimes \sigma)$ of $H$, where
\begin{equation}
\label{eq-prop-Jacquet-Module-Unramified-Component-eq2}
	\tau^\prime= \Ind_{P_{2, \cdots, 2}}^{\GL_{2n}} \mu_1 ({\det}_{\GL_2}) \otimes \cdots \otimes \mu_n ({\det}_{\GL_2})  .
\end{equation}
To prove that the twisted Jacquet module of $\pi_{\tau\otimes \sigma}$ is zero, it suffices to prove that the twisted Jacquet module of $\Ind_{P_{2n}}^{H}(\tau^\prime \otimes \sigma)$ is zero, since the twisted Jacquet functor corresponding to the descent is exact.

If $w_0=y_{\alpha}$ for $\alpha\in k^\times$ such that  $J_{\delta,\alpha}$ is split over $k$, then we obtain $J_{\psi_{\ell,\alpha}}(\Ind_{P_{2n}}^{H}(\tau^\prime \otimes \sigma))=0$ for $\ell\ge n+1$ by applying Theorem~\ref{thm-JacquetModule} (2). Here, we have used the fact that $\tau^\prime_{(\ell)}=0$, $\delta_{\alpha}=1$, and 
${\tau^{\prime}}^{(\ell)}=0$.

If $W_0\in V_0$, then by Theorem~\ref{thm-JacquetModule} (3), we have 
\begin{equation*}
	J_{\psi_{{2n},w_0}}(\Ind_{P_{2n}}^{H}(\tau^\prime \otimes \sigma))\equiv d_{\tau^\prime}\cdot  J_{\psi_{0,w_0}}( \sigma^{\omega_{b,{2n}}^{2n}}),
\end{equation*}
where $d_{\tau^\prime}$ is the dimension of the $\psi$-Whittaker functionals of $\tau^\prime$. Since $\tau^\prime$ is given by \eqref{eq-prop-Jacquet-Module-Unramified-Component-eq2}, we conclude that $d_{\tau^\prime}=0$. Thus, $J_{\psi_{{2n},w_0}}(\Ind_{P_{2n}}^{H}(\tau^\prime \otimes \sigma))=0$. 

Now we assume that  $J_{\delta}$ is split over $k$, so $\wt m=2n+1$, $V_0=0$, and $\delta_{\alpha}=0$ for any choice of $\alpha\in k^\times$.  
In this case, $\pi_{\tau\otimes \sigma}$ is the unramified constituent of the representation of $H$ induced from the character 
\begin{equation}
\label{eq-Jacquet-Module-Unramified-Component-eq3}
 \mu_1|\cdot|^{\frac{1}{2}}\otimes\cdots \otimes\mu_n|\cdot|^{\frac{1}{2}}\otimes \mu_n^{-1}|\cdot|^{\frac{1}{2}}\omega_{\sigma} \otimes \cdots \otimes   \mu_1^{-1}|\cdot|^{\frac{1}{2}}\omega_{\sigma}\otimes (\sigma_1\otimes\sigma_2).
 \end{equation}
If $n$ is even,  $\pi_{\tau\otimes \sigma}$ is the unramified constituent of the representation $\Ind_{P_{2n}}^{H}(\tau^\prime \otimes (\sigma_1\otimes\sigma_2))$ of $H$, where
\begin{equation}
\label{eq-prop-Jacquet-Module-Unramified-Component-eq4}
	\tau^\prime= \Ind_{P_{2, \cdots, 2}}^{\GL_{2n}} \mu_1 ({\det}_{\GL_2}) \otimes \cdots \otimes \mu_n ({\det}_{\GL_2})  .
\end{equation}
If $n$ is odd, $\pi_{\tau\otimes \sigma}$ is the unramified constituent of the representation $\Ind_{P_{2n}}^{H}(\tau^\prime \otimes (\sigma_2\otimes\sigma_1))$ of $H$.
Then by Theorem~\ref{thm-JacquetModule} (1) with $j=2n$, we have $J_{\psi_{\ell,\alpha}}(\pi_{\tau\otimes \sigma})=0$ for $\ell\geq  n+1$. Here we have used the fact that ${\tau^\prime}^{(\ell)}=0$, $\tau^\prime_{(\ell)}=0$ and $\delta_{\alpha}=0$. This completes the proof of Proposition~\ref{prop-Jacquet-Module-Unramified-Component-vanishing}.
\end{proof}

For later use, we also compute the twisted Jacquet module $J_{\psi_{\ell,\alpha}}(\pi_{\tau\otimes \sigma})$ for unramified $\tau$ and $\sigma$ as above, in the case $\ell=n$ and the form $J_{\delta, \alpha}$ is split. Recall that $J_{\delta,\alpha}$ is given by \eqref{eq-J-delta-alpha}, and thus the form $J_{\delta, \alpha}$ being split means that the quadratic form $x^2+\delta y^2+\alpha z^2$ represents 0 over $k$, and the group $H_{n,w_0}=H_{n,\alpha}=\GSpin_{2n+1}$ is split over $k$. 

\begin{proposition}
\label{prop-Jacquet-Module-Unramified-l=n}
Let $\tau$ and $\sigma$ be unramified representations as above. Assume that the form $J_{\delta, \alpha}$ is split. Then we have
\begin{equation}
\label{eq-prop-Jacquet-Module-Unramified-l=n}
J_{\psi_{n,\alpha}}(\pi_{\tau\otimes \sigma})\prec \Ind_{B_{\GSpin(2n+1)}}^{\GSpin(2n+1)}  \mu_1\otimes \cdots \otimes \mu_n \otimes \omega_\sigma,
\end{equation}
where the central character of the representation on the right-hand side of \eqref{eq-prop-Jacquet-Module-Unramified-l=n} is $\omega_\sigma$. 
Here ``$\pi_1\prec\pi_2$'' denotes that $\pi_1$ is a subquotient of $\pi_2$. 
\end{proposition}

\begin{proof}
We first assume that $J_{\delta}$ is non-split over $k$, so $\wt m=2n$ and $V_0\not=0$. As in the proof of Proposition~\ref{prop-Jacquet-Module-Unramified-Component-vanishing}, $\pi_{\tau\otimes \sigma}$ is the unramified constituent of the representation $\Ind_{P_{2n}}^{H}(\tau^\prime \otimes \sigma)$ of $H$, where
\begin{equation*}
	\tau^\prime= \Ind_{P_{2, \cdots, 2}}^{\GL_{2n}} \mu_1 ({\det}_{\GL_2}) \otimes \cdots \otimes \mu_n ({\det}_{\GL_2})  .
\end{equation*}
Since the form $J_{\delta, \alpha}$ is split and $V_0\not=0$, we have $\delta_{\alpha}=1$ for any choice of $\alpha\in k^\times$. By Theorem~\ref{thm-JacquetModule} (2) with $j=2n$ and $\ell=n$, we have 
\begin{equation*}
	J_{\psi_{n,\alpha}}(\Ind_{P_{2n}}^{H}(\tau^\prime \otimes \sigma))\equiv  \ind_{P_{\alpha}^{\prime}}^{H_{n,\alpha}} |\cdot|^{\frac{1-n}{2}} {\tau^\prime}^{(n)}\otimes  J_{N_{0}^\prime, \psi_{0,v_\alpha} ^\prime}(\sigma^{\omega_{b,2n}^n}),
\end{equation*}
noting that $\tau^\prime_{(n)}=0$. Here, $P_{\alpha}^\prime$ is a parabolic subgroup of $H_{n,\alpha}= \GSpin_{2n+1}$ given in \eqref{eq-P_alpha^prime} by
\begin{equation*}
P_{\alpha}^\prime=H_{n,\alpha} \cap \eta_{n,\gamma_\alpha}^{-1} P_{n,2n}^{(w)} \eta_{n,\gamma_\alpha}= \eta_{n,\gamma_\alpha}^{-1} (H_{n, \eta_{n,\gamma_\alpha} \cdot y_\alpha }\cap P_{n,2n}^{(w)}) \eta_{n,\gamma_\alpha}, 
\end{equation*}
with the Levi subgroup of $H_{n, \eta_{n,\gamma_\alpha} \cdot y_\alpha }\cap P_{n,2n}^{(w)}$ isomorphic to $\GL_n\times \GSpin(v_{\alpha}^\perp \cap V_0)$, where $\GSpin(v_\alpha^\perp \cap V_0)$ is isomorphic to $\GL_1$.
Note that 
\begin{equation}
\label{eq-prop-Jacquet-Module-Unramified-l=n-eq1}
|\cdot|^{\frac{1-n}{2}}{\tau'}^{(n)}=\Ind_{B_{\GL_{n}}}^{\GL_{n}}\mu_1\otimes \cdots \otimes \mu_{n}.
\end{equation}
 Also, the Jacquet module $J_{N_{0}^\prime, \psi_{0,v_\alpha} ^\prime}(\sigma^{\omega_{b,2n}^n})$ is the restriction of $\sigma^{\omega_{b,2n}^n}$ to $\GSpin(v_{\alpha}^\perp \cap V_0)\cong \GL_1$, which is just the character $\omega_\sigma$. By induction in stages, we obtain \eqref{eq-prop-Jacquet-Module-Unramified-l=n}.

Now we assume that  $J_{\delta}$ is split over $k$, so $\wt m=2n+1$, $V_0=0$, and $\delta_{\alpha}=0$ for any choice of $\alpha\in k^\times$.
Then $\pi_{\tau\otimes \sigma}$ is the unramified constituent of the representation of $H$ induced from the character 
\begin{equation*}
 \mu_1|\cdot|^{\frac{1}{2}}\otimes\cdots \otimes\mu_n|\cdot|^{\frac{1}{2}}\otimes \mu_n^{-1}|\cdot|^{\frac{1}{2}}\omega_{\sigma} \otimes \cdots \otimes   \mu_1^{-1}|\cdot|^{\frac{1}{2}}\omega_{\sigma}\otimes (\sigma_1\otimes\sigma_2).
 \end{equation*}
If $n$ is even,  $\pi_{\tau\otimes \sigma}$ is the unramified constituent of the representation $\Ind_{P_{2n}}^{H}(\tau^\prime \otimes (\sigma_1\otimes\sigma_2))$ of $H$, where
\begin{equation*}
	\tau^\prime= \Ind_{P_{2, \cdots, 2}}^{\GL_{2n}} \mu_1 ({\det}_{\GL_2}) \otimes \cdots \otimes \mu_n ({\det}_{\GL_2})  .
\end{equation*}
If $n$ is odd, $\pi_{\tau\otimes \sigma}$ is the unramified constituent of the representation $\Ind_{P_{2n}}^{H}(\tau^\prime \otimes (\sigma_2\otimes\sigma_1))$ of $H$.
In either case, $\pi_{\tau\otimes \sigma}$ is the unramified constituent of $\Ind_{P_{2n}}^{H}(\tau^\prime \otimes \sigma^{\omega_{b,2n}^n})$.
By Theorem~\ref{thm-JacquetModule} (1) with $j=2n$ and $\ell=n$, we have 
\begin{equation*}
	J_{\psi_{n,\alpha}} (\Ind_{P_{2n}}^{H}(\tau^\prime \otimes \sigma^{\omega_{b,2n}^n})) \equiv \ind_{P^\prime_{n,n}}^{H_{n,\alpha}} |\cdot|^{\frac{1-n}{2}} {\tau^\prime}^{(n)} \otimes J_{N_{0}^\prime, \psi_{0,\alpha} ^\prime}(\sigma). 
\end{equation*}
Here, $P_{n,n}^\prime$ is a parabolic subgroup of $H_{n,\alpha}$ whose Levi subgroup is isomorphic to $\GL_n\times \GL_1$. 
The Jacquet module $J_{N_{0}^\prime, \psi_{0,\alpha} ^\prime}(\sigma)$ is the restriction of $\sigma$ to the center of $H_{n,\alpha}$, namely $\omega_\sigma$. 
Now we use \eqref{eq-prop-Jacquet-Module-Unramified-l=n-eq1} and induction in stages again to obtain \eqref{eq-prop-Jacquet-Module-Unramified-l=n}.
\end{proof}

\subsection{Vanishing property and cuspidality of the twisted automorphic descent}

We now go back to the global setting. As an immediate corollary to Proposition~\ref{prop-Jacquet-Module-Unramified-Component-vanishing}, we obtain the vanishing of the ``deeper" twisted automorphic descent for $\ell\ge n+1$. 

\begin{corollary}[Theorem~\ref{thm-Main} (1)]
\label{cor-MainThm-Part1}
Let the notation be as in Theorem~\ref{thm-Main}. For all $\ell>n$, the $\psi_{\ell,w_0}$-Fourier coefficients of the residual representation $\CE_{\tau\otimes\sigma}$ are zero for all anisotropic vectors $w_0\in W_{\ell}$. 
\end{corollary}

\begin{proof}
By Remark~\ref{remark-w_0-local}, it suffices to consider $w_0=y_{\alpha}$ with $\alpha\in F^\times$ or $w_0\in V_0$. Let $v$ be a finite place of the number field $F$, such that all data involved in the $\psi_{\ell,w_0}$-Fourier coefficients of the residual representation $\CE_{\tau\otimes\sigma}$ are unramified. Let $\mathcal{I}_{\tau_v\otimes\sigma_v}$ be the unramified local component of the residual representation $\CE_{\tau\otimes\sigma}$ at the place $v$. We take any integer $\ell>n$. Suppose that the residual representation $\CE_{\tau\otimes\sigma}$ has a non-zero $\psi_{\ell,w_0}$-Fourier coefficient. Then the corresponding local twisted Jacquet module $J_{\psi_{\ell,w_0}}(\mathcal{I}_{\tau_v\otimes\sigma_v})$ is non-zero, which is a contradiction of Proposition~\ref{prop-Jacquet-Module-Unramified-Component-vanishing}. 
\end{proof}

We now consider the case $\ell=n$ and $w_0=y_{\alpha}$, and prove that the twisted automorphic descent on $\GSpin(y_{\alpha}^\perp \cap W_{n})(\A)=\GSpin^{\delta,\alpha}_{2n+1}(\A)$ is cuspidal. 

\begin{proposition}[Theorem~\ref{thm-Main} (2)]
\label{prop-cuspidality}
Let the notation be as in Theorem~\ref{thm-Main}. For any square class $\alpha$ in $F^\times$, the representation $\TD_{\psi_{n,\alpha}}(\CE_{\tau\otimes\sigma})$ of $\GSpin^{\delta,\alpha}_{2n+1}(\A)$ is cuspidal.
\end{proposition}

\begin{proof}
Let $\wt m_{n,\alpha}$ be the Witt index of the $(2n+1)$-dimensional space $y_{\alpha}^\perp\cap W_{n}$. For any integer $1\le p \le \wt m_{n,\alpha}$, we let $P_p^*$ be the standard maximal parabolic subgroup of $H_{n,\alpha}$, which preserves the totally isotropic subspace  $V_{n,p}^+=y_{\alpha}^\perp \cap \Span\{e_{n+1}, \cdots, e_{n+p}\}\subset y_{\alpha}^\perp\cap W_{n}$. Let $U_p^*$ be the unipotent radical of $P_p^*$. Given a Bessel coefficient $f^{\psi_{n,\alpha}}$ with $f\in \CE_{\tau\otimes\sigma}$, we denote
\begin{equation}
\label{eq-prop-cuspidality-c_p}
c_p(f^{\psi_{n,\alpha}})=\int_{U_p^*(F)\backslash U_p^*(\A)} f^{\psi_{n,\alpha}} (u)du.	
\end{equation}
This is the constant term of $f^{\psi_{n,\alpha}}$ along the unipotent radical $U_p^*$. To prove that $f^{\psi_{n,\alpha}}$ is cuspidal, it suffices to prove that $c_p(f^{\psi_{n,\alpha}})$ is zero for all $1\le p \le \wt m_{n,\alpha}$.

By Proposition~\ref{prop-ConstantTerm-prop2} below, we can prove that, if
\begin{equation}
\label{eq-prop-cuspidality-1}
( f^{U_{p-i}} )^{\psi_{n+i,\alpha}} = 0 \quad \text{ for all }0\le i \le p-1,	
\end{equation}
then $c_p(f^{\psi_{n,\alpha}})$ can be expressed as a sum of integrals of $f^{\psi_{n+p,\alpha}}$.
Here, $U_j$ is the unipotent radical of the standard parabolic subgroup $P_j$ of $H^\delta$, and $f^{U_j}$ is the constant term of $f$ along $U_j$. 
By Corollary~\ref{cor-MainThm-Part1}, $f^{\psi_{n+p,\alpha}}=0$ for all $1\le p \le \wt m_{n,\alpha}$, which implies that $c_p(f^{\psi_{n,\alpha}})$ is zero for all $1\le p \le \wt m_{n,\alpha}$, as desired.

It suffices to prove \eqref{eq-prop-cuspidality-1}. We consider the Eisenstein series $E(g,s,\phi_{\tau\otimes\sigma})$, which produces the residual representation $\CE_{\tau\otimes\sigma}$. Recall that $P$ is the standard parabolic subgroup whose Levi subgroup is isomorphic to $\GL_{2n}\times \GSpin(V_0)$.  By \cite[II.1.7]{MoglinWaldspurger1995}, the constant term of the Eisenstein series $E(g,s,\phi_{\tau\otimes\sigma})$ along the unipotent radical of a standard parabolic subgroup $\wt P$ is always zero unless $\wt P=P$. It follows that the constant term $E^{U_{p-i}}(g,s,\phi_{\tau\otimes\sigma})$ is always zero for any $1\le p \le \wt m_{n,\alpha}$, $0\le i \le p-1$. Thus, we have \eqref{eq-prop-cuspidality-1} as desired. This finishes the proof of Proposition~\ref{prop-cuspidality}. 
\end{proof}

\section{Certain Fourier coefficient of the residual representation}
\label{section-certain-fourier-coef}
\subsection{Root exchange lemma}

In this section, we recall the process of root exchanges, following \cite[\S 7.1]{GinzburgRallisSoudry2011}; see also \cite[\S 6.1]{AsgariCogdellShahidi2024}. As explained in \cite[\S 6.1]{AsgariCogdellShahidi2024}, the proof of the root exchange lemma is the same as that in \cite[\S 7.1]{GinzburgRallisSoudry2011}. We therefore refer the reader to \cite[\S 7.1]{GinzburgRallisSoudry2011} for the proofs of Lemma~\ref{lemma-RootExchange}, Corollary~\ref{corollary-RootExchange1} and Corollary~\ref{corollary-RootExchange2} below. 

Temporarily, we introduce the following notation, following \cite[\S 6.1]{AsgariCogdellShahidi2024}.
Let $H$ be a  connected reductive algebraic $k$-group, such as $H^\delta$ in our setup. 
Let $U<H$ be a maximal unipotent $F$-subgroup.  Suppose $C<U$ is a $F$-subgroup of $U$, and $\psi=\psi_C$ is a non-trivial character 
of $C(F)\bs C(\A)$.  Suppose we have two other $F$-subgroups $X$ and $Y$ of $U$ such that the following six axioms hold:
\begin{enumerate}
\item $X$ and $Y$ normalize $C$. 
\item $X\cap C$ is normal in $X$ and $(X\cap C)\bs X$ is abelian, and similarly $Y\cap C$ is normal in $Y$ and $(Y\cap C)\bs Y$ is abelian. 
\item When $X(\A)$ and $Y(\A)$ act on $C(\A)$ by conjugation, they preserve $\psi_C$. 
\item $\psi_C$ is trivial on $(X\cap C)(\A)$ and $(Y\cap C)(\A)$. 
\item The commutator $(X,Y)\subset C$. (Recall that $(x,y)=x^{-1}y^{-1}xy$ and $(X,Y)=\{(x,y):x\in X, y\in Y\}$.)
\item The pairing of $(X\cap C)(\A)\bs X(\A) \times (Y\cap C)(\A)\bs Y(\A)$ given by 
$$(x,y)\mapsto \psi_C((x,y))$$ is bilinear and non-degenerate, and identifies 
\begin{equation*}
(Y\cap C)(F)\bs Y(F)\simeq \left[(X(F)(X\cap C)(\A))\bs X(\A)\right]^{\wedge}  
\end{equation*} 
and 
\begin{equation*}
(X\cap C)(F)\bs X(F)\simeq \left[(Y(F)(Y\cap C)(\A))\bs Y(\A)\right]^\wedge.
\end{equation*} 
\end{enumerate}
Note that the first five conditions imply that, for each fixed $y\in Y(\A)$, the map $x\mapsto \psi_C((x,y))$ defines a character of 
$X(\A)$, which is trivial on $(X\cap C)(\A)$. (Similarly, for a fixed $x\in X(\A)$, the map $y\mapsto \psi_C((x,y))$ defines a character of $Y(\A)$ 
which is trivial on $(Y\cap C)(\A)$.)

Let $B=CY=YC$, $D=CX=XC$ and $A=CXY$. We represent the above (1)-(6) conditions in the following diagram:
\begin{equation}
\label{eq-RootExchange-diagram}
\begin{tikzcd}[row sep=3em, column sep=4em]
& A & \\
B = CY \arrow[ur] & & D = CX \arrow[ul] \\
& C \arrow[ul] \arrow[ur] &
\end{tikzcd}
\end{equation}
and the following quadruple
\begin{equation}
\label{eq-RootExchange-quadruple}
(C, \psi_C, X, Y).
\end{equation}
We extend $\psi_C$ to a character $\psi_B$ of $B(\A)$ which is trivial on $B(F)$ by making it trivial on 
$Y(\A)$. Similarly, we extend $\psi_C$ to a character $\psi_D$ of $D(\A)$ which is trivial on $D(F)$ by making it trivial on $X(\A)$. 

\begin{lemma}[Root exchange] 
\label{lemma-RootExchange}
Let $f$ be an automorphic function on $H(\A)$ which is smooth and of uniform moderate growth. Then
\[
\int_{B(F)\bs B(\A)}f(v)\psi_B(v)^{-1} dv=\int_{(Y\cap C)(\A)\bs Y(\A)}\int\limits_{D(F)\bs D(\A)} f(uy)\psi_D(u)^{-1}\ du dy,
\]
where the right hand side converges in the sense that
\[
\int_{(Y\cap C)(\A)\bs Y(\A)} \left| \int_{D(F)\bs D(\A)}f(uyh)\psi_D(u)^{-1}\ du \right| dy <\infty
\]
uniformly as $h$ varies in any compact subset of $H(\A)$.
\end{lemma}

\begin{corollary}
\label{corollary-RootExchange1}
Let $f$ be an automorphic function on $H(\A)$ which is smooth and of uniform moderate growth. Then
\begin{equation*}
\int_{B(F)\backslash B(\A)} f(va)\psi_B^{-1}(v)dv \equiv 0, \quad \forall a\in A(\A),	
\end{equation*}
if and only if 
\begin{equation*}
\int_{D(F)\backslash D(\A)} f(ua)\psi_D^{-1}(u)du \equiv 0, \quad \forall a\in A(\A).
\end{equation*}
\end{corollary}

\begin{corollary}
\label{corollary-RootExchange2}
Let $f$ be an automorphic function on $H(\A)$ which is smooth and of uniform moderate growth. Then there exist smooth, uniform moderate growth, 
automorphic functions  
$f_1,\dots, f_r$ and Schwartz functions $\phi_1,\dots,\phi_r\in \mathcal S((Y\cap C)(\A) \bs Y(\A))$ such that, for all 
$y\in (Y\cap C)(\A)\bs Y(\A)$, we have
\begin{equation*}
\int_{D(F)\bs D(\A)} f(uy)\psi^{-1}_D(u)\ du=\sum_{i=1}^r \phi_i(y)\int_{D(F)\bs D(\A)} f_i(uy)\psi^{-1}_D(u)\ du.
\end{equation*}
\end{corollary}

\subsection{Constant term of the Bessel coefficient of the residual representation}
The goal of this section is to prove a formula for the constant term of the Bessel coefficient of the residual representation, and to finish the proof of Proposition~\ref{prop-cuspidality}. We continue with the setup as in the proof of  Proposition~\ref{prop-cuspidality}.
Recall that $c_p(f^{\psi_{n,\alpha}})$ is given in \eqref{eq-prop-cuspidality-c_p}.
For any $1\le j \le 2n$ and $g\in \GL(V_j^+)(\A)=\GL_j(\A)$, we denote by $\hat{g}$ the lift of $g$ into the Levi $\GL(V_j^+)(\A)\times\GSpin(W_j)(\A)$ of $P_j(\A)$, described in Lemma~\ref{lemma-Levi-lifting}. 
We denote 
\begin{equation*}
U_{n+p}^i =\left\{ \begin{pmatrix} I_{p+n-i}& * \\ & z \end{pmatrix}^\wedge: z\in Z_i \right\}\cdot U_{n+p}\subset N_{n+p}, \quad 1\le i \le p+n-1,
\end{equation*}
and
\begin{equation*}
\begin{split}
\mathcal{L} &=\mathcal{L}_{p,n} =\left\{ \begin{pmatrix} I_p&\\ * &I_{n}\end{pmatrix}^\wedge \right\}, \\
\mathcal{L}^{(i)} &= \left\{ \begin{pmatrix} I_p&\\ \lambda &I_{n}\end{pmatrix}^\wedge \in \mathcal{L}: \lambda=\begin{pmatrix} \lambda_1 \\ \vdots \\ \lambda_n \end{pmatrix}, \lambda_j=0, \  \forall j\not=n-i \right\}, \quad i=0, \cdots, n-1,\\
\mathcal{L}_i &= \left\{ \begin{pmatrix} I_p&\\ \lambda &I_{n}\end{pmatrix}^\wedge \in \mathcal{L}: \lambda=\begin{pmatrix} \lambda_1 \\ \vdots \\ \lambda_n \end{pmatrix}, \lambda_{n-i}=\lambda_{n-i+1}=\cdots =\lambda_{n}=0 \right\}, \quad i=0, \cdots, n-1.
\end{split}
\end{equation*}
Note that
\begin{equation*}
\mathcal{L}_{n-1}\subset \mathcal{L}_{n-2} \subset \cdots \subset \mathcal{L}_{0}.
\end{equation*}
For $1\le i \le n-1$, denote
\begin{equation*}
\mathcal{L}^{(i,i-1, \cdots, 0)}= \mathcal{L}^{(i)} \mathcal{L}^{(i-1)}  \cdots  \mathcal{L}^{(0)}.
\end{equation*}
For each $1\le i \le p$, denote 
\begin{equation*}
P_{p-i}^1=	\left\{ \begin{pmatrix} g & x\\& z\end{pmatrix}\in \GL_p: z\in Z_i \right\}.
\end{equation*}
Denote
\begin{equation*}
	\beta =\beta_{p,n}= \begin{pmatrix} &I_p \\ I_n &\end{pmatrix}^\wedge . 
\end{equation*}

\begin{proposition}
\label{prop-ConstantTerm-prop1}
Let $f$ be an automorphic function on $H^{\delta}(\A)$ which is smooth and of uniform moderate growth. Let $1\le p \le \wt m_{n,\alpha}$. We have the following formulas for $c_p(f^{\psi_{n,\alpha}})$.
\begin{enumerate}
\item We have
\begin{equation}
\label{eq-prop-ConstantTerm-prop1}
\begin{split}
&c_p(f^{\psi_{n,\alpha}})\\
=&\sum_{\gamma\in P_{p-1}^1(F)\backslash \GL_p(F)} \int_{\mathcal{L}(\A)} \int_{U_{n+p}^n(F)\backslash U_{n+p}^n(\A)} f(u\hat{\gamma}  \lambda \beta) \psi_{n+p,\alpha}^{-1}(u)dud\lambda+ \int_{\mathcal{L}(\A)} (f^{U_p})^{\psi_{n,\alpha}}(\lambda \beta) d\lambda. 
\end{split}
\end{equation}

\item We have
\begin{equation}
\label{eq-prop-ConstantTerm-prop1-2}
\begin{split}
&c_p(f^{\psi_{n,\alpha}})\\
=&\sum_{\gamma\in Z_p(F)\backslash \GL_p(F)} \int_{\mathcal{L}(\A)}   f^{\psi_{n+p,\alpha}}(  \hat{\gamma}  \lambda \beta) d\lambda+ \sum_{i=0}^{p-1} \sum_{\gamma\in P_{p-i}^1(F)\backslash \GL_p(F)}  \int_{\mathcal{L}(\A)} (f^{U_{p-i}})^{\psi_{n+i,\alpha}}(\hat{\gamma} \lambda \beta) d\lambda \\
=&\sum_{i=0}^{p} \sum_{\gamma\in P_{p-i}^1(F)\backslash \GL_p(F)}  \int_{\mathcal{L}(\A)} (f^{U_{p-i}})^{\psi_{n+i,\alpha}}(\hat{\gamma} \lambda \beta) d\lambda .
\end{split}
\end{equation}
\end{enumerate}
\end{proposition}

\begin{proof}
Note that $\beta\in H^\delta(F)$, and $f$ is left-invariant under $H^{\delta}(F)$. Thus,
\begin{equation*}
c_p(f^{\psi_{n,\alpha}})=\int_{U_p^*(F)\backslash U_p^*(\A)} \int_{N_n(F)\backslash N_n(\A)} f(\beta v r \beta^{-1} \beta )\psi_{n,\alpha}^{-1}(v)dvdr. 	
\end{equation*}
Denote $S=\beta U_p^* N_n \beta^{-1}$. Then elements of $S$ take the following form
\begin{equation}
\label{eq-prop-ConstantTerm-prop1-S}
\beta v r \beta^{-1}= \begin{pmatrix}
 I_p & 0 &x & d &y\\
 u &z &a &e &d^\prime\\
 &&I_{2n+2-2p} & a^\prime &x^\prime\\
 &&&z^* &0\\
 &&&u^\prime & I_p	
 \end{pmatrix}
	=s(z;u,a,d,e;x,y)
\end{equation}
where $z\in Z_{n}$, and $x\cdot y_{\alpha}^{(n+p)}=0$, where $y_{\alpha}^{(n+p)}$ is the vector $y_{\alpha}$, viewed as a column vector in the vector space $W_{n+p}\cong F^{2n+2-2p}$. Note that
\begin{equation*}
\begin{split}
\beta U_p^* \beta^{-1}  &= \{ s(I_n;0,0,0,0;x,y)\in H^\delta: 	x\cdot y_{\alpha}^{(n+p)}=0 \}, \\
\beta N_n \beta^{-1}  &= \{ s(z;u,a,d,e;0,0)\in H^\delta: 	z\in Z_n \},\\
\psi_{n,\alpha}(v)&= \psi(z_{1,2}+\cdots+ z_{n-1,n} + a_{n}\cdot y_{\alpha}^{(n+p)}), \quad a_n=\text{$n$-th row of }a.
\end{split}
\end{equation*}
Let $\psi_S$ be the character of $S(\A)$, given by 
\begin{equation*}
\psi_S(\beta v r \beta^{-1})= 	\psi(z_{1,2}+\cdots+ z_{n-1,n} + a_{n}\cdot y_{\alpha}^{(n+p)}).
\end{equation*}
This is obtained by first trivially extending the character $\psi_{n,\alpha}$ of $N_n(\A)$ to $U_p^*(\A)N_n(\A)$, and then defining, for $s\in S(\A)$, $\psi_S(s)=\psi_{n,\alpha}(\beta^{-1}s\beta)$. 
Let $\wt S$ be the following subgroup of $S$, given by
\begin{equation*}
\wt S=\{s(I_n;0,a,d,e;x,y)\in S\}.	
\end{equation*}
Note that 
\begin{equation*}
\mathcal{L}=\{s(I_n;u,0,0,0;0,0)\in S\}, \quad Z_n\cong\{s(z;0,0,0,0;0,0)\in S\},
\end{equation*}
and $\psi_S$ is trivial on $\mathcal{L}(\A)$.
Thus,
\begin{equation*}
c_p(f^{\psi_{n,\alpha}})=\int_{Z_n(F)\backslash Z_n(\A)} \int_{\mathcal{L}(F)\backslash \mathcal{L}(\A)} \int_{\wt S (F)\backslash \wt S(\A)} f(s \lambda z \beta )\psi_{S}^{-1}(sz) dsd\lambda dz. 	
\end{equation*}

Denote
\begin{equation*}
J=	\{ s(I_n;0,0,0,0;x,y)\in H^\delta \}, \quad 
\end{equation*}
and 
\begin{equation*}
J_0=J\cap S=\beta U_p^* \beta^{-1}  = \{ s(I_n;0,0,0,0;x,y)\in H^\delta: 	x\cdot y_{\alpha}^{(n+p)}=0 \}.	
\end{equation*}
Then the quotient $J_0\backslash J$ is naturally identified with $F^p$, and hence it is abelian. 
The groups
\begin{equation*}
C=\wt S, \quad  Y=\mathcal{L}^{(0)}, \quad X=J,  \quad B= \wt S \mathcal{L}^{(0)}, \quad D=	\wt S J
\end{equation*}
satisfy the conditions represented in \eqref{eq-RootExchange-diagram}. Note that $D=\wt S J=U_{n+p}$ and $\psi_D=\psi_{n+p,\alpha}\Big|_{D}$. By applying Lemma~\ref{lemma-RootExchange}, we obtain that
\begin{equation*}
\int_{\mathcal{L}^{(0)}(F)\backslash \mathcal{L}^{(0)}(\A)} \int_{\wt S(F)\backslash \wt S(\A)} f(s\lambda h)\psi_S^{-1}(s)dsd\lambda =\int_{\mathcal{L}^{(0)}(\A)} \int_{U_{n+p}(F)\backslash U_{n+p}(\A)} f(u\lambda h)\psi_{n+p,\alpha}^{-1}(u)dud\lambda. 	
\end{equation*}
Since $\mathcal{L}=\mathcal{L}_0 \mathcal{L}^{(0)}$, we have
\begin{equation*}
\begin{split}
c_p(f^{\psi_{n,\alpha}})&=\int_{Z_n(F)\backslash Z_n(\A)} \int_{\mathcal{L}_{0}(F)\backslash \mathcal{L}_{0}(\A)} \int_{\mathcal{L}^{(0)}(\A)} \int_{U_{n+p}(F)\backslash U_{n+p}(\A)} f(u\lambda \lambda^{(0)} z\beta )\psi_{n+p,\alpha}^{-1}(uz)dud\lambda  d\lambda^{(0)} dz.
\end{split}
\end{equation*}
Note that  $\mathcal{L}_0=\mathcal{L}^{(1)}\mathcal{L}_1$, and
\begin{equation*}
Z_n=Z^{(1)} Z_{n-1},	
\end{equation*}
where 
\begin{equation*}
Z^{(1)}=\left\{\begin{pmatrix} I_{n-1}&*\\&1\end{pmatrix}\right\} \cong\{s\left( \begin{pmatrix} I_{n-1}&*\\&1\end{pmatrix};0,0,0,0;0,0\right)\in S \}, 	
\end{equation*}
and 
\begin{equation*}
	Z_{n-1}=\left\{\begin{pmatrix} z&\\&1\end{pmatrix}:z\in Z_{n-1}\right\} \cong\{s\left( \begin{pmatrix} z&\\&1\end{pmatrix};0,0,0,0;0,0\right)\in S:z\in Z_{n-1} \}.
\end{equation*}
Also note that $Z_{n}$ normalizes $\mathcal{L}_{0}$.
Then
\begin{equation}
\label{eq-prop-ConstantTerm-prop1-eq2}
c_p(f^{\psi_{n,\alpha}})=\int_{\mathcal{L}^{(0)}(\A)} \int_{Z_{n-1}(F)\backslash Z_{n-1}(\A)} \int_{\mathcal{L}_1(F)\backslash \mathcal{L}_1(\A)} I_1(f,\psi)( \lambda z \lambda^{(0)} \beta) \psi_{n+p,\alpha}^{-1}(z) d\lambda dz d\lambda^{(0)},	
\end{equation}
where
\begin{equation*}
	I_1(f,\psi)( h) =\int_{\mathcal{L}^{(1)}(F)\backslash \mathcal{L}^{(1)}(\A)} \int_{Z^{(1)}(F)\backslash Z^{(1)}(\A)} \int_{U_{n+p}(F)\backslash U_{n+p}(\A)} f( u z^{(1)} \lambda^{(1)} h )\psi_{n+p,\alpha}^{-1}(u z^{(1)}) du dz^{(1)} d\lambda^{(1)}.
\end{equation*}
By Corollary~\ref{corollary-RootExchange2}, there is a $f_0\in W(f)$, the closed cyclic module generated by $f$, such that, for all $z\in Z_n(\A)$, $\lambda_0\in \mathcal{L}_0(\A)$, we have 
\begin{equation}
\label{eq-prop-ConstantTerm-prop1-eq3}
\int_{\mathcal{L}^{(0)}(\A)}	\int_{U_{n+p}(F)\backslash U_{n+p}(\A)} f( u \lambda_0 z \lambda \beta )\psi_{n+p,\alpha}^{-1}(u  ) du   d\lambda = \int_{U_{n+p}(F)\backslash U_{n+p}(\A)} f_0( u \lambda_0 z  \beta )\psi_{n+p,\alpha}^{-1}(u  ) du.
\end{equation}
Then
\begin{equation}
\label{eq-prop-ConstantTerm-prop1-eq4}
c_p(f^{\psi_{n,\alpha}})= \int_{Z_{n-1}(F)\backslash Z_{n-1}(\A)} \int_{\mathcal{L}_1(F)\backslash \mathcal{L}_1(\A)} I_1(f_0,\psi)( \lambda z  ) \psi_{n+p,\alpha}^{-1}(z) d\lambda dz.
\end{equation}

The groups
\begin{equation*}
C=U_{n+p}Z^{(1)}, \quad Y=\mathcal{L}^{(1)}, \quad X=\left\{ \begin{pmatrix} I_p & 0 &x\\ &I_{n-1}&0\\&&1\end{pmatrix}^\wedge \right\}	
\end{equation*}
and the character 
\begin{equation*}
\psi_C=\psi_{n+p,\alpha}\Big|_{C}	
\end{equation*}
satisfy the conditions represented by \eqref{eq-RootExchange-diagram}. By Lemma~\ref{lemma-RootExchange}, we obtain
\begin{equation}
\label{eq-prop-ConstantTerm-prop1-eq5}
I_1(f,\psi)(h)=\int_{\mathcal{L}^{(1)}(\A)}  \int_{U_{n+p}(F)\backslash U_{n+p}(\A)} f( u \lambda^{(1)} h) \psi_{n+p,\alpha}^{-1}(u) du d\lambda^{(1)}.	
\end{equation}
By \eqref{eq-prop-ConstantTerm-prop1-eq2} and  \eqref{eq-prop-ConstantTerm-prop1-eq3}, we obtain that
\begin{equation}
\label{eq-prop-ConstantTerm-prop1-eq6}
\begin{split}
&c_p(f^{\psi_{n,\alpha}})\\
=&\int_{\mathcal{L}^{(1,0)}(\A)} \int_{Z_{n-1}(F)\backslash Z_{n-1}(\A)} \int_{\mathcal{L}_1(F)\backslash \mathcal{L}_1(\A)} \int_{U_{n+p}^1(F)\backslash U_{n+p}^1(\A)} f( u \lambda z \lambda^\prime \beta) \psi_{n+p,\alpha}^{-1}(uz) dud\lambda dz d\lambda^\prime,	
\end{split}
\end{equation}
where $\mathcal{L}^{(1,0)}=\mathcal{L}^{(1)} \mathcal{L}^{(0)}$. By Corollary~\ref{corollary-RootExchange2} again as in \eqref{eq-prop-ConstantTerm-prop1-eq3}, we obtain that there is $f_1\in W(f)$, such that
\begin{equation}
\label{eq-prop-ConstantTerm-prop1-eq7}
\begin{split}
c_p(f^{\psi_{n,\alpha}})= \int_{Z_{n-1}(F)\backslash Z_{n-1}(\A)} \int_{\mathcal{L}_1(F)\backslash \mathcal{L}_1(\A)} \int_{U_{n+p}^1(F)\backslash U_{n+p}^1(\A)} f_1( u \lambda z  ) \psi_{n+p,\alpha}^{-1}(uz) dud\lambda dz .	
\end{split}
\end{equation}

For each $1\le i \le n-1$, denote
\begin{equation*}
Z_{n-i}=\left\{ \begin{pmatrix} z & \\ &I_{i}\end{pmatrix}\in Z_n\right\} \cong\{s\left( \begin{pmatrix} z&\\&I_i\end{pmatrix};0,0,0,0;0,0\right)\in S:z\in Z_{n-i} \}.   	
\end{equation*}
Then $Z_{n-i}=Z^{(i+1)} Z_{n-i-1}$ with
\begin{equation*}
Z^{(i+1)}=	\left\{ \begin{pmatrix} I_{n-i-1} & x&0\\ &1&0\\ & &I_{i}\end{pmatrix}\in Z_n\right\} 
\end{equation*}
Also note that $\mathcal{L}_i=\mathcal{L}^{(i+1)}\mathcal{L}_{i+1}$. Applying the previous arguments inductively and using
\begin{equation*}
C=U_{n+p}^{i} Z^{(i+1)},\quad  Y=\mathcal{L}^{(i+1)}, \quad X=\left\{ \begin{pmatrix} I_p &0&x&\\ &I_{n-i-1} &0&\\&&1&\\ &&&I_i\end{pmatrix} \right\},	
\end{equation*}
and the character $\psi_C=\psi_{n+p,\alpha}\Big|_{C}$, we obtain that, for $1\le i \le n-1$, we have
\begin{equation}
\label{eq-prop-ConstantTerm-prop1-eq8}
\begin{split}
&c_p(f^{\psi_{n,\alpha}})\\
=&\int_{\mathcal{L}^{(i,i-1, \cdots, 0)}(\A)} \int_{Z_{n-i}(F)\backslash Z_{n-i}(\A)} \int_{\mathcal{L}_i(F)\backslash \mathcal{L}_i(\A)} \int_{U_{n+p}^i(F)\backslash U_{n+p}^i(\A)} f( u \lambda z \lambda^\prime \beta) \psi_{n+p,\alpha}^{-1}(uz) dud\lambda dz d\lambda^\prime,	
\end{split}
\end{equation}
where $\mathcal{L}^{(i, i-1, \cdots,0)}=\mathcal{L}^{(i)} \mathcal{L}^{(i-1)} \cdots \mathcal{L}^{(0)}$. We also obtain that there is $f_i\in W(f)$, such that
\begin{equation}
\label{eq-prop-ConstantTerm-prop1-eq9}
\begin{split}
c_p(f^{\psi_{n,\alpha}})= \int_{Z_{n-i}(F)\backslash Z_{n-i}(\A)} \int_{\mathcal{L}_i(F)\backslash \mathcal{L}_i(\A)} \int_{U_{n+p}^i(F)\backslash U_{n+p}^i(\A)} f_i( u \lambda z  ) \psi_{n+p,\alpha}^{-1}(uz) dud\lambda dz .	
\end{split}
\end{equation}
For $i=n-1$, \eqref{eq-prop-ConstantTerm-prop1-eq8} becomes
\begin{equation}
\label{eq-prop-ConstantTerm-prop1-eq10}
\begin{split}
c_p(f^{\psi_{n,\alpha}})
=\int_{\mathcal{L}(\A)}     \int_{U_{n+p}^{n-1}(F)\backslash U_{n+p}^{n-1}(\A)} f( u \lambda   \beta) \psi_{n+p,\alpha}^{-1}(u ) dud\lambda.	
\end{split}
\end{equation}
Similar to \eqref{eq-prop-ConstantTerm-prop1-eq3}, we also obtain that, for all $g\in P_{p-1}^1(\A)$, we have
\begin{equation}
\label{eq-prop-ConstantTerm-prop1-eq11}
\int_{\mathcal{L}(\A)}	\int_{U_{n+p}^{n-1}(F)\backslash U_{n+p}^{n-1}(\A)} f( u \hat{g}   \lambda \beta )\psi_{n+p,\alpha}^{-1}(u  ) du   d\lambda = \int_{U_{n+p}^{n-1}(F)\backslash U_{n+p}^{n-1}(\A)} f_{n-1}( u \hat{g} )\psi_{n+p,\alpha}^{-1}(u  ) du.
\end{equation}
Denote $\varphi=f_{n-1}$. Let $x\in \A^p$ be a column vector and define
\begin{equation*}
u_x= \begin{pmatrix} I_p & x \\ & 1\end{pmatrix}^\wedge\in N_{n+p}(\A).
\end{equation*}
The function
\begin{equation*}
x\mapsto \phi(x):=\int_{U_{n+p}^{n-1}(F)\backslash U_{n+p}^{n-1}(\A)} \varphi( u u_x)\psi_{n+p,\alpha}^{-1}(u)du
\end{equation*}
is a smooth function on $F^p\backslash \A^p$, and hence has a Fourier expansion in terms of the characters of $F^p\backslash \A^p$. Any character of $u_x$ which is trivial on the rational points, is of the form $u_x\mapsto \psi({}^t \eta \cdot x)$, where $\eta\in F^p$ is a column vector. 
The Fourier coefficient of $\phi$ corresponding to the trivial character (i.e., when $\eta=0$) is 
\begin{equation*}
\int_{F^p\backslash \A^p} 	\int_{U_{n+p}^{n-1}(F)\backslash U_{n+p}^{n-1}(\A)} \varphi( u u_x) \psi_{n+p,\alpha}^{-1}(u)dudx = (\varphi^{U_p})^{\psi_{n,\alpha}}(1).
\end{equation*}
Now we consider $\eta\not=0$ and write ${}^t \eta=(0, \cdots, 0, 1)\gamma$, with $\gamma\in P_{p-1}^1(F)\backslash \GL_p(F)$. The Fourier coefficient of $\phi$ corresponding to this $\eta$ is
\begin{equation}
\label{eq-prop-ConstantTerm-prop1-eq12}
\begin{split}
\int_{F^p\backslash \A^p} 	\int_{U_{n+p}^{n-1}(F)\backslash U_{n+p}^{n-1}(\A)} \varphi( u u_x) \psi_{n+p,\alpha}^{-1}(u ) \psi^{-1}( (0, \cdots, 0,1)\gamma x)dudx.
\end{split}
\end{equation}
Since $\varphi$ is automorphic, we have
\begin{equation*}
\varphi( u u_x)= \varphi( \hat{\gamma} u u_x) = 	\varphi( \hat{\gamma} u \hat{\gamma}^{-1}  u_{\gamma x} \hat{\gamma}).
\end{equation*}
Also, $\psi^{-1}( (0, \cdots, 0,1)\gamma x)=\psi_{n+p,\alpha}^{-1}( u_{\gamma x})$.
Thus, \eqref{eq-prop-ConstantTerm-prop1-eq12} becomes
\begin{equation*}
	\int_{F^p\backslash \A^p} 	\int_{U_{n+p}^{n-1}(F)\backslash U_{n+p}^{n-1}(\A)} \varphi(  \hat{\gamma} u \hat{\gamma}^{-1}  u_{\gamma x} \hat{\gamma} ) \psi_{n+p,\alpha}^{-1}(u u_{\gamma x})dudx
\end{equation*}
We now make a change of variables $u\mapsto \hat{\gamma}^{-1}u\hat{\gamma}$, $x\mapsto \gamma^{-1}x$, 
to obtain that \eqref{eq-prop-ConstantTerm-prop1-eq12} is equal to
\begin{equation*}
	 	\int_{U_{n+p}^{n}(F)\backslash U_{n+p}^{n}(\A)} \varphi(  u \hat{\gamma} ) \psi_{n+p,\alpha}^{-1}(u )du.
\end{equation*}
By \eqref{eq-prop-ConstantTerm-prop1-eq10} and \eqref{eq-prop-ConstantTerm-prop1-eq11}, we obtain
\begin{equation*}
\begin{split}
c_p(f^{\psi_{n,\alpha}}) &=\phi(0)\\
&=	\sum_{\gamma\in P_{p-1}^1(F)\backslash \GL_p(F)} \int_{U_{n+p}^{n}(F)\backslash U_{n+p}^{n}(\A)} \varphi(  u \hat{\gamma} ) \psi_{n+p,\alpha}^{-1}(u )du + (\varphi^{U_p})^{\psi_{n,\alpha}}(1).
\end{split}
\end{equation*}
Finally, notice that
\begin{equation*}
\begin{split}
&\int_{U_{n+p}^{n}(F)\backslash U_{n+p}^{n}(\A)} \varphi(  u \hat{\gamma} ) \psi_{n+p,\alpha}^{-1}(u )du \\
=&	\int_{F^p\backslash \A^p} 	\int_{U_{n+p}^{n-1}(F)\backslash U_{n+p}^{n-1}(\A)} f_{n-1}( u u_{x} \hat{\gamma} ) \psi_{n+p,\alpha}^{-1}( u u_x) du dx\\
=& \int_{F^p\backslash \A^p} \int_{\mathcal{L}(\A)}	\int_{U_{n+p}^{n-1}(F)\backslash U_{n+p}^{n-1}(\A)} f( u u_x \hat{\gamma}   \lambda \beta )\psi_{n+p,\alpha}^{-1}(u u_x ) du   d\lambda \ \ (\text{by } \eqref{eq-prop-ConstantTerm-prop1-eq11}) \\
=& \int_{\mathcal{L}(\A)} \int_{U_{n+p}^{n}(F)\backslash U_{n+p}^{n}(\A)} f( u \hat{\gamma}   \lambda \beta )\psi_{n+p,\alpha}^{-1}(u  ) du   d\lambda
\end{split}
\end{equation*}
and similarly
\begin{equation*}
(\varphi^{U_p})^{\psi_{n,\alpha}}(1)=\int_{\mathcal{L}(\A)} (f^{U_p})^{\psi_{n,\alpha}}(\lambda \beta) d\lambda.	
\end{equation*}
Thus, we obtain \eqref{eq-prop-ConstantTerm-prop1}. 

To prove \eqref{eq-prop-ConstantTerm-prop1-2}, we continue with the Fourier expansion process. First, we consider a column vector $x\in \mathbb{A}^{p-1}$ and the subgroup
\begin{equation*}
u_x= \begin{pmatrix} I_{p-1} & x \\ & 1\end{pmatrix}^\wedge\in N_{n+p}(\A).
\end{equation*}
For fixed $\gamma\in \GL_p(F)$, the function
\begin{equation*}
x\mapsto \phi(x):=\int_{U_{n+p}^{n}(F)\backslash U_{n+p}^{n}(\A)} \varphi( u u_x \hat{\gamma})\psi_{n+p,\alpha}^{-1}(u)du
\end{equation*}
is a smooth function on $F^{p-1}\backslash \A^{p-1}$, and hence has a Fourier expansion. Then we continue the Fourier expansion along
\begin{equation*}
\begin{pmatrix} I_{p-i} & x \\ & 1\end{pmatrix}^\wedge\in N_{n+p}(\A),
\end{equation*}
where $x\in \A^{p-i}$, for $i=p-1, p-2, \cdots, 1$. Notice that $U_{n+p}^{n+p-1}=N_{n+p}$ and for  for $i=p-1$, $P_{p-i}^1=Z_p$. We then obtain \eqref{eq-prop-ConstantTerm-prop1-2}.
\end{proof}

As an immediate corollary to Proposition~\ref{prop-ConstantTerm-prop1} (2), we obtain the following result. 

\begin{proposition}
\label{prop-ConstantTerm-prop2}
Let $f$ be an automorphic function on $H^{\delta}(\A)$ which is smooth and of uniform moderate growth. Let $1\le p \le \wt m_{n,\alpha}$. If
\begin{equation*}
( f^{U_{p-i}} )^{\psi_{n+i,\alpha}} = 0 \quad \text{ for all }0\le i \le p-1,	
\end{equation*}
then
\begin{equation*}
c_p(f^{\psi_{n,\alpha}})=\sum_{\gamma\in Z_p(F)\backslash \GL_p(F)} \int_{\mathcal{L}(\A)}   f^{\psi_{n+p,\alpha}}(  \hat{\gamma}  \lambda \beta) d\lambda.
\end{equation*}
\end{proposition}

\section{The non-vanishing of the twisted automorphic descent}
\label{section-non-vanishing}
The goal of this section is to prove parts (3), (5) and (6) of Theorem~\ref{thm-Main}.

\subsection{Generalized and degenerate Whittaker-Fourier coefficients}

In this section we recall the notion of generalized and degenerate Whittaker-Fourier coefficients attached to nilpotent orbits, following \cite{GomezGourevitchSahi2017}.
Let $\RG$ be a reductive group defined over a number field $F$ and let $\mathfrak{g}$ be the Lie algebra of $\RG(F)$.
Given any semi-simple element $s \in \mathfrak{g}$, under the adjoint action, $\mathfrak{g}$ is decomposed to a direct sum of eigenspaces $\mathfrak{g}^s_i$ corresponding to eigenvalues $i$.
We call a semi-simple element
$s$ {\it rational semi-simple} if all its eigenvalues are in $\mathbb{Q}$.
Given a nilpotent element $u$, a {\it Whittaker pair} is a pair $(s,u)$ with $s \in \mathfrak{g}$ being a rational semi-simple element, and $u \in \mathfrak{g}^s_{-2}$. The element $s$ in a Whittaker pair $(s, u)$ is called a {\it neutral element} for $u$ if there is a nilpotent element $v \in \mathfrak{g}^{s}_{2}$ such that $(v,s,u)$ is an $\mathfrak{sl}_2$-triple. 
We call a Whittaker pair $(s,u)$ a {\it neutral pair} if $s$ is a neutral element for $u$.

For a Whittaker pair $(s,u)$, we define an anti-symmetric form $\omega_{u}$ on $\mathfrak{g}$ by 
\begin{equation*}
\omega_{u}(X, Y):=\kappa(u, [X, Y]),
\end{equation*}
where $\kappa$ is the Killing form on $\mathfrak{g}\times \mathfrak{g}$. For any rational number $r\in \mathbb{Q}$, let $\mathfrak{g}^s_{\geq r} = \oplus_{r' \geq r} \mathfrak{g}^s_{r'}$.
Let $\mathfrak{u}_s= \mathfrak{g}^s_{\geq 1}$ and let $\mathfrak{n}_{s,u}$ be the radical of $\omega_u |_{\mathfrak{u}_s}$. Then we have $[\mathfrak{u}_s, \mathfrak{u}_s] \subset \mathfrak{g}^s_{\geq 2} \subset \mathfrak{n}_{s,u}$. By \cite[Lemma 3.2.6]{GomezGourevitchSahi2017}, we have $\mathfrak{n}_{s,u} = \mathfrak{g}^s_{\geq 2} + \mathfrak{g}^s_1 \cap \mathfrak{g}_u$.
If the Whittaker pair $(s,u)$ comes from an $\mathfrak{sl}_2$-triple $(v,s,u)$, then $\mathfrak{n}_{s,u}=\mathfrak{g}^s_{\geq 2}$. Let $U_{s}=\exp(\mathfrak{u}_s)$ and $N_{s,u}=\exp(\mathfrak{n}_{s,u})$ be the corresponding unipotent subgroups of $\RG$. We define a character of $N_{s,u}(\A)$ by
\begin{equation}
\label{eq-GGS-char}
\psi_u(n)=\psi(\kappa(u,\log(n)))	
\end{equation}
Here, we extend the killing form $\kappa$ to $\mathfrak{g}(\A)\times\mathfrak{g}(\A)$. 
 Let $N_{s,u}' = N_{s,u} \cap \ker (\psi_u)$. Then $U_s/N_{s,u}'$ is a Heisenberg group, and its center  is $N_{s,u}/N_{s,u}'$.

Let $\pi$ be an irreducible automorphic representation of $\RG(\A)$. For any automorphic form $\phi \in V_\pi$, the {\it degenerate Whittaker-Fourier coefficient} of $\phi$ attached to a Whittaker pair $(s,u)$ is defined to be
\begin{equation}
\label{eq-GomezGourevitchSahi-degenerateWhittakerFC}
\mathcal{F}_{s,u}(\phi)(g):=\int_{[N_{s,u}]} \phi(ng) \psi_u^{-1}(n) \, \mathrm{d}n\,.
\end{equation}
If $s$ is a neutral element for $u$, then $\mathcal{F}_{s,u}(\phi)$ is also called a {\it generalized Whittaker-Fourier coefficient} of $\phi$.
Let 
\begin{equation*}
\mathcal{F}_{s,u}(\pi)=\{\mathcal{F}_{s,u}(\phi)|\phi \in V_\pi\}.	
\end{equation*}
The {\it wave-front set} $\mathfrak{n}(\pi)$ of $\pi$ is defined to the set of nilpotent orbits $\mathcal{O}$ such that $\mathcal{F}_{s,u}(\pi)$ is non-zero, for some neutral pair $(s,u)$ with $u \in \mathcal{O}$. Note that if $\mathcal{F}_{s,u}(\pi)$ is non-zero for some neutral pair $(s,u)$ with $u \in \mathcal{O}$, then it is non-zero for any such neutral pair $(s,u)$, since the non-vanishing property of such generalized Whittaker-Fourier coefficients does not depend on the choices of representatives of $\mathcal{O}$.
Let $\mathfrak{n}^m(\pi)$ be the set of maximal elements in $\mathfrak{n}(\pi)$ under the natural order of nilpotent orbits.
In the following we recall a theorem from \cite{GomezGourevitchSahi2017}.

\begin{theorem}[Theorem C, \cite{GomezGourevitchSahi2017}]\label{thm-GomezGourevitchSahi-global1}
Let $\pi$ be an irreducible automorphic representation of $\RG(\A)$.
Given a Whittaker pair $(s',u)$ and a neutral pair $(s,u)$, if $\mathcal{F}_{s',u}(\pi)$ is non-zero, then $\mathcal{F}_{s,u}(\pi)$ is non-zero.
\end{theorem}

For $\RG=\GSpin_{4n+2}^\delta$, since the projection $\pr:\GSpin_{4n+2}^\delta\to \SO_{4n+2}^{\delta}$ splits over unipotent subgroups, as in the special orthogonal case, the nilpotent orbits of $\RG$ are parameterized by pairs $(\underline{p}, \underline{q})$, where $\underline{p}$ is an orthogonal partition and $\underline{q}$ is a set of non-degenerate quadratic forms (see, for example, \cite[Section I.6]{Waldspurger2001Nilpotent}).  
Here, an orthogonal partition means a partition such that even parts occur with even multiplicities. 

Given an automorphic representation $\pi$ of $\RG(\A)$, let $\mathfrak{p}^m(\pi)$ be the set of partitions corresponding to nilpotent orbits in $\mathfrak{n}^m(\pi)$.
Given an orthogonal partition $\underline{p}$, by a generalized Whittaker-Fourier coefficient of $\pi$ attached to $\underline{p}$ we mean
 a generalized
Whittaker-Fourier coefficient $\mathcal{F}_{s,u}(\phi)$ attached to a nilpotent orbit $\mathcal{O}$ parametrized by a pair $(\underline{p}, \underline{q})$ for some $\underline{q}$, with $\phi \in V_\pi$, $u\in \mathcal{O}$ and $(s,u)$ being a neutral pair. 
To simplify notation, we sometimes also write a generalized Whittaker-Fourier coefficient attached to $\underline{p}$ simply as $\mathcal{F}^{\psi_{\underline{p}}}(\phi)$,  without specifying the $F$-rational nilpotent orbit $\mathcal{O}$ and the neutral pair.

For $\RG=\GSpin_{4n+2}^\delta$, as for $\SO_{4n+2}^\delta$, an orthogonal partition $\underline{p}$ is called {\it special} if it has an even number of odd parts between two consecutive even parts and an even number of odd parts greater than the largest even part (see \cite[Section 6.3]{CollingwoodMcGovern1993}). 
 By the main results of \cite{JiangLiuSavin2016}, if $\underline{p} \in \mathfrak{p}^m(\pi)$, then $\underline{p}$ is special.

\subsection{Non-vanishing of the descent}

\begin{proposition}
\label{prop-nonvanishing-toporbit}
The residual representation $\CE_{\tau \otimes \sigma}$ has a non-zero generalized Whittaker-Fourier coefficient attached to the partition $[(2n)^2, 1^2]$.
\end{proposition}

\begin{proof}
Let $\alpha_i=e_i-e_{i+1}$, $1\le i \le 2n$, $\alpha_{2n+1}=e_{2n}+e_{2n+1}$ be the set of simple roots for $\GSpin_{4n+2}^\delta$. Let $x_{\alpha_i}$ be the one-dimensional root subgroup corresponding to $\alpha_i$, $1\le i \le 2n+1$. By \cite[\S I.6]{Waldspurger2001Nilpotent}, there is only one nilpotent orbit $\mathcal{O}$ corresponding to the partition $[(2n)^2,1^2]$. A representative of the nilpotent orbit $\mathcal{O}$ can be taken as $u=\sum_{i=1}^{n-1}x_{-\alpha_i}(1)$. Let $s$ be the following semi-simple element:
\begin{equation*}
s=\diag( 2n-1, 2n-3, \cdots, 1-2n, 0, 0, 2n-1, \cdots, 3-2n,1-2n).
\end{equation*}
Then $(s,u)$ is a neutral pair. Let $s^\prime$ be the following semi-simple element
\begin{equation*}
s^\prime=\diag( 4n, 4n-2, \cdots, 2, 0, 0, -2, \cdots, 2-4n, -4n).	
\end{equation*}
Then $(s^\prime, u)$ is a Whittaker pair. We consider $\mathcal{F}_{s^\prime, u}(\CE_{\tau \otimes \sigma})$. Recall that $P=M\ltimes U$ is the standard parabolic subgroup of $\GSpin_{4n+2}^\delta$ with Levi subgroup $M\cong \GL_{2n}\times \GSpin_{2}^\delta$ and unipotent radical $U$. By definition, for any $\phi\in \CE_{\tau \otimes \sigma}$, $\mathcal{F}_{s^\prime, u}(\CE_{\tau \otimes \sigma})$ is the constant term integral over $U(F)\backslash U(\A)$ combined with a non-degenerate Whittaker-Fourier coefficient of $\tau$. 
The constant term integral over $U(F)\backslash U(\A)$ is non-zero since the parabolic subgroup $P=M\ltimes U$ is the parabolic subgroup used to define the Eisenstein series and the residual representation $\CE_{\tau \otimes \sigma}$. 
The non-degenerate Whittaker-Fourier coefficient of $\tau$ is non-zero since $\tau$ is generic. 
Hence $\mathcal{F}_{s^\prime, u}(\CE_{\tau \otimes \sigma})$ is non-vanishing. By Theorem~\ref{thm-GomezGourevitchSahi-global1}, $\mathcal{F}_{s, u}(\CE_{\tau \otimes \sigma})$ is non-vanishing. Thus $\CE_{\tau \otimes \sigma}$ has a non-zero generalized Whittaker-Fourier coefficient attached to the partition $[(2n)^2,1^2]$. 
\end{proof}

We now prove part (6) of Theorem~\ref{thm-Main}.
\begin{proposition}[Theorem~\ref{thm-Main} (6)]
\label{prop-nonvanishing-toporbitraise}
The residual representation $\CE_{\tau \otimes \sigma}$ has a non-zero generalized Whittaker-Fourier coefficient attached to the partition $[2n+1,2n-1, 1^2]$.
\end{proposition}

\begin{proof}
By Proposition~\ref{prop-nonvanishing-toporbit} and \cite[Theorem 11.2]{JiangLiuSavin2016}, $\mathfrak{p}^m(\CE_{\tau \otimes \sigma})$ contains the special expansion of the partition $[(2n)^2,1^2]$, i.e., the smallest special partition of $4n+2$ which is greater than $[(2n)^2,1^2]$. The special expansion of $[(2n)^2,1^2]$ is $[2n+1,2n-1 ,1^2]$. The result follows. 
\end{proof}

We now prove part (3) of Theorem~\ref{thm-Main}. 
In the following, given $\alpha\in F^\times$, we do not distinguish $\alpha$ with its square class or the quadratic form corresponding to it. 

\begin{proposition}[Theorem~\ref{thm-Main} (3)]
\label{prop-nonvanishing}
There exists a square class $\alpha$ in $F^\times$ such that the representation $\TD_{\psi_{n,\alpha}}(\CE_{\tau\otimes\sigma})$ of $\GSpin^{\delta,\alpha}_{2n+1}(\A)$ is non-zero, and in this case
    $$
    \TD_{\psi_{n,\alpha}}(\CE_{\tau\otimes\sigma})=\pi_1\oplus\pi_2\oplus\cdots\oplus\pi_k\oplus\cdots,
    $$
    where $\pi_i$ are irreducible cuspidal automorphic representations of $\GSpin^{\delta,\alpha}_{2n+1}(\A)$, which are nearly equivalent, but
    are not globally equivalent, i.e. the decomposition is multiplicity-free.
\end{proposition}

\begin{proof}
The second part of the statement follows from the uniqueness of the local Bessel models for $\GSpin$ groups, which is established in \cite{Yan2025}.	

We now prove the first part of the statement, and we prove it by contradiction following a similar argument in \cite[Proposition 3.3]{JiangLiuXuZhang2016}. Assume that for all $\alpha$ in $F^\times$, the representation $\TD_{\psi_{n,\alpha}}(\CE_{\tau\otimes\sigma})$ of $\GSpin^{\delta,\alpha}_{2n+1}(\A)$ is zero.
As in the proof of Proposition~\ref{prop-nonvanishing-toporbit}, we let $\mathcal{O}$ be the nilpotent orbit corresponding to the partition $[(2n)^2,1^2]$, and we take $u=\sum_{i=1}^{n-1}x_{-\alpha_i}(1)$ which is a representative of this orbit. let $U_{\mathcal{O}}$ the unipotent subgroup corresponding to the $\mathcal{O}$.
Then the character on $U_{\mathcal{O}}(\A)$ defined in \eqref{eq-GGS-char} is 
\begin{equation*}
\psi_{U_{\mathcal{O}}} (x):=\psi_u(x)= \psi(\tr( u \log(x)) )=	 \psi( x_{1,2}+x_{2,3}+\cdots + x_{2n-1,2n}), \quad x\in U_{\mathcal{O}}(\A).
\end{equation*}

For any $1\le i, j\le 2n+1$, we let $E_{i,j}$ be the unipotent subgroup of $\GSpin_{4n+2}^\delta$ consisting of elements $x$ with the property that $x$ has ones on the diagonal, and for $k\not=\ell$, we have $x_{k,\ell}=0$ unless $(k,\ell)=(i,j)$ or $(k,\ell)=(4n+3-j,4n+3-i)$. 
Let $Y = E_{2n+1,n+1} E_{n,2n+2}$. 
Let $\varphi\in \CE_{\tau\otimes\sigma}$ and $g\in \GSpin_{4n+2}^\delta(\A)$.
By Lemma~\ref{lemma-RootExchange}, the integral 
\begin{equation}
\label{eq-prop-nonvanishing-eq1}
\int_{[U_{\mathcal{O}}]} 	\varphi( vg ) \psi_{U_{\mathcal{O}}}^{-1}(v)dv 
\end{equation}
is non-zero if and only if the integral 
\begin{equation}
\label{eq-prop-nonvanishing-eq2}
f_{\varphi}(g):=\int_{[U_{\mathcal{O}}Y]} 	\varphi( v y g ) \psi_{U_{\mathcal{O}}}^{-1}(v)dv dy
\end{equation}
is non-zero. By Proposition~\ref{prop-nonvanishing-toporbit}, there exist $\varphi$ and $g$ such that the integral \eqref{eq-prop-nonvanishing-eq1} is non-zero. Thus, there exist $\varphi$ and $g$ such that 
 \begin{equation}
\label{eq-prop-nonvanishing-eq3}
f_{\varphi}(g)\not=0.
\end{equation}
Choose $\omega\in \GSpin_{4n+2}^\delta(F)$ such that
\begin{equation}
\label{eq-prop-nonvanishing-eq4}
\pr(\omega)=  \begin{pmatrix}
0 & 0 & I_{2n}\\
0 & I_2 & 0\\
I_{2n} & 0 & 0
\end{pmatrix}.
\end{equation}
Then we have $f_{\varphi}(g)=f_{\varphi}(\omega g)$.

Note that the following quadruple
\begin{equation*}
(U_{\mathcal{O}} {E_{n,2n+2}},\   \psi_{U_{\mathcal{O}}}, \ E_{2n+1,n+1}, \ E_{n,2n+1}).
\end{equation*}
satisfies the condition represented in \eqref{eq-RootExchange-diagram}. By applying Lemma~\ref{lemma-RootExchange}, we obtain that 
\begin{equation}
\label{eq-prop-nonvanishing-eq5}
f_\varphi(g)=\int_{E_{2n+1,n+1}(\A)}\int_{[U_{\mathcal{O}} Y']}
\varphi(vyxg) \psi^{-1}_{[(2n)^21^2]}(v) dvdydx,
\end{equation}
where $Y' = E_{n,2n+1} E_{n,2n+2}$.

Let $W=U_{\mathcal{O}}Y^\prime$, which consists of elements of the form
\begin{equation}
\label{eq-prop-nonvanishing-eq6}
w=\begin{pmatrix}
z & q_1 & q_2\\
0 & I_2 & q_1^*\\
0 & 0 & z^*
\end{pmatrix}\begin{pmatrix}
I_{2n} & 0 & 0\\
p_1& I_2 & 0\\
p_2 &p_1^*& I_{2n}
\end{pmatrix} \in \GSpin^{\delta}_{4n+2},
\end{equation}
where $z \in Z_{2n}$, the standard maximal unipotent subgroup of $\GL_{2n}$; $q_1 \in \{x\in \mathrm{Mat}_{2n \times 2}: x_{i,j}=0 \text{ for } n+1 \leq i \leq 2n, 1 \leq j \leq 2 \}$;
$q_2, p_2\in \{x\in \mathrm{Mat}_{2n\times 2n}: {}^t x w_{2n}+w_{2n}x=0, x_{i,j}=0 \text{ for }i\ge j \}$;
$p_1 \in \{ x\in \mathrm{Mat}_{2 \times 2n}: x_{i,j}=0 \text{ for }1 \leq i \leq 2, 1 \leq j \leq n+1 \}$.
We define a character $\psi_W(w):= \psi_{U_{\mathcal{O}}}(v)$,
for $w=vy \in W(\A)$, where $v \in U_{\mathcal{O}}(\A)$ and $y \in Y'(\A)$. 
For $w\in W(\A)$ of the form in $\eqref{eq-prop-nonvanishing-eq6}$, we have
\begin{equation*}
	\psi_W(w)=\psi(\sum_{i=1}^{2n-1} z_{i,i+1}).
\end{equation*}
Let $\wt{W}$ be the subgroup of $W$ with elements of the form as in \eqref{eq-prop-nonvanishing-eq6}, but with $p_1=0$ and $p_2=0$. Let $\psi_{\wt{W}} = \psi_W|_{\wt{W}}$. 
If we have a subgroup of $W$ containing $\wt{W}$, we extend the character $\psi_{\wt{W}}$ trivially to this subgroup and still denote the character by $\psi_{\wt{W}}$.

Next, we define the following sequence of unipotent subgroups:
\begin{equation*}
	X^i_j=E_{i,2n+3+i-j}, \  Y^i_j=E_{2n+3+i-j,i+1}, \quad \text{ for }  1 \leq i \leq n-1, 1 \leq j \leq i, 
\end{equation*}
\begin{equation*}
X^i_j=E_{i,4n+2-i-j}, \ Y^i_j=E_{4n+2-i-j,i+1}, \quad \text{ for } n \leq i \leq 2n-2, 1 \leq j \leq 2n-i-1, 	
\end{equation*}
and 
\begin{equation*}
X_i = E_{i,2n+1} E_{i,2n+2}, \ Y_i = E_{2n+1,i+1} E_{2n+2,i+1}	\quad \text{ for } n+1 \leq i \leq 2n-1.
\end{equation*}
We also define
\begin{equation*}
	\wt{W}_i = \left( \prod_{s=1}^{i-1} \prod_{j=1}^{s} X^s_j \right)  \wt{W} \left( \prod_{t=n+1}^{2n-1} Y_t \prod_{j=1}^{2n-t-1} Y^t_j \right) \left( \prod_{k=i+1}^{n} \prod_{j=1}^{k} Y^k_j\right), \quad \text{ for } 1\le i \le n-1.
\end{equation*}

We consider the following sequences of quadruples, as the index $i$ ranges from $i=1$ to $i=n-1$:
\begin{align}\label{eq-prop-nonvanishing-eq7}
\begin{split}
& (\wt{W}_i \prod_{j=2}^i Y^i_j, \psi_{\wt{W}}, X^i_1, Y^i_1),\\
& (X^i_1 \wt{W}_i \prod_{j=3}^i Y^i_j, \psi_{\wt{W}}, X^i_2, Y^i_2),\\
& \cdots,\\
& (\prod_{j=1}^{k-1} X^i_j \wt{W}_i \prod_{j=k+1}^i Y^i_j, \psi_{\wt{W}}, X^i_k, Y^i_k),\\
& \cdots,\\
& (\prod_{j=1}^{i-1} X^i_j \wt{W}_i, \psi_{\wt{W}}, X^i_{i}, Y^i_{i}).
\end{split}
\end{align}
Note that all the quadruples in \eqref{eq-prop-nonvanishing-eq7} satisfy the condition represented in \eqref{eq-RootExchange-diagram}. 
By applying Lemma~\ref{lemma-RootExchange} repeatedly to the sequence of quadruples in \eqref{eq-prop-nonvanishing-eq7} for $i$ going from $i=1$ to $i=n-1$, we obtain that
\begin{equation*}
f_{\varphi}(g)=\int_{\prod_{s=1}^{n-1} \prod_{j=1}^s Y^s_j
E_{2n+1,n+1}
(\A)}\int_{[\wt{W}_{n-1}']} \varphi(wxg) \psi^{-1}_{\wt{W}_{n-1}'}(w)dwdx,
\end{equation*}
where $\wt{W}_{i}'$ consists of elements of the form
\begin{equation}
\label{eq-prop-nonvanishing-eq8}
w=\begin{pmatrix}
z & q_1 & q_2\\
0 & I_2 & q_1^*\\
0 & 0 & z^*
\end{pmatrix}\begin{pmatrix}
I_{2n} & 0 & 0\\
p_1& I_2 & 0\\
p_2 &p_1^*& I_{2n}
\end{pmatrix} \in \GSpin^{\delta}_{4n+2},
\end{equation}
where $z \in Z_{2n}$; $q_1 \in \{x\in \mathrm{Mat}_{2n \times 2}: x_{k,j}=0 \text{ for } i+1 \leq k \leq 2n, 1 \leq j \leq 2 \}$;
$q_2\in \{x\in \mathrm{Mat}_{2n\times 2n}: {}^t x w_{2n}+w_{2n}x=0, x_{k,j}=0 \text{ for }1\le k\le 2, 1\le j \le 2n-i-1 \}$;
$p_1 \in \{ x\in \mathrm{Mat}_{2 \times 2n}: x_{k,j}=0 \text{ for }1 \leq k \leq 2, 1 \leq j \leq i+1 \}$; 
$p_2\in \{x\in \mathrm{Mat}_{2n\times 2n}: {}^t x w_{2n}+w_{2n}x=0, x_{k,j}=0 \text{ for }2n-i\le k \le 2n \text{ or } 1\le j \le i+1 \}$.
The character $\psi_{\wt{W}_{i}'}$ is defined by
\begin{equation*}
	\psi_{\wt{W}_{i}'}(w)=\psi(\sum_{i=1}^{2n-1} z_{i,i+1})
\end{equation*}
for $w\in \wt{W}_{i}'(\A)$ of the form \eqref{eq-prop-nonvanishing-eq8}.

Next, we take the Fourier expansion of $f_{\varphi}$ along the unipotent subgroup $E_{n,3n+2}$. Under the action of $\GL_1$, there are two types of Fourier coefficients, which correspond to the two orbits of the dual of $[E_{n,3n+2}]$. The two orbits are the trivial one and the non-trivial one. Because we have assumed that the residual representation $\CE_{\tau\otimes\sigma}$ has no non-zero Fourier coefficients attached to the partition $[2n+1,1^{2n+1}]$, all the Fourier coefficients corresponding to the non-trivial orbit are identically zero. Thus, we obtain that 
\begin{equation}\label{eq-prop-nonvanishing-eq9}
f_{\varphi}(g)=\int_{\prod_{s=1}^{n-1} \prod_{j=1}^s Y^s_j
E_{2n+1,n+1}(\A)} \int_{[E_{n,3n+2}]}
\int_{[\wt{W}_{n-1}']} \varphi(wxyg) \psi^{-1}_{\wt{W}_{n-1}'}(w)dwdxdy.
\end{equation}

Next, we consider the following sequence of quadruples:
\begin{align}\label{eq-prop-nonvanishing-eq10}
\begin{split}
& (E_{n,3n+2}\wt{W}_n \prod_{j=2}^{n-1} Y^n_j, \psi_{\wt{W}}, X^n_1, Y^n_1),\\
& (X^n_1 E_{n,3n+2}\wt{W}_n \prod_{j=3}^{n-1} Y^n_j, \psi_{\wt{W}}, X^n_2, Y^n_2),\\
& \cdots,\\
& (\prod_{j=1}^{k-1} X^n_j E_{n,3n+2}\wt{W}_n \prod_{j=k+1}^{n-1} Y^n_j, \psi_{\wt{W}}, X^n_k, Y^n_k),\\
& \cdots,\\
& (\prod_{j=1}^{n-2} X^n_j E_{n,3n+2}\wt{W}_n, \psi_{\wt{W}}, X^n_{n-1}, Y^n_{n-1}),
\end{split}
\end{align}
where
\begin{equation*}
\wt{W}_n =
\left( \prod_{t=1}^{n-1} \prod_{j=1}^{t} X^t_j \right) \wt{W} \left( \prod_{k=n+1}^{2n-1} Y_k \prod_{j=1}^{2n-k-1} Y^k_j \right).
\end{equation*}
All the above quadruples satisfy the condition represented in \eqref{eq-RootExchange-diagram}. 
By applying Lemma~\ref{lemma-RootExchange} repeatedly to the above sequence of quadruples in \eqref{eq-prop-nonvanishing-eq10}, we obtain that
\begin{equation}\label{eq-prop-nonvanishing-eq11}
f_{\varphi}(g)=\int_{\prod_{j=1}^{n-1} Y^n_j \prod_{s=1}^{n-1} \prod_{j=1}^s Y^s_j
E_{2n+1,n+1}(\A)}
\int_{[\wt{W}_{n}']} \varphi(wxg) \psi^{-1}_{\wt{W}_{n}'}(w)\ \mathrm{d}w\mathrm{d}x.
\end{equation}
Here, 
\begin{equation*}
\wt{W}_{n}' = E_{n,3n+2} \left( \prod_{j=1}^{n-1} X^n_j\right)
\left( \prod_{t=1}^{n-1} \prod_{j=1}^{t} X^t_j \right) \wt{W} \left( \prod_{k=n+1}^{2n-1} Y_k \prod_{j=1}^{2n-k-1} Y^k_j \right),
\end{equation*}
and $\psi_{\wt{W}_{n}'}$ is trivially extended from
$\psi_{\wt{W}}$.

For an index $i$ such that $n+1\le i \le 2n-1$, we define
\begin{equation}\label{eq-prop-nonvanishing-eq12}
\wt{W}_{i}' = \left( \prod_{s=n+1}^{i} X_s \prod_{j=1}^{2n-s-1} X^s_j \right)  \left( \prod_{\ell=n}^{i} E_{\ell, 4n+3-\ell}\right) 
\left( \prod_{t=1}^{n-1} \prod_{j=1}^{t} X^t_j \right) \wt{W} \left( \prod_{k=i+1}^{2n-1} Y_k \prod_{j=1}^{2n-k-1} Y^k_j\right),
\end{equation}
and extend the character $\psi_{\wt{W}}$ trivially to a character $\psi_{\wt{W}_{i}'}$ of $\wt{W}_{i}'(\A)$. 
For $n+1\le i \le 2n-2$, we define
\begin{equation*}
	\wt{W}_i = \left( \prod_{s=n+1}^{i-1} X_s \prod_{j=1}^{2n-s-1} X^s_j \right) \left( \prod_{\ell=n}^{i-1} E_{\ell,4n+2-\ell} \right)
\left( \prod_{t=1}^{n-1} \prod_{j=1}^{t} X^t_j \right)\wt{W} \left( \prod_{k=i+1}^{2n-1} Y_k \prod_{j=1}^{2n-k-1} Y^k_j\right).
\end{equation*}
We consider the following sequence of quadruples as $i$ ranges from $i=n+1$ to $i=2n-2$:
\begin{align*}
& (E_{i,4n+2-i}\wt{W}_i Y_i \prod_{j=2}^{2n-i-1} Y^i_j, \psi_{\wt{W}}, X^i_1, Y^i_1),\\
& (X^i_1 E_{i,4n+2-i}\wt{W}_i Y_i \prod_{j=3}^{2n-i-1} Y^i_j, \psi_{\wt{W}}, X^i_2, Y^i_2),\\
& \cdots,\\
& (\prod_{j=1}^{k-1} X^i_jE_{i,4n+2-i}\wt{W}_i Y_i  \prod_{j=k+1}^{2n-i-1} Y^i_j, \psi_{\wt{W}}, X^i_k, Y^i_k),\\
& \cdots,\\
& (\prod_{j=1}^{2n-i-2} X^i_j E_{i,4n+2-i}\wt{W}_i Y_i , \psi_{\wt{W}}, X^i_{2n-i-1}, Y^i_{2n-i-1}),\\
& (\prod_{j=1}^{2n-i-1} X^i_j E_{i,4n+2-i} \wt{W}_i, \psi_{\wt{W}}, X_i, Y_i).
\end{align*}
Note that all the above quadruples satisfy the condition represented in \eqref{eq-RootExchange-diagram}. 
We apply Lemma~\ref{lemma-RootExchange} repeatedly to the above sequence of quadruples, to obtain
\begin{equation}\label{eq-prop-nonvanishing-eq13}
f_{\varphi}(g)=\int_{\prod_{s=n+1}^{2n-1} Y_s \prod_{j=1}^{2n-s-1} Y^s_j\prod_{k=1}^{n-1} \prod_{\ell=1}^k Y^k_{\ell}
E_{2n+1,n+1}
(\A)}\int_{[N_{2n}']} \varphi^U(uxg) \psi^{-1}_{N_{2n}'}(u) dudx.
\end{equation}
Here, $N_{2n}^\prime=M\cap N_{2n}$, $\psi_{N_{2n}'}=\psi_{N_{2n}}|_{N_{2n}^\prime}$ where $\psi_{N_{2n}}(u)=\psi(\sum_{i=1}^{2n-1} u_{i,i+1})$. We remind the reader that $P=M\ltimes U$ is the standard parabolic subgroup of $\GSpin_{4n+2}^\delta$ whose Levi part is $M\cong \GL_{2n}\times \GSpin_2^\delta$, and $Q_{2n}=L_{2n}\ltimes N_{2n}$ is the standard parabolic subgroup with Levi $L_{2n}\cong (\GL_1)^{2n}\times \GSpin_2^\delta$. 
 
 For $t\in \A^\times$, we denote $D(t)=(tI_{2n}, 1)$ and view it as an element of $M(\A)$. Then we have $\omega D(t) \omega^{-1}= D(t^{-1})$, where $\omega$ is the Weyl element defined in \eqref{eq-prop-nonvanishing-eq4}. 
 Note that conjugation by $D(t)$ preserves the character $\psi_{N_{2n}'}$. If we conjugate $D(t)$ to the left of $u$ in \eqref{eq-prop-nonvanishing-eq13}  and make a change of variables on 
 $$\prod_{s=n+1}^{2n-1} Y_s \prod_{j=1}^{2n-s-1} Y^s_j\prod_{k=1}^{n-1} \prod_{\ell=1}^k Y^k_{\ell}
E_{2n+1,n+1}
(\A),$$
we obtain a factor $|t|^{-2n^2+1}$. Also, we have
\begin{equation*}
\varphi^U(D(t)uxg)=\delta_P^{\frac{1}{2}} (D(t)) |\det(tI_{2n})|^{-\frac{1}{2}} \omega_{\tau}(t) \varphi^U(uxg)=\omega_{\tau}(t)|t|^{2n^2} \varphi^U(uxg). 
\end{equation*}
Therefore, we have
\begin{equation*}
f_{\varphi}(D(t)g)=\omega_{\tau}(t) |t| f_{\varphi}(g). 	
\end{equation*}
On the other hand, we have
\begin{equation*}
f_{\varphi}(D(t)g)=f_{\varphi}(\omega D(t) g)=f_{\varphi}(D(t^{-1})\omega g) 	=\omega_{\tau}(t)^{-1}|t|^{-1} f_{\varphi}(\omega g)=\omega_{\tau}(t)^{-1}|t|^{-1} f_{\varphi}(g).
\end{equation*}
Thus
\begin{equation*}
\omega_{\tau}(t) |t| f_{\varphi}(g) = 	\omega_{\tau}^{-1}(t) |t|^{-1} f_{\varphi}(g)
\end{equation*}
for all $t\in \A^\times$. Since $\tau$ is unitary, $\omega_\tau |\cdot|$ can not be the trivial character. Thus, $f_{\varphi}(g)=0$. 
This is a contradiction to \eqref{eq-prop-nonvanishing-eq3} and thus we have completed the proof.
\end{proof}

We now prove part (5) of Theorem~\ref{thm-Main}.
\begin{proposition}[Theorem~\ref{thm-Main} (5)]
\label{prop-part5}
When $\TD_{\psi_{n,\alpha}}(\CE_{\tau\otimes\sigma})$ is non-zero, every irreducible direct summand of $\TD_{\psi_{n,\alpha}}(\CE_{\tau\otimes\sigma})$ has a non-zero Fourier coefficient attached to the partition $[2n-1,1^2]$.
\end{proposition}

\begin{proof}
Assume $\TD_{\psi_{n,\alpha}}(\CE_{\tau\otimes\sigma})$ is non-zero and let $\pi$ be an irreducible direct summand of the twisted automorphic descent $\TD_{\psi_{n,\alpha}}(\CE_{\tau\otimes\sigma})$.  We consider the following integral
\begin{equation}
\label{eq-prop-part5-eq1}
\langle \varphi_\pi, \overline{\xi^{\psi_{n,\alpha}}} \rangle=\int_{\ker(\pr)(\A) \GSpin_{2n+1}^{\delta,\alpha}(F)\backslash \GSpin_{2n+1}^{\delta,\alpha}(\A)} \varphi_\pi(h) \xi^{\psi_{n,\alpha}}(h)dh, 	
\end{equation}
where $\varphi_\pi\in V_\pi$, $\xi\in \CE_{\tau\otimes\sigma}$. Since $\pi$ is an irreducible component of $\TD_{\psi_{n,\alpha}}(\CE_{\tau\otimes\sigma})$, there exist a choice of data $\varphi_\pi\in V_\pi$ and  $\xi\in \CE_{\tau\otimes\sigma}$ such that the integral \eqref{eq-prop-part5-eq1} is non-zero.
Assume that $\xi(h)=\Res_{s=\frac{1}{2}}E(h, s, \phi_{\tau\otimes\sigma})$, where $\phi_{\tau\otimes\sigma}\in \mathcal{A}(M(F)U(\A)\bs \GSpin_{4n+2}^\delta(\A))_{\tau\otimes \sigma}$. It follows from \eqref{eq-prop-part5-eq1} that the following integral 
\begin{equation}
\label{eq-prop-part5-eq2}
\langle \varphi_\pi, \overline{E(\cdot, s, \phi_{\tau\otimes\sigma})^{\psi_{n,\alpha}}} \rangle=\int_{\ker(\pr)(\A) \GSpin_{2n+1}^{\delta,\alpha}(F)\backslash \GSpin_{2n+1}^{\delta,\alpha}(\A)} \varphi_\pi(h) E(h, s, \phi_{\tau\otimes\sigma})^{\psi_{n,\alpha}}dh, 
\end{equation}
is also non-zero for some choice of data. This is a Rankin--Selberg integral studied in \cite{Yan2025GGP} (explicitly, this is the integral $\mathcal{Z}(\varphi_\pi,f_{\tau,\sigma,s})$ with $m^\prime=4n+2$, $k=2n$, $\ell=n$, $\beta=k-\ell=n$ in \cite{Yan2025GGP}). By \cite[Corollary 3.16]{Yan2025GGP}, $\varphi_\pi^{\psi_{n-1,-\alpha}}$ is non-vanishing for some choice of data. Thus, $\pi$ has a non-vanishing Fourier coefficient attached to the partition $[(2n-1),1^2]$. This completes the proof. 
\end{proof}

\section{Langlands functorial transfer}
\label{section-functorial}

The goal of this section is to prove Theorem~\ref{thm-Main} (4).  

\begin{proposition}[Theorem~\ref{thm-Main} (4)]
\label{prop-part4}
When $\TD_{\psi_{n,\alpha}}(\CE_{\tau\otimes\sigma})$ is non-zero, any direct summand $\pi$ of $\TD_{\psi_{n,\alpha}}(\CE_{\tau\otimes\sigma})$ has a weak Langlands functorial transfer to $\tau$ in the sense that the Satake parameter of the local unramified component $\tau_v$ of $\tau$ is the local functorial transfer of that of the local unramified component $\pi_v$ of $\pi$ for almost all unramified local places $v$ of $F$.
\end{proposition}

\begin{proof}
Assume that the twisted descent $\TD_{\psi_{n,\alpha}}(\CE_{\tau\otimes\sigma})$ is non-zero. By Theorem~\ref{thm-Main} (2), $\TD_{\psi_{n,\alpha}}(\CE_{\tau\otimes\sigma})$ is cuspidal on $\GSpin_{2n+1}^{\delta,\alpha}(\A)$. By Theorem~\ref{thm-Main} (3), we write
$$
\TD_{\psi_{n,\alpha}}( \CE_{\tau\otimes\sigma})
=
\oplus_i\pi_i
$$
where each $\pi_i$ is an irreducible cuspidal automorphic representation of $\GSpin^{\delta,\alpha}_{2n+1}(\A)$. 
To prove Proposition~\ref{prop-part4}, it suffices to prove the following two statements:
\begin{enumerate}
\item the irreducible summands $\pi_i$ are nearly equivalent, that is, their local components at almost all unramified local places are equivalent; and
\item the weak Langlands functorial transfer from $\GSpin^{\delta,\alpha}_{2n+1}$ to $\GL_{2n}$ takes $\pi_i$ to the given $\tau$.
\end{enumerate}
It is clear that statement (2) implies statement (1). Thus, it suffices to prove statement (2). 
Recall that the Langlands dual group of $\GSpin_{2n+1}^{\delta,\alpha}$ is $\GSp_{2n}(\C)$. The Langlands functorial transfer considered here is the one associated to the natural embedding of $\GSp_{2n}(\C)$ into $\GL_{2n}(\C)$. 

Let $\pi$ be a direct summand of $\TD_{\psi_{n,\alpha}}(\CE_{\tau\otimes\sigma})$ and write $\pi=\otimes_v^\prime  \pi_v$. Note that the group $\GSpin_{2n+1}^{\delta,\alpha}(F_v)$ is split at almost all finite places $v$ of $F$. Let $v$ be a finite local place of $F$ such that $\pi_v, \tau_v,\sigma_v$ and the other relevant data in the construction of the twisted automorphic descent are unramified. Here $\pi_v$ is an irreducible unramified representation of the $F_v$-split $\GSpin_{2n+1}(F_v)$. Note that $\pi_v$ occurs as the unramified constituent of the twisted Jacquet module $J_{\psi_{n,\alpha}}(\Ind_{P_{2n}(F_v)}^{H^\delta(F_v)} \tau_v|\cdot|_v^{\frac{1}{2}}\otimes \sigma_v )$. 

The representation $\sigma_v$ is an irreducible unramified representation of $\GSpin(W_{2n})(F_v)$. When the group $\GSpin(W_{2n})(F_v)$ is split, we have $\GSpin(W_{2n})(F_v)\cong \GL_1(F_v) \times \GL_1(F_v)$, and the representation $\sigma_v=\sigma_{v,1}\otimes\sigma_{v,2}$ for a pair $(\sigma_{v,1}, \sigma_{v,2})$ of unramified characters  of $F_v^\times$, with $\omega_{\sigma_v}=\sigma_{v,1}\sigma_{v,2}$. When $\GSpin(W_{2n})(F_v)$ is $F_v$-quasi-split but not $F_v$-split, we have $\GSpin(W_{2n})(F_v)\cong \Res_{E_v/F_v}(\GL_1)$ for a quadratic extension $E_v/F_v$, and the representation $\sigma_v$ is an unramified character of $E_v^\times$, with $\omega_{\sigma_v}=\sigma_v|_{F_v^\times}$. 

Since $\tau_v$ is an irreducible unramified generic representation of $\GL_{2n}(F_v)$ with central character $\omega_{\tau_v}$, such that $\wt{\tau_v}\cong \tau_v\otimes\omega_{\sigma_v}^{-1}$, we have $\omega_{\tau_v}=\omega_{\sigma_v}^n$, and we may write $\tau_v$ as 
\begin{equation*}
\tau_v=\Ind_{B_{\GL_{2n}(F_v)}}^{\GL_{2n}(F_v)} (\mu_1 \otimes   \cdots \otimes 	\mu_n \otimes \mu_n^{-1}  \omega_{\sigma_v} \otimes \cdots \mu_1^{-1}\omega_{\sigma_v})
\end{equation*}
where $\mu_1, \cdots, \mu_n$ are unramified characters of $F_v^\times$.  
By Proposition~\ref{prop-Jacquet-Module-Unramified-l=n}, the Satake parameter of $\pi_v$ is transferred to the Satake parameter of $\tau_v$. 
\end{proof}

\section{Application to global Jacquet--Langlands correspondence between $\GL_2$ and $D^\times$}
\label{section-application-JacquetLanglands}
The goal of this section is to give an explicit description on the twisted automorphic descent $\TD_{\psi_{n,\alpha}}(\CE_{\tau\otimes\sigma})$ when $n=1$. 
As an application, we obtain another proof of the global Jacquet--Langlands correspondence for $\GL_2$. 

Let $V=V^+\oplus V_0 \oplus V^-$ be a 6-dimensional quadratic space over $F$, with 
\begin{equation*}
V^+=\Span\{e_1, e_2\}, \quad V_0=\Span\{e_0^{(1)}, e_0^{(2)}\}, \quad 	V^-=\Span\{e_{-2}, e_{-1}\}
\end{equation*}
associated to the quadratic form defined by the following matrix
\begin{equation*}
\begin{pmatrix} 
&&w_2\\
&J_\delta &\\
w_2&&	
\end{pmatrix}, \quad \text{ with } J_{\delta}=\begin{pmatrix} 1&\\ & \delta \end{pmatrix},
\end{equation*}
and let $H^\delta=\GSpin_6^\delta$. Recall that $-\delta\not\in {F^\times}^2$, and we denote by $E$ the quadratic extension $E=F(\sqrt{-\delta})$ over $F$. Let $\eta_{E/F}$ be the quadratic character of $F^\times\backslash \A^\times$ associated to the quadratic extension $E/F$ via the global class field theory.  We take $\ell=1$, and the quadratic subspace $W_\ell$ defined in \eqref{eq-W_ell} is $W_{\ell}=W_1=\Span\{e_2, e_0^{(1)}, e_0^{(2)}, e_{-2}\}$. Recall that $Q_\ell=Q_1=L_1\ltimes N_1$ is the standard maximal parabolic subgroup of $\GSpin_6^\delta$ with Levi subgroup $L_1\cong \GL_1\times \GSpin(W_1)$, and the unipotent radical is given by
\begin{equation}
\label{eq-unipotent-N1}
N_{1} = 	\left\{ n_1(x)=\begin{pmatrix} 1 & x &-\frac{1}{2}q_V(x,x)\\ &I_{4} &x^\prime\\ &&1\end{pmatrix}\in \GSpin_6^\delta \right\}.
\end{equation}
Since the Witt index of $W_1$ is greater than zero, the group $\SO(W_1)$ (and hence $\GSpin(W_1)$) acts transitively on the set of vectors of the same length. As in \eqref{eq-w_0-global}
we choose an isotropic vector $w_0=y_{\alpha}=e_2-\frac{\alpha}{2}e_{-2}$ for some $\alpha\in F^\times$. Then the character $\psi_{1,\alpha}$ on $N_1(\A)$ is 
\begin{equation*}
\psi_{1,\alpha}(n_1(x))=\psi(x_1-\frac{\alpha}{2}x_4). 	
\end{equation*}
The stabilizer of the character $\psi_{1,\alpha}$ in $L_1$ is the group 
\begin{equation*}
H_{1,\alpha}=\GSpin(w_0^\perp \cap W_1)=\GSpin_{3}^{\delta,\alpha}
\end{equation*}
associated with a quadratic form given by $J_{\delta,\alpha}$ in \eqref{eq-J-delta-alpha}. 
If the quadratic form $J_{\delta,\alpha}$ is non-split over $F$, then $\GSpin_{3}^{\delta,\alpha}$ is isomorphic to $D_{\delta,\alpha}^\times$, where $D_{\delta,\alpha}$ is the quaternion algebra $\left( \frac{-\delta,-\alpha}{F}\right)$. 
Note that the quaternion algebra $D_{\delta,\alpha}$ is uniquely determined by the class of $\alpha$ in the quotient group $F^\times/ \mathrm{Nm}_{E/F} (E^\times)$.

Recall that $\sigma$ is an irreducible cuspidal automorphic representation of $\GSpin_2^{\delta}(\A)=\Res_{E/F} \GL_1(\A)$, and $\tau$ is an irreducible unitary cuspidal automorphic representation of $\GL_{2}(\A)$.
We assume that $L(s, \tau, \wedge^2\otimes \omega_\sigma^{-1})$ has a pole at $s=1$ and $L(\frac{1}{2}, \sigma\times \tau\otimes \omega_\sigma^{-1})\neq 0$. Then the residual representation $\CE_{\tau\otimes\sigma}$ is non-zero by  Proposition~\ref{prop-Eisenstein-pole}.
 
Let $\pi$ be an irreducible cuspidal automorphic representation of $\GSpin_3^{\delta,\alpha}(\A)$ and let $\varphi_\pi\in V_\pi$ be an automorphic form in the space of $\pi$. We define the following global zeta integral
\begin{equation}
\label{eq-JL-eq1}
\mathcal{Z}(s, \phi_{\tau\otimes\sigma}, \varphi_\pi, \psi_{1,\alpha})=\int_{\ker(\pr)(\A) \GSpin_3^{\delta,\alpha}(F)\backslash \GSpin_3^{\delta,\alpha}(\A) }	\varphi_\pi(h) E^{\psi_{1,\alpha}}(h, s, \phi_{\tau\otimes\sigma}) dh,
\end{equation}
where $E^{\psi_{1,\alpha}}(\cdot , s, \phi_{\tau\otimes\sigma})$ is the Bessel coefficient of the Eisenstein series $E(\cdot , s, \phi_{\tau\otimes\sigma})$ defined in \eqref{eq-Bessel-coef}. 
This global zeta integral is a special case of the global Rankin--Selberg integral considered in \cite{Yan2025GGP}. 

Let $\epsilon\in \GSpin_6^\delta(F)$ be a representative of the double coset $P_2(F)\backslash \GSpin_6^\delta(F)/P_1(F)$ such that
\begin{equation*}
\pr(\epsilon)= \begin{pmatrix}
0&1&&&0&0\\
0&0&&&0&1\\
&&1&0&&\\
&&0&-1&&\\
1&0&&&0&0\\
0&0&&&1&0
\end{pmatrix}  ,
\end{equation*}
corresponding to the $\epsilon_{\alpha,\beta}$ given in \eqref{eq-double-coset-epsilon-alpha-beta} with index $(\alpha,\beta)=(0,1)$. Then the complement of $N_1^\epsilon:=N_1\cap \epsilon^{-1} P_1 \epsilon$ in $N_1$ is 
\begin{equation*}
\overline{N_1^{\epsilon}}:=\{ n_1((x,y,z,0)) \mid (x,y,z,0)\in W_1 \}.	
\end{equation*}
Define
\begin{equation}
\label{eq-JL-eqJ}
\mathcal{J}(s,\phi_\alpha^{W})(h)=\int_{N^{\epsilon}_{1}(\A)\backslash N_{1}(\A)}
\lambda_{s}\phi_\alpha^{W}(\epsilon n_1 h)\psi^{-1}_{1,\alpha}(n_1) d n_1,
\end{equation}
where
$$
\phi_\alpha^{W}(h)=\int_{[N_{1,2}]}\phi_{\tau \otimes \sigma}(u_{1,2}(x)h) \psi^{-1}(\frac{\alpha}{2}x) d x,
$$
and
 \begin{equation*}
 N_{1,2} = \left\{u_{1,2}(x)=
 	\begin{pmatrix}
1 & x & & & &\\
& 1 &&&&\\
&& 1 & &&\\
&&& 1 &&\\
&&&& 1 &-x\\
&&&&&1
\end{pmatrix}\right\}.
 \end{equation*}
Note that for any $x\in \GSpin_2^{\delta}$, we have $\epsilon x \epsilon^{-1} = x^{-1}$, where we view $x$ as an element of the Levi subgroup of $P_2$. Hence 
\begin{align*}
\mathcal{J}(s, \phi_\alpha^{W})(xh)=\sigma^{-1}(x)\mathcal{J}(s, \phi_\alpha^{W})(h)
\end{align*}
for any $x\in \GSpin_2^{\delta}(\A)$.

By the unfolding computation in \cite[Theorem 3.3]{Yan2025GGP}, the global zeta integral $\mathcal{Z}(s, \phi_{\tau\otimes\sigma}, \varphi_\pi, \psi_{1,\alpha})$ unfolds to
\begin{equation}
\label{eq-JL-eq2}
\mathcal{Z}(s, \phi_{\tau\otimes\sigma}, \varphi_\pi, \psi_{1,\alpha})= \int_{\GSpin^{\delta}_{2}(\A)\backslash \GSpin^{\delta,\alpha}_{3}(\A)}
\mathcal{J}(s, \phi_\alpha^{W})(\epsilon   h)
\mathcal{P}(\varphi_\pi, \sigma^{-1})(h) d h,	
\end{equation}
where
\begin{equation}
\label{eq-JL-eqBesselPeriod}
\mathcal{P}(\varphi_\pi, \sigma^{-1})(h):=\int_{ \ker(\pr)(\A) \GSpin_2^{\delta}(F)\backslash \GSpin_2^{\delta}(\A)  }\varphi_\pi(xh)\sigma^{-1}(x)d x	
\end{equation}
is the Bessel period of $(\varphi_\pi,\sigma^{-1})$. 

Let $\JL(\pi)$ be the Jacquet--Langlands transfer of $\pi$, which is a cuspidal automorphic representation of $\GL_2(\A)$.
 By the Waldspurger formula (see \cite[Theorem 2]{Waldspurger1985} for $\omega_\sigma=1$ and \cite[Theorem 1.4]{YuanZhangZhang2013} for general $\omega_\sigma$), the Bessel period
\begin{equation*}
		\mathcal{P}(\varphi_\pi, \sigma^{-1}) \not=0
\end{equation*}
if and only if
\begin{equation*}
L(\frac{1}{2}, \BC(\JL(\pi))\otimes\sigma^{-1})\not=0 \text{ and } \Hom_{\GSpin_2^{\delta}(F_v)} (\pi_v,\sigma_v)\not=0 \text{ for all }v.	
\end{equation*}
Here, $\BC(\JL(\pi))$ is the base change of $\JL(\pi)$ to $\GL_2(\mathbb{A}_E)$. 
By the celebrated theorem of Tunnell-Saito \cite{Tunnell1983, Saito1993} (see also \cite[Theorem 1.3]{YuanZhangZhang2013}), the space $\Hom_{\GSpin_2^{\delta}(F_v)} (\pi_v,\sigma_v)$ is at most one-dimensional, and it is one-dimensional if and only if 
\begin{equation*}
\epsilon(\frac{1}{2}, \BC( \JL(\pi_v))\times \sigma_v, \psi_v) = \sigma_v(-1) \eta_{E/F,v}(-1) \epsilon( \GSpin_{3,v}^{\delta,\alpha}).	
\end{equation*}
Here, $\epsilon( \GSpin_{3,v}^{\delta,\alpha})=1$ if $\GSpin_{3}^{\delta,\alpha}\cong D_{\delta,\alpha}^{\times}$ splits at the place $v$, and $\epsilon( \GSpin_{3,v}^{\delta,\alpha})=-1$ if $\GSpin_{3}^{\delta,\alpha}$ is ramified at $v$. The term $\BC(\JL(\pi_v))$ is the local base change to $\GL_2(E_v)$ of the local Jacquet--Langlands transfer $\JL(\pi_v)$ of $\pi_v$.
 Therefore, the statement that $\Hom_{\GSpin_2^{\delta}(F_v)} (\pi_v,\sigma_v)\not=0$ is equivalent to the statement that 
\begin{equation*}
\GSpin_3^{\delta, \alpha} \text{ is ramified at $v$ if and only if } \epsilon(\frac{1}{2}, \BC( \JL(\pi_v))\times \sigma_v, \psi_v) \sigma_v(-1) \eta_{E/F,v}(-1)=-1.  
\end{equation*}
Thus, for a pure tensor product $\varphi_\pi=\otimes_v \varphi_{v}$, we have
\begin{equation}
\label{eq-JL-eq3}
	\mathcal{P}(\varphi_\pi, \sigma^{-1})(h)= 
	C_0 \prod_{v} l_v (\pi_v(h_v) \varphi_{v})
\end{equation}
where $l_v$ is a non-zero element in $\Hom_{\GSpin_2^{\delta}(F_v)} (\pi_v,\sigma_v)$ when $\Hom_{\GSpin_2^{\delta}(F_v)} (\pi_v,\sigma_v)\not=0$, and 
\begin{equation*}
C_0=\begin{cases}
1 & \text{ if }   L(\frac{1}{2}, \BC(\JL(\pi))\otimes\sigma^{-1})\not=0 \text{ and } \Hom_{\GSpin_2^{\delta}(F_v)} (\pi_v,\sigma_v)\not=0 \text{ for all }v \\
0 &  \text{ otherwise}.
\end{cases}
\end{equation*}

Let $S$ be a finite set of places of $F$ which contains all archimedean places such that all data are unramified at places outside $S$. Then by the integral representation established in \cite{Yan2025GGP}, we have
\begin{equation}
\label{eq-JL-eq4}
\mathcal{Z}(s, \phi_{\tau\otimes\sigma}, \varphi_\pi, \psi_{1,\alpha})= C_0 \frac{L^S(s+\frac{1}{2}, \pi\times\tau)}{L^S(s+1, \sigma\times\tau\otimes\omega_{\pi}) L^S(2s+1, \tau, \wedge^2\otimes\omega_\pi) } \prod_{v\in S}\mathcal{Z}_v (s, \phi_v,\varphi_v,\psi_{1,\alpha,v})
\end{equation}
where the local zeta integral $\mathcal{Z}_v (s, \phi_v,\varphi_v,\psi_{1,\alpha,v})$ at the ramified local places are given by
\begin{equation*}
\mathcal{Z}_v (s, \phi_v,\varphi_v,\psi_{1,\alpha,v})= \int_{\GSpin_2^{\delta}(F_v)\backslash \GSpin_3^{\delta, \alpha}(F_v)}	\mathcal{J}_v(s, \phi_{\alpha,v}^{W})(\epsilon  h)l_v(\pi_v(h)\varphi_v)dh, 
\end{equation*}
with $\mathcal{J}_v(s, \phi_{\alpha,v}^{W})$ the local part of \eqref{eq-JL-eqJ} at $v$ defined by
\begin{equation*}
\mathcal{J}_v(s, \phi_{\alpha,v}^{W})(h) =\int_{\overline{N_1^{\epsilon}}(F_v)}	\lambda_{s}\phi_{\alpha,v}^{W}(\epsilon n_1 h)\psi^{-1}_{1,\alpha,v}(n_1) d n_1.
\end{equation*}

Our first main result in this section is the following.

\begin{theorem}
\label{thm-JL}
Let $\delta\in F^\times$ be such that $-\delta\not\in (F^\times)^2$. Let $\sigma$ be an irreducible cuspidal automorphic representation of $\GSpin_2^{\delta}(\A)=\Res_{E/F} \GL_1(\A)$, and $\tau$ an irreducible unitary cuspidal automorphic representation of $\GL_{2}(\A)$. 
Assume that the residual representation $\CE_{\tau\otimes\sigma}$ is not identically zero. Then the following hold. 
\begin{enumerate}
\item  The set of $\alpha\in F^\times$ such that the $\psi_{1,\alpha}$-Fourier coefficient $\CE^{\psi_{1,\alpha}}_{\tau\otimes\sigma}$ is not identically zero is a single coset $\alpha_0\cdot \mathrm{Nm}_{E/F} (E^\times)$.
\item For any $\alpha$ in the coset $\alpha_0\cdot \mathrm{Nm}_{E/F} (E^\times)$, the twisted automorphic descent $\TD_{\psi_{1,\alpha}} (\CE_{\tau\otimes\sigma})$ is irreducible, and has the property that 
\begin{equation*}
\JL(\TD_{\psi_{1,\alpha}} (\CE_{\tau\otimes\sigma}) ) \cong \tau. 
\end{equation*}

\item The norm class $\alpha_0$ is determined by the property that  $D_{\delta, \alpha_0}$ is ramified at a place $v$ of $F$ if and only if $\epsilon(\frac{1}{2}, \BC( \tau_v)\times \sigma_v, \psi_v) \sigma_v(-1) \eta_{E/F,v}(-1)=-1$.
\end{enumerate}
\end{theorem}

\begin{proof}
By Theorem~\ref{thm-Main} (2), the $\psi_{1,\alpha}$-Fourier coefficient 
$\CE_{\tau\otimes\sigma}^{\psi_{1,\alpha}}(\phi)$ associated to $\phi$ in \eqref{eq-phi-section}
is an element of $L_{\text{cusp}}^2(\GSpin_3^{\delta,\alpha}(F)\backslash \GSpin_3^{\delta,\alpha}(\A))$. Applying the spectral decomposition on $L_{\text{cusp}}^2(\GSpin_3^{\delta,\alpha}(F)\backslash \GSpin_3^{\delta,\alpha}(\A))$, we obtain that 
\begin{equation}
\label{eq-JL-spectralside}
\CE^{\psi_{1,\alpha}}_{\tau\otimes\sigma}(\phi)=\sum_{\pi\in\mathcal{A}_{\text{cusp}}(\GSpin_3^{\delta,\alpha})}
\sum_{\varphi_{\pi}\in \mathcal{B}_\pi} \langle\CE^{\psi_{1,\alpha}}_{\tau\otimes\sigma}(\phi),\varphi_{\pi}\rangle\varphi_{\pi}
\end{equation}
where  $\langle\cdot,\cdot\rangle$ is the inner product in $L^{2}_{\text{cusp}}(\GSpin_3^{\delta,\alpha}(F)\backslash \GSpin_3^{\delta,\alpha}(\A))$ and
$\mathcal{B}_\pi$ is an orthonormal basis of the space $\pi$.

We compute the spectral coefficient $\langle\CE^{\psi_{1,\alpha}}_{\tau\otimes\sigma}(\phi),\varphi_{\pi}\rangle$. Note that the integration domain $N_1(F)\backslash N_1(\A)$ is compact, and the cusp form $\varphi_\pi$ is rapidly decreasing, so we can interchange the order of taking the residue and taking the integration to obtain
\begin{equation*}
\begin{split}	
\Res_{s=\frac{1}{2}}\mathcal{Z}(s,\phi_{\tau\otimes\sigma},\bar{\varphi}_{\pi},\psi_{1,\alpha})
=&\Res_{s=\frac{1}{2}}\int_{\ker(\pr)(\A) \GSpin_3^{\delta,\alpha}(F)\backslash \GSpin_3^{\delta,\alpha}(\A)}
E^{\psi_{1,\alpha}}(\phi_{\tau\otimes\sigma},s)(h)\bar{\varphi}_{\pi}(h)d h\\
=&\int_{\ker(\pr)(\A) \GSpin_3^{\delta,\alpha}(F)\backslash \GSpin_3^{\delta,\alpha}(\A)}\CE^{\psi_{1,\alpha}}_{\tau\otimes\sigma}(\phi)(h)\bar{\varphi}_{\pi}(h)dh\\
=&\langle\CE^{\psi_{1,\alpha}}_{\tau\otimes\sigma}(\phi),\varphi_{\pi}\rangle.
\end{split}
\end{equation*}
By Theorem~\ref{thm-Main} (3), $\Res_{s=\frac{1}{2}}\mathcal{Z}(s,\phi_{\tau\otimes\sigma},\bar{\varphi}_{\pi},\psi_{1,\alpha})$ is zero unless $\pi\in\mathcal{A}_{cusp}(\GSpin_3^{\delta,\alpha})$ satisfies the condition that $\JL(\pi)\cong \tau$. 

We now consider $\pi$ such that $\JL(\pi)\cong \tau$. Then
$\Res_{s=\frac{1}{2}}\mathcal{Z}(s,\phi_{\tau\otimes\sigma},\bar{\varphi}_{\pi},\psi_{1,\alpha})\not=0$.
Note that $L^S(s, \sigma\times \tau\otimes\omega_\pi)$ and $L^S(s, \tau, \wedge^2\otimes\omega_\pi)$ converge absolutely and are non-vanishing for $\Re(s)>1$. 
By \cite{Yan2025GGP}, we can choose data so that the finite product of local zeta integrals $\prod_{v\in S}\mathcal{Z}_v (s, \phi_v, \overline{\varphi}_v,\psi_{1,\alpha,v})$ is holomorphic and non-vanishing at $s=\frac{1}{2}$.
Since  $\mathcal{Z}(s,\phi_{\tau\otimes\sigma},\bar{\varphi}_{\pi},\psi_{1,\alpha})$ has a simple pole at $s=\frac{1}{2}$, by \eqref{eq-JL-eq4} we conclude that $C_0=1$ and $L^S(s+\frac{1}{2},\overline{\pi}\times\tau)$ has a pole at $s=\frac{1}{2}$. Here, $\overline{\pi}$ is obtained from $\pi$ by taking the complex conjugate.  
Therefore, $\Res_{s=\frac{1}{2}}\mathcal{Z}(s,\phi_{\tau\otimes\sigma},\bar{\varphi}_{\pi},\psi_{1,\alpha})\not=0$ if and only if
\begin{enumerate}
\item[(i)] $L^{S}(s+\frac{1}{2},\overline{\pi}\times \tau)$ has a pole at $s=\frac{1}{2}$;
\item[(ii)] $L(\frac{1}{2}, \BC(\JL(\pi))\otimes\sigma^{-1})\not=0$;
\item[(iii)] $\GSpin_3^{\delta, \alpha}$ is ramified at $v$ if and only if  
$\epsilon(\frac{1}{2}, \BC( \JL(\pi_v))\times \sigma_v, \psi_v) \sigma_v(-1) \eta_{E/F,v}(-1)=-1.$
\end{enumerate}

Note that 
\begin{equation*}
L^S(s+\frac{1}{2},\overline{\pi}\times\tau)=L^S(s+\frac{1}{2}, \JL(\overline{\pi})\times\tau). 	
\end{equation*}
Thus, the statement that $L^{S}(s+\frac{1}{2},\overline{\pi}\times \tau)$ has a simple pole at $s=\frac{1}{2}$ is equivalent to $\JL(\overline{\pi})\cong\wt \tau$, which is equivalent to $\JL(\pi)\cong \tau$ since $\overline{\pi}=\wt \pi$.

Moreover, by Proposition~\ref{prop-Eisenstein-pole}, $\CE_{\tau\otimes\sigma}\not=0$ implies that $L(\frac{1}{2},\sigma\times\tau\otimes\omega_\sigma^{-1})\not=0$.
When $\JL(\pi)=\tau$, note that 
\begin{equation*}
	L(\frac{1}{2}, \BC(\JL(\pi))\otimes\sigma^{-1})=L(\frac{1}{2}, \BC(\tau)\otimes\sigma^{-1})=L(\frac{1}{2}, \tau\times \theta_{\sigma^{-1}}).
\end{equation*}
Here, $\theta_{\sigma^{-1}}$ is the automorphic induction of $\sigma^{-1}$, and we used \cite[page 102]{Bump1997book} for the second equality. 
By a similar computation as in \cite[pages 102-103]{Bump1997book}, we can check that 
\begin{equation}
\label{eq-JL-L-identity}
L(\frac{1}{2}, \tau\times \theta_{\sigma^{-1}})=L(\frac{1}{2}, \sigma\times\tau\otimes\omega_{\sigma}^{-1}).	
\end{equation}
It follows that
\begin{equation*}
L(\frac{1}{2}, \BC(\JL(\pi))\otimes\sigma^{-1})=L(\frac{1}{2}, \sigma\times\tau\otimes \omega_\sigma^{-1}).	
\end{equation*}
Thus, condition (ii) holds if $\CE_{\tau\otimes\sigma}\not=0$.

Finally, the twisted automorphic descent $\TD_{\psi_{1,\alpha}} (\CE_{\tau\otimes\sigma})$ is zero unless the group $\GSpin_3^{\delta, \alpha}$ is determined by condition (iii). 
In this case, by the uniqueness of the local Bessel model \cite{Yan2025}, if the twisted automorphic descent $\TD_{\psi_{1,\alpha}} (\CE_{\tau\otimes\sigma})$ is non-zero, there is a unique cuspidal automorphic representation $\pi$ such that 
$$\Res_{s=\frac{1}{2}}\mathcal{Z}(s,\phi_{\tau\otimes\sigma},\bar{\varphi}_{\pi},\psi_{1,\alpha})$$ 
is non-zero and then $\TD_{\psi_{1,\alpha}} (\CE_{\tau\otimes\sigma})$ is irreducible. 
By Theorem~\ref{thm-Main} (3), the twisted automorphic descent $\TD_{\psi_{1,\alpha}} (\CE_{\tau\otimes\sigma})$ is non-zero for some $\alpha\in F^\times$. If $\alpha_0\in F^\times$ and $\alpha\in F^\times$ represent the same coset in the quotient $F^\times/ \mathrm{Nm}_{E/F} (E^\times)$, then $\GSpin_3^{\delta,\alpha}\cong \GSpin_3^{\delta, \alpha_0}$, and  $\TD_{\psi_{1,\alpha}} (\CE_{\tau\otimes\sigma})$  is also a cuspidal automorphic representation of $\GSpin_3^{\delta,\alpha_0}(\A)$, which satisfies conditions (i)-(iii). Then the twisted automorphic descent $\TD_{\psi_{1,\alpha_0}} (\CE_{\tau\otimes\sigma})$ is non-zero. Thus, the set of $\alpha\in F^\times$ such that $\TD_{\psi_{1,\alpha}} (\CE_{\tau\otimes\sigma})$ is non-zero is a single coset $\alpha_0\cdot \mathrm{Nm}_{E/F} (E^\times)$, and we have $\JL( \TD_{\psi_{1,\alpha}} (\CE_{\tau\otimes\sigma}))=\tau$ by the spectral decomposition \eqref{eq-JL-spectralside}.
\end{proof}

In our second main result of this section below, we show that by choosing suitable $\delta$ and $\sigma$, we are able to obtain all infinitely-dimensional cuspidal automoprhic representations of $D^\times(\A)$. 

\begin{theorem}
\label{thm-JL-2}	
Let $\delta\in F^\times$ be such that $-\delta\not\in (F^\times)^2$. Let $D$ be a quaternion division algebra containing the quadratic extension $E=F(\sqrt{-\delta})$ of $F$. For any given infinite-dimensional irreducible cuspidal automorphic representation $\pi$ of $D^\times(\A)$, there exists a character $\sigma$ of $\GSpin_2^{\delta}(\A)$ such that the residual representation $\CE_{\JL(\pi)\otimes \sigma}$ is non-zero, and an element $\alpha\in F^\times$ such that $D_{\delta,\alpha}\cong D$ and 
\begin{equation*}
	\TD_{\psi_{1,\alpha}} (\CE_{\JL(\pi)\otimes\sigma})\cong \pi.
\end{equation*}
Moreover, any such $\sigma$ satisfies the period condition 
\begin{equation*}
	\mathcal{P}(\varphi_\pi, \sigma^{-1}) \not=0, \quad \text{ for some }\varphi_\pi \in V_\pi,
\end{equation*}
which is equivalent to the conditions that
\begin{equation*}
L(\frac{1}{2}, \BC(\JL(\pi))\otimes\sigma^{-1})\not=0 \text{ and } \Hom_{\GSpin_2^{\delta}(F_v)} (\pi_v,\sigma_v)\not=0 \text{ for all }v.	
\end{equation*}
\end{theorem}

\begin{proof}
Consider $\GSpin_2^{\delta}(\A)$ as a subgroup of $D^\times(\A)$. The restriction of $\pi$ onto the subgroup $\GSpin_2^{\delta}(F)\backslash \GSpin_2^{\delta}(\A)$ is semi-simple and non-zero. By the spectral decomposition of the $L^2$-space of automorphic forms on $\GSpin_2^{\delta}(F)\backslash \GSpin_2^{\delta}(\A)$, we can choose a character $\sigma$ of $\GSpin_2^{\delta}(\A)$ with $\omega_{\sigma}=\pi|_{\A^\times}$ such that $\mathcal{P}(\varphi_\pi, \sigma^{-1}) \not=0$ for some $\varphi_\pi\in V_\pi$. (This also follows from the Burger-Sarnak principle; see \cite[Lemma 1]{Prasad2007}). By the Waldspurger formula (see \cite[Theorem 2]{Waldspurger1985} for $\omega_\sigma=1$ and \cite[Theorem 1.4]{YuanZhangZhang2013} for general $\omega_\sigma$), the condition that $\mathcal{P}(\varphi_\pi, \sigma^{-1}) \not=0$ for some $\varphi_\pi\in V_\pi$ is equivalent to the conditions that 
\begin{equation*}
L(\frac{1}{2}, \BC(\JL(\pi))\otimes\sigma^{-1})\not=0 \text{ and } \Hom_{\GSpin_2^{\delta}(F_v)} (\pi_v,\sigma_v)\not=0 \text{ for all }v.	
\end{equation*}

By the proof of Theorem~\ref{thm-JL}, we have
\begin{equation*}
L(\frac{1}{2}, \BC(\JL(\pi))\otimes\sigma^{-1})=L(\frac{1}{2}, \sigma\times \JL(\pi)\otimes\omega_{\sigma}^{-1}).	
\end{equation*}
Thus, $L(\frac{1}{2}, \sigma\times \JL(\pi)\otimes\omega_{\sigma}^{-1})\not=0$. 
Since the central character of $\JL(\pi)$ is equal to $\omega_{\sigma}$,  $L(s,\JL(\pi),\wedge^2\otimes\omega_{\sigma}^{-1})=L(s,\mathbf{1})$ has a pole at $s=1$. Here, $\mathbf{1}$ is the trivial character of $\GL_1(\A)$. By Proposition~\ref{prop-Eisenstein-pole}, the Eisenstein series $E(g, s, \phi_{\JL(\pi)\otimes\sigma})$ has a simple pole at $s=\frac{1}{2}$ and the residual representation $\CE_{\JL(\pi)\otimes \sigma}$ is non-zero. By Theorem~\ref{thm-JL}, there exists a unique $\GSpin_3^{\delta,\alpha}$ such that $\JL( \TD_{\psi_{1,\alpha}}(\CE_{\JL(\pi)\otimes \sigma}))=\JL(\pi)$. 

It remains to show that $\GSpin_3^{\delta,\alpha}$ is isomorphic to $D^\times$. 
Since $\JL( \TD_{\psi_{1,\alpha}}(\CE_{\JL(\pi)\otimes \sigma}))=\JL(\pi)$, by the proof of Theorem~\ref{thm-JL}, we have that $\Res_{s=\frac{1}{2}}\mathcal{Z}(s,\phi_{\TD_{\psi_{1,\alpha}}(\CE_{\JL(\pi)\otimes \sigma}) \otimes\sigma},\bar{\varphi}_{\pi},\psi_{1,\alpha})\not=0$ and $\mathcal{P}( \varphi, \sigma^{-1})\not=0$ for some $\varphi\in \TD_{\psi_{1,\alpha}}(\CE_{\JL(\pi)\otimes \sigma})$. By condition (iii) in the proof of Theorem~\ref{thm-JL}, we conclude that both $\GSpin_3^{\delta,\alpha}$ and $D^\times$ are ramified at $v$ if and only if 
$\epsilon(\frac{1}{2}, \BC( \JL(\pi_v))\times \sigma_v, \psi_v) \sigma_v(-1) \eta_{E/F,v}(-1)=-1.$
We conclude that $\GSpin_3^{\delta,\alpha}\cong D^\times$ and $\TD_{\psi_{1,\alpha}} (\CE_{\JL(\pi)\otimes\sigma})\cong \pi$. This completes the proof. 
\end{proof}

\section{Application to global Gan--Gross--Prasad conjecture for $(\GSpin_{2n+1}, \GSpin^{\delta}_2)$}
\label{section-application-GGP}

The goal of this section is to prove Theorem~\ref{Intro-thm-GGP-converse}, which we restate below. 

\begin{theorem}
\label{thm-GGP-converse}
Let $\delta\in F^\times$ be such that $-\delta\not\in (F^\times)^2$. 
Let $\sigma$ be an automorphic character of $\GSpin_2^{\delta}(\A)$.
Let $\pi$ be an irreducible cuspidal automorphic representation of $\GSpin_{2n+1}(\A)$ with central character $\omega_{\sigma}^{-1}$ that admits a weak functorial transfer to an isobaric sum automorphic representation $\tau$ of $\GL_{2n}(\A)$, and assume that $\tau$ is also a weak functorial transfer of an irreducible generic cuspidal representation of $\GSpin_{2n+1}(\A)$.
If
\begin{equation*}
L(\frac{1}{2}, \pi\times\sigma )\not=0,	
\end{equation*}
then the following holds.
\begin{enumerate}
\item There exist some $\alpha\in F^\times$ such that the twisted automorphic descent $\TD_{\psi_{n,\alpha}}(\CE_{\tau\otimes\sigma^{-1}})$ to $\GSpin^{\delta,\alpha}_{2n+1}(\A)$ is non-zero.
\item The twisted automorphic descent has a multiplicity-free direct sum decomposition
\begin{equation*}
\TD_{\psi_{n,\alpha}}(\CE_{\tau\otimes\sigma^{-1}})= \bigoplus_{i} \pi_\alpha^{(i)},	
\end{equation*}
where each $\pi_\alpha^{(i)}$ is an irreducible cuspidal automorphic representation of $\GSpin^{\delta,\alpha}_{2n+1}(\A)$ with central character $\omega_{\sigma}^{-1}$, and each $\pi_\alpha^{(i)}$ is nearly equivalent to $\pi$.
\item 	The Bessel period for $(\pi_{\alpha}^{(i)}, \sigma)$ is non-zero for each $\pi_{\alpha}^{(i)}$.
\end{enumerate} 
\end{theorem}

\begin{proof}
Write $\tau=\tau_1\boxplus \tau_2\boxplus \cdots \boxplus \tau_r$, where $\tau_i$ is an irreducible cuspidal automorphic representation of $\GL_{k_i}(\A)$ for each $1\le i \le r$, and $k_1+\cdots +k_r=2n$. By the description of image of functoriality for generic cuspidal representations of $\GSpin_{2n+1}(\A)$ in \cite{AsgariShahidi2006, AsgariShahidi2014}, we have $\tau_i\not\cong \tau_j$ if $i\not=j$, and $L(s,\tau_i,\wedge^2\otimes \omega_{\sigma})$ has a pole at $s=1$ for each $1\le i \le r$. By a similar computation as in \eqref{eq-JL-L-identity}, we have
\begin{equation*}
L(\frac{1}{2}, \sigma^{-1}\times \tau\otimes \omega_{\sigma})=L(\frac{1}{2}, \tau\times \sigma)=L(\frac{1}{2},\pi\times\sigma) \not=0.	
\end{equation*}
By Proposition~\ref{prop-Eisenstein-pole},  the Eisenstein series $E(g, s, \phi_{\tau\otimes\sigma^{-1}})$ has a pole of order $r$ at $s=\frac{1}{2}$. By Theorem~\ref{thm-Main}, there exists some $\alpha$ such that the twisted automorphic descent $\TD_{\psi_{n,\alpha}}(\CE_{\tau\otimes\sigma^{-1}})$ is non-zero, and has a multiplicity-free direct sum decomposition, and any direct summand of it has a weak Langlands transfer to $\tau$, and has a non-zero Fourier coefficient attached to partition $[2n-1, 1^2]$. Let $\pi_{\alpha}^{(i)}$ be a direct summand of $\TD_{\psi_{n,\alpha}}(\CE_{\tau\otimes\sigma^{-1}})$, which is an irreducible cuspidal automorphic representation of $\GSpin^{\delta,\alpha}_{2n+1}(\A)$ with central character $\omega_{\sigma}^{-1}$. By construction, there exist a cusp form 
 $\varphi \in V_{\pi_\alpha^{(i)}}$ and a residue $\Res_{s=\frac{1}{2}}E(\cdot, s, \phi_{\tau\otimes\sigma^{-1} })\in \CE_{\tau\otimes\sigma^{-1} }$ such that 
\begin{equation*}
\int_{ \ker(\pr)(\A) \GSpin_{2n+1}^{\delta,\alpha}(F)\backslash \GSpin_{2n+1}^{\delta,\alpha}(\A)} \varphi(g)  \Res_{s=\frac{1}{2}} E^{\psi_{n,\alpha}} (g, s, \phi_{\tau\otimes\sigma^{-1}})dg\not=0.
\end{equation*}
Replacing the residue by the Eisenstein series itself, we obtain that the integral
\begin{equation}
\label{eq-application-GGP-eq1}
\int_{ \ker(\pr)(\A) \GSpin_{2n+1}^{\delta,\alpha}(F)\backslash \GSpin_{2n+1}^{\delta,\alpha}(\A)} \varphi(g)    E^{\psi_{n,\alpha}} (g, s, \phi_{\tau\otimes\sigma^{-1} })dg 
\end{equation}
is non-zero for $\Re(s)\gg 0$. By the unfolding computation in \cite{Yan2025GGP}, this integral unfolds to the Bessel period $\mathcal{P}(\varphi, \sigma)$. If the Bessel period $\mathcal{P}(\varphi, \sigma)$ is vanishing for all $\varphi \in V_{\pi_\alpha^{(i)}}$, then the global integral \eqref{eq-application-GGP-eq1} would be identically zero, leading to a contradiction. Hence we conclude that the Bessel period $\mathcal{P}(\varphi,\sigma)\not=0$ for some $\varphi \in V_{\pi_\alpha^{(i)}}$.
\end{proof}

\bibliographystyle{alpha}
\bibliography{References}

\begin{thebibliography}{CKPSS04}

\bibitem[ACS24]{AsgariCogdellShahidi2024}
Mahdi Asgari, James~W. Cogdell, and Freydoon Shahidi.
\newblock Rankin-{S}elberg {$L$}-functions for $\mathrm{GSpin} \times
  \mathrm{GL}$ groups.
\newblock {\em arXiv preprint, available at
  \url{https://arxiv.org/abs/2409.17323}}, 2024.

\bibitem[Art13]{Arthur2013}
James Arthur.
\newblock {\em The endoscopic classification of representations}, volume~61 of
  {\em American Mathematical Society Colloquium Publications}.
\newblock American Mathematical Society, Providence, RI, 2013.
\newblock Orthogonal and symplectic groups.

\bibitem[AS06]{AsgariShahidi2006}
Mahdi Asgari and Freydoon Shahidi.
\newblock Generic transfer for general spin groups.
\newblock {\em Duke Math. J.}, 132(1):137--190, 2006.

\bibitem[AS14]{AsgariShahidi2014}
Mahdi Asgari and Freydoon Shahidi.
\newblock Image of functoriality for general spin groups.
\newblock {\em Manuscripta Math.}, 144(3-4):609--638, 2014.

\bibitem[Bum97]{Bump1997book}
Daniel Bump.
\newblock {\em Automorphic forms and representations}, volume~55 of {\em
  Cambridge Studies in Advanced Mathematics}.
\newblock Cambridge University Press, Cambridge, 1997.

\bibitem[BZ76]{BernsteinZelevinsky1976}
J.~Bernstein and A.~Zelevinsky.
\newblock Representations of the group {$GL(n,F),$} where {$F$} is a local
  non-{A}rchimedean field.
\newblock {\em Uspehi Mat. Nauk}, 31(3(189)):5--70, 1976.

\bibitem[CFK24]{CaiFriedbergKaplan2024}
Yuanqing Cai, Solomon Friedberg, and Eyal Kaplan.
\newblock Doubling constructions: global functoriality for non-generic cuspidal
  representations.
\newblock {\em Ann. of Math. (2)}, 200(3):893--966, 2024.

\bibitem[CKPSS01]{CogdellKimPSShahidi2001}
J.~W. Cogdell, H.~H. Kim, I.~I. Piatetski-Shapiro, and F.~Shahidi.
\newblock On lifting from classical groups to {${\rm GL}_N$}.
\newblock {\em Publ. Math. Inst. Hautes \'Etudes Sci.}, (93):5--30, 2001.

\bibitem[CKPSS04]{CogdellKimPSShahidi2004}
J.~W. Cogdell, H.~H. Kim, I.~I. Piatetski-Shapiro, and F.~Shahidi.
\newblock Functoriality for the classical groups.
\newblock {\em Publ. Math. Inst. Hautes \'Etudes Sci.}, (99):163--233, 2004.

\bibitem[CM93]{CollingwoodMcGovern1993}
David~H. Collingwood and William~M. McGovern.
\newblock {\em Nilpotent orbits in semisimple {L}ie algebras}.
\newblock Van Nostrand Reinhold Mathematics Series. Van Nostrand Reinhold Co.,
  New York, 1993.

\bibitem[CPSS11]{CogdellPSShahidi2011}
J.~W. Cogdell, I.~I. Piatetski-Shapiro, and F.~Shahidi.
\newblock Functoriality for the quasisplit classical groups.
\newblock In {\em On certain {$L$}-functions}, volume~13 of {\em Clay Math.
  Proc.}, pages 117--140. Amer. Math. Soc., Providence, RI, 2011.

\bibitem[Emo20]{Emory2020}
Melissa Emory.
\newblock On the global {G}an-{G}ross-{P}rasad conjecture for general spin
  groups.
\newblock {\em Pacific J. Math.}, 306(1):115--151, 2020.

\bibitem[GGP12]{GanGrossPrasad2012}
Wee~Teck Gan, Benedict~H. Gross, and Dipendra Prasad.
\newblock Symplectic local root numbers, central critical {$L$} values, and
  restriction problems in the representation theory of classical groups.
\newblock {\em Ast\'erisque}, (346):1--109, 2012.
\newblock Sur les conjectures de Gross et Prasad. I.

\bibitem[GGS17]{GomezGourevitchSahi2017}
Raul Gomez, Dmitry Gourevitch, and Siddhartha Sahi.
\newblock Generalized and degenerate {W}hittaker models.
\newblock {\em Compos. Math.}, 153(2):223--256, 2017.

\bibitem[GJR04]{GinzburgJiangRallis2004}
David Ginzburg, Dihua Jiang, and Stephen Rallis.
\newblock On the nonvanishing of the central value of the {R}ankin-{S}elberg
  {$L$}-functions.
\newblock {\em J. Amer. Math. Soc.}, 17(3):679--722, 2004.

\bibitem[GJR05]{GinzburgJiangRallis2005}
David Ginzburg, Dihua Jiang, and Stephen Rallis.
\newblock On the nonvanishing of the central value of the {R}ankin-{S}elberg
  {$L$}-functions. {II}.
\newblock In {\em Automorphic representations, {$L$}-functions and
  applications: progress and prospects}, volume~11 of {\em Ohio State Univ.
  Math. Res. Inst. Publ.}, pages 157--191. de Gruyter, Berlin, 2005.

\bibitem[GJR09]{GinzburgJiangRallis2009}
David Ginzburg, Dihua Jiang, and Stephen Rallis.
\newblock Models for certain residual representations of unitary groups.
\newblock In {\em Automorphic forms and {$L$}-functions {I}. {G}lobal aspects},
  volume 488 of {\em Contemp. Math.}, pages 125--146. Amer. Math. Soc.,
  Providence, RI, 2009.

\bibitem[GP92]{GrossPrasad1992}
Benedict~H. Gross and Dipendra Prasad.
\newblock On the decomposition of a representation of {${\rm SO}_n$} when
  restricted to {${\rm SO}_{n-1}$}.
\newblock {\em Canad. J. Math.}, 44(5):974--1002, 1992.

\bibitem[GP94]{GrossPrasad1994}
Benedict~H. Gross and Dipendra Prasad.
\newblock On irreducible representations of {${\rm SO}_{2n+1}\times{\rm
  SO}_{2m}$}.
\newblock {\em Canad. J. Math.}, 46(5):930--950, 1994.

\bibitem[GRS97]{GinzburgRallisSoudry1997IMRN}
David Ginzburg, Stephen Rallis, and David Soudry.
\newblock Self-dual automorphic {${\rm GL}_n$} modules and construction of a
  backward lifting from {${\rm GL}_n$} to classical groups.
\newblock {\em Internat. Math. Res. Notices}, (14):687--701, 1997.

\bibitem[GRS98]{GinzburgRallisSoudry1998}
David Ginzburg, Stephen Rallis, and David Soudry.
\newblock {$L$}-functions for symplectic groups.
\newblock {\em Bull. Soc. Math. France}, 126(2):181--244, 1998.

\bibitem[GRS99a]{GinzburgRallisSoudry1999Duke}
David Ginzburg, Stephen Rallis, and David Soudry.
\newblock Lifting cusp forms on {${\rm GL}_{2n}$} to {$\tilde{\rm Sp}_{2n}$}:
  the unramified correspondence.
\newblock {\em Duke Math. J.}, 100(2):243--266, 1999.

\bibitem[GRS99b]{GinzburgRallisSoudry1999JAMS}
David Ginzburg, Stephen Rallis, and David Soudry.
\newblock On a correspondence between cuspidal representations of {${\rm
  GL}_{2n}$} and {$\widetilde{\rm Sp}_{2n}$}.
\newblock {\em J. Amer. Math. Soc.}, 12(3):849--907, 1999.

\bibitem[GRS99c]{GinzburgRallisSoudry1999Annals}
David Ginzburg, Stephen Rallis, and David Soudry.
\newblock On explicit lifts of cusp forms from {${\rm GL}_m$} to classical
  groups.
\newblock {\em Ann. of Math. (2)}, 150(3):807--866, 1999.

\bibitem[GRS01]{GinzburgRallisSoudry2001IMRN}
David Ginzburg, Stephen Rallis, and David Soudry.
\newblock Generic automorphic forms on {${\rm SO}(2n+1)$}: functorial lift to
  {${\rm GL}(2n)$}, endoscopy, and base change.
\newblock {\em Internat. Math. Res. Notices}, (14):729--764, 2001.

\bibitem[GRS11]{GinzburgRallisSoudry2011}
David Ginzburg, Stephen Rallis, and David Soudry.
\newblock {\em The descent map from automorphic representations of {${\rm
  GL}(n)$} to classical groups}.
\newblock World Scientific Publishing Co. Pte. Ltd., Hackensack, NJ, 2011.

\bibitem[How79]{Howe1979}
R.~Howe.
\newblock {$\theta $}-series and invariant theory.
\newblock In {\em Automorphic forms, representations and {$L$}-functions
  ({P}roc. {S}ympos. {P}ure {M}ath., {O}regon {S}tate {U}niv., {C}orvallis,
  {O}re., 1977), {P}art 1}, volume XXXIII of {\em Proc. Sympos. Pure Math.},
  pages 275--285. Amer. Math. Soc., Providence, RI, 1979.

\bibitem[HS16]{HundleySayag2016}
Joseph Hundley and Eitan Sayag.
\newblock Descent construction for {GS}pin groups.
\newblock {\em Mem. Amer. Math. Soc.}, 243(1148):v+124, 2016.

\bibitem[II10]{IchinoIkeda2010}
Atsushi Ichino and Tamutsu Ikeda.
\newblock On the periods of automorphic forms on special orthogonal groups and
  the {G}ross-{P}rasad conjecture.
\newblock {\em Geom. Funct. Anal.}, 19(5):1378--1425, 2010.

\bibitem[JL70]{JacquetLanglands1970}
H.~Jacquet and R.~P. Langlands.
\newblock {\em Automorphic forms on {${\rm GL}(2)$}}, volume Vol. 114 of {\em
  Lecture Notes in Mathematics}.
\newblock Springer-Verlag, Berlin-New York, 1970.

\bibitem[JLS16]{JiangLiuSavin2016}
Dihua Jiang, Baiying Liu, and Gordan Savin.
\newblock Raising nilpotent orbits in wave-front sets.
\newblock {\em Represent. Theory}, 20:419--450, 2016.

\bibitem[JLX20]{JiangLiuXu2020}
Dihua Jiang, Baiying Liu, and Bin Xu.
\newblock A reciprocal branching problem for automorphic representations and
  global {V}ogan packets.
\newblock {\em J. Reine Angew. Math.}, 765:249--277, 2020.

\bibitem[JLXZ16]{JiangLiuXuZhang2016}
Dihua Jiang, Baiying Liu, Bin Xu, and Lei Zhang.
\newblock The {J}acquet-{L}anglands correspondence via twisted descent.
\newblock {\em Int. Math. Res. Not. IMRN}, (18):5455--5492, 2016.

\bibitem[JLZ13]{JiangLiuZhang2013}
Dihua Jiang, Baiying Liu, and Lei Zhang.
\newblock Poles of certain residual {E}isenstein series of classical groups.
\newblock {\em Pacific J. Math.}, 264(1):83--123, 2013.

\bibitem[JZ14]{JiangZhang2014}
Dihua Jiang and Lei Zhang.
\newblock A product of tensor product {$L$}-functions of quasi-split classical
  groups of {H}ermitian type.
\newblock {\em Geom. Funct. Anal.}, 24(2):552--609, 2014.

\bibitem[JZ17]{JiangZhang2017}
Dihua Jiang and Lei Zhang.
\newblock Automorphic integral transforms for classical groups {II}: {T}wisted
  descents.
\newblock In {\em Representation theory, number theory, and invariant theory},
  volume 323 of {\em Progr. Math.}, pages 303--335. Birkh\"auser/Springer,
  Cham, 2017.

\bibitem[JZ20a]{JiangZhang2020Annals}
Dihua Jiang and Lei Zhang.
\newblock Arthur parameters and cuspidal automorphic modules of classical
  groups.
\newblock {\em Ann. of Math. (2)}, 191(3):739--827, 2020.

\bibitem[JZ20b]{JiangZhang2020JEMS}
Dihua Jiang and Lei Zhang.
\newblock On the non-vanishing of the central value of certain {$L$}-functions:
  unitary groups.
\newblock {\em J. Eur. Math. Soc. (JEMS)}, 22(6):1759--1783, 2020.

\bibitem[JZ21]{JiangZhang2021BesselDescent}
Dihua Jiang and Lei Zhang.
\newblock Bessel descents and branching problems.
\newblock In {\em Relative trace formulas}, Simons Symp., pages 253--290.
  Springer, Cham, [2021] \copyright 2021.

\bibitem[KMSW14]{KalethaMinguezShinWhite}
Tasho Kaletha, Alberto Minguez, Sug~Woo Shin, and Paul-James White.
\newblock Endoscopic classification of representations: Inner forms of unitary
  groups.
\newblock {\em arXiv preprint, available at
  \url{https://arxiv.org/abs/1409.3731}}, 2014.

\bibitem[LM17]{LapidMao2017}
Erez Lapid and Zhengyu Mao.
\newblock Whittaker-{F}ourier coefficients of cusp forms on {$\widetilde{\rm
  Sp}_n$}: reduction to a local statement.
\newblock {\em Amer. J. Math.}, 139(1):1--55, 2017.

\bibitem[LX23]{LiuXu2023}
Baiying Liu and Bin Xu.
\newblock Automorphic descent for symplectic groups: the branching problems and
  {$L$}-functions.
\newblock {\em Amer. J. Math.}, 145(3):807--859, 2023.

\bibitem[Mg11]{Moeglin2011}
C.~M\oe~glin.
\newblock Image des op\'erateurs d'entrelacements normalis\'es et p\^oles des
  s\'eries d'{E}isenstein.
\newblock {\em Adv. Math.}, 228(2):1068--1134, 2011.

\bibitem[MgW95]{MoglinWaldspurger1995}
C.~M\oe~glin and J.-L. Waldspurger.
\newblock {\em Spectral decomposition and {E}isenstein series}, volume 113 of
  {\em Cambridge Tracts in Mathematics}.
\newblock Cambridge University Press, Cambridge, 1995.
\newblock Une paraphrase de l'\'Ecriture [A paraphrase of Scripture].

\bibitem[Mok15]{Mok2015}
Chung~Pang Mok.
\newblock Endoscopic classification of representations of quasi-split unitary
  groups.
\newblock {\em Mem. Amer. Math. Soc.}, 235(1108):vi+248, 2015.

\bibitem[Pol18]{Pollack2018}
Aaron Pollack.
\newblock Unramified {G}odement-{J}acquet theory for the spin similitude group.
\newblock {\em J. Ramanujan Math. Soc.}, 33(3):249--282, 2018.

\bibitem[Pra07]{Prasad2007}
Dipendra Prasad.
\newblock Relating invariant linear form and local epsilon factors via global
  methods.
\newblock {\em Duke Math. J.}, 138(2):233--261, 2007.
\newblock With an appendix by Hiroshi Saito.

\bibitem[PTB11]{PrasadTakloo-Bighash2011}
Dipendra Prasad and Ramin Takloo-Bighash.
\newblock Bessel models for {GS}p(4).
\newblock {\em J. Reine Angew. Math.}, 655:189--243, 2011.

\bibitem[Sai93]{Saito1993}
Hiroshi Saito.
\newblock On {T}unnell's formula for characters of {${\rm GL}(2)$}.
\newblock {\em Compositio Math.}, 85(1):99--108, 1993.

\bibitem[Sha90]{Shahidi1990}
Freydoon Shahidi.
\newblock A proof of {L}anglands' conjecture on {P}lancherel measures;
  complementary series for {$p$}-adic groups.
\newblock {\em Ann. of Math. (2)}, 132(2):273--330, 1990.

\bibitem[Sha11]{Shahidi2011}
Freydoon Shahidi.
\newblock Arthur packets and the {R}amanujan conjecture.
\newblock {\em Kyoto J. Math.}, 51(1):1--23, 2011.

\bibitem[Shi72]{Shimizu1972}
Hideo Shimizu.
\newblock Theta series and automorphic forms on {${\rm GL}\sb{2}$}.
\newblock {\em J. Math. Soc. Japan}, 24:638--683, 1972.

\bibitem[Shi04]{Shimura2004}
Goro Shimura.
\newblock {\em Arithmetic and analytic theories of quadratic forms and
  {C}lifford groups}, volume 109 of {\em Mathematical Surveys and Monographs}.
\newblock American Mathematical Society, Providence, RI, 2004.

\bibitem[Sou05]{Soudry2005}
David Soudry.
\newblock On {L}anglands functoriality from classical groups to {${\rm GL}_n$}.
\newblock {\em Ast\'{e}risque}, (298):335--390, 2005.
\newblock Automorphic forms. I.

\bibitem[Sou06]{Soudry2006ICM}
David Soudry.
\newblock Rankin-{S}elberg integrals, the descent method, and {L}anglands
  functoriality.
\newblock In {\em International {C}ongress of {M}athematicians. {V}ol. {II}},
  pages 1311--1325. Eur. Math. Soc., Z\"urich, 2006.

\bibitem[Tun83]{Tunnell1983}
Jerrold~B. Tunnell.
\newblock Local {$\epsilon $}-factors and characters of {${\rm GL}(2)$}.
\newblock {\em Amer. J. Math.}, 105(6):1277--1307, 1983.

\bibitem[Wal85]{Waldspurger1985}
J.-L. Waldspurger.
\newblock Sur les valeurs de certaines fonctions {$L$} automorphes en leur
  centre de sym\'etrie.
\newblock {\em Compositio Math.}, 54(2):173--242, 1985.

\bibitem[Wal01]{Waldspurger2001Nilpotent}
Jean-Loup Waldspurger.
\newblock Int\'egrales orbitales nilpotentes et endoscopie pour les groupes
  classiques non ramifi\'es.
\newblock {\em Ast\'erisque}, (269):vi+449, 2001.

\bibitem[Yan25a]{Yan2025GGP}
Pan Yan.
\newblock Rankin-{S}elberg integrals for {GS}pin groups with application to the
  global {G}an-{G}ross-{P}rasad conjecture.
\newblock {\em arXiv preprint, available at
  \url{https://arxiv.org/abs/2508.09066}}, 2025.

\bibitem[Yan25b]{Yan2025}
Pan Yan.
\newblock Uniqueness of {B}essel models for {GS}pin groups.
\newblock {\em arXiv preprint, available at
  \url{https://arxiv.org/abs/2503.20116}}, 2025.

\bibitem[YZZ13]{YuanZhangZhang2013}
Xinyi Yuan, Shou-Wu Zhang, and Wei Zhang.
\newblock {\em The {G}ross-{Z}agier formula on {S}himura curves}, volume 184 of
  {\em Annals of Mathematics Studies}.
\newblock Princeton University Press, Princeton, NJ, 2013.

\end{thebibliography}

\end{document}